\documentclass[12pt,twoside,leqno,openany]{amsart}
\usepackage{amssymb,amsbsy,amsmath,amsfonts,amscd,times}
\usepackage{graphics,color,xy,xypic,footmisc,fancyhdr,multicol}
\usepackage{fancybox,graphicx,mathrsfs}
\usepackage[T1]{fontenc}
\usepackage{breqn}
\newtheorem{theorem}{Theorem}
\newtheorem{corollary}{Corollary}
\newtheorem{lemma}{Lemma}
\newtheorem{definition}{Definition}

\let\mathcal\mathscr
\newcommand{\C}{\mathbb{C}}

\newcommand{\R}{\mathbb{R}}

\newcommand{\ov}{\overline}
\newcommand{\LL}{\mathcal{L}_1}
\newcommand{\Lb}{\overline{\mathcal{L}_1}}
\newcommand{\Lk}{\mathcal{L}_1(k)}
\newcommand{\Lbk}{\overline{\mathcal{L}_1}(k)}
\newcommand{\Lkb}{\mathcal{L}_1(\ov{k})}
\newcommand{\Lbkb}{\overline{\mathcal{L}_1}(\overline{k})}
\newcommand{\cc}{{\sf c}}
\newcommand{\cb}{\overline{\sf c}}
\newcommand{\eb}{\overline{\sf e}}
\newcommand{\dd}{{\sf d}}
\newcommand{\ee}{{\sf e}}
\newcommand{\db}{\overline{\sf d}}
\newcommand{\bb}{{\sf b}}
\newcommand{\bbb}{\overline{\sf b}}
\newcommand{\T}{\mathcal{T}}
\newcommand{\Tk}{\mathcal{T}(k)}
\newcommand{\Tkb}{\mathcal{T}(\ov{k})}
\newcommand{\KK}{\mathcal{K}}
\newcommand{\Kb}{\ov{\mathcal{K}}}
\newcommand{\ff}{\sf f}
\newcommand{\fb}{\ov{\sf f}}
\newcommand{\smallbullet}{{\scriptscriptstyle{\bullet}}}

\begin{document}

\setcounter{page}{24}

\title[]{
Explicit absolute parallelism
\\
for $2$-nondegenerate real hypersurfaces
\\
$M^5 \subset \C^3$ of constant Levi rank $1$
}
\author{Samuel Pocchiola}

\address{Samuel Pocchiola ---  D\'epartement de math\'ematiques, b\^atiment 425, Facult\'e des sciences d'Orsay,
Universit\'e Paris-Sud, F-91405 Orsay Cedex, france}
\email{samuel.pocchiola@math.u-psud.fr}

\maketitle

\bigskip

\section*{abstract}
We study the local equivalence problem for five dimensional real
hypersurfaces $M^5$ of $\C^3$ which are $2$-nondegenerate and of
constant Levi rank $1$ under biholomorphisms.  We find two invariants,
$J$ and $W$, which are expressed explicitly in terms of the graphing
function $F$ of $M$, the annulation of which give a necessary and
sufficient condition for $M$ to be locally biholomorphic to a model
hypersurface, the tube over the light cone.  If one of the two
invariants $J$ or $W$ does not vanish on $M$, we show that the
equivalence problem under biholomophisms reduces to an equivalence
problem between $\{e \}$-structures, that is we construct an absolute
parallelism on $M$.

\section{Introduction}
A smooth $5$-dimensional real hypersurface $M \subset \C^3$ is locally
represented as the graph of a smooth function $F$ over the
$5$-dimensional real hyperplane $\C_{ z_1} \times \C_{ z_2} \times
\R_v$:
\[
u
=
F\big(z_1,z_2,\overline{z_1},\overline{z_2},v\big)
.\]
Such a hypersurface $M$ 
is said to be of CR-dimension $2$ if at each point $p$ of $M$, 
the vector space  
\begin{equation*}
T^{1,0}_pM := \left. \C \otimes T_pM  \right. \cap T^{1,0}_p \C
\end{equation*} 
is of complex dimension $2$ (for background, see \cite{MPS, Boggess-1991, BER-1999}). 

We recall that the Levi form $LF$ of $M$ at a point $p$ is the
skew-symmetric hermitian form defined on $T^{1,0}_pM$ by
\begin{equation*}
LF(X,Y) = i \, \big[ \widetilde{X}, \ov{\widetilde{Y}} \big]_p \mod T^{1,0}_pM \oplus
T^{0,1}_pM
,\end{equation*}
where $\widetilde{X}$ and $\widetilde{Y}$ are two local sections $M
\longrightarrow T^{1,0}M$ such that $\widetilde{X}_p=X$ and $\widetilde{Y}_p =
Y$. 

The aim of this paper is to study the equivalence problem under
biholomorphisms of the hypersurfaces $M \subset \C^3$ which are of
CR-dimension $2$, and whose Levi form is degenerate and of constant
rank $1$. For well-known natural reasons, we will also assume that the
hypersurfaces we consider are $2$-nondegenerate, i.e. that their Freeman forms
are non-zero (see for example \cite{MPS}, p.~91). Two other approaches on this problem 
have been recently provided
by Isaev-Zaitsev and Medori-Spiro (\cite{Medori-Spiro, Isaev-Zaitsev}). 
We refer to \cite{EMS} for an historical perspective on equivalence problems for hypersurfaces of complex spaces. 

We start by exhibiting two vector fields $\LL$ and $\mathcal{L}_2$
which constitute a basis of $T^{1,0}_pM$ at each point $p$ of
$M$. This provides an identification of $T^{1,0}_pM$ with $\C^2$ at
each point. We also exhibit a
real $1$-form $\sigma$ on $TM$ whose prolongation to $\C \otimes TM$ satisfies:
\begin{equation*}
\{ \sigma =0 \} = T^{1,0}M \oplus T^{0,1}M,
\end{equation*}
which provides an identication of the projection
\begin{equation*}
\C \otimes T_pM \longrightarrow \left. \C \otimes T_pM \right.\big/ \left( T^{1,0}_pM \oplus T^{0,1}_pM \right)
\end{equation*}
with the map $\sigma_p$: \, $\C \otimes T_pM \longrightarrow \C$.
With these two identifications, the Levi form $LF$ can be viewed at
each point $p$ as a skew hermitian form on $\C^2$ represented by the matrix:
\begin{equation*}
LF= 
\begin{pmatrix}
\sigma_p \left( i \, \big[ \LL , \Lb \big] \right)& \sigma_p \left( i \,\big[ \mathcal{L}_2, {\Lb} \big] \right)\\
\sigma_p \left( i \,\big[ \LL , \ov{\mathcal{L}_2} \big] \right) & \sigma_p \left( i \,\big[ \mathcal{L}_2, \ov{\mathcal{L}_2} \big] \right) \\
\end{pmatrix}
.\end{equation*}

The fact that $LF$ is supposed to be of constant rank $1$ ensures the
existence of a certain function $k$ such that the vector field 
\begin{equation*}
\KK := k \,
\LL + \mathcal{L}_2
\end{equation*}
lies in the kernel of $LF$. Our explicit
computation of $LF$ provides us with an explicit expression of $k$ in
terms of the graphing function $F$ for $M$. In fact, here are the expressions 
of $\LL$ and $\KK$:
\begin{align*}
\LL &= \partial_{z_1}  - i \, \frac{F_{z_1}}{1 + i \, F_v} \, \partial_v, \\
\KK &= k \, \partial_{z_1} +  \partial_{z_2}  - \frac{i}{1 + i \, F_v}  \left( k \, F_{z_1} + F_{z_2} \right) \partial_v,
\end{align*}
and also of $k$: 
{\tiny
\begin{equation*}
k = - \frac{F_{z_2, \ov{z_1}} + F_{z_2, \ov{z_1}} \,  F_{v}^2 - i \, F_{\ov{z_1}} \,  F_{z_2, v} - F_{\ov{z_1}} \,  F_{v} \, F_{v, z_2} + i \, F_{z_2} \,  F_{\ov{z_1}} \, 
F_{v,v} - F_{z_2} \,  F_v \,  F_{v, \ov{z_1}}}{
F_{z_1, \ov{z_1}} + F_{z_1, \ov{z_1}} \,  F_v^2 - i \, F_{\ov{z_1}} \,  F_{z_1, v} - F_{\ov{z_1}} \,  F_v \,  F_{ z_1, v} + i \, F_{z_1} \, F_{\ov{z_1}, v} 
+ F_{z_1} \, F_{\ov{z_1}} \,  F_{v, v} -F_{z_1} \, F_v  \, F_{v,\ov{z_1}}},
\end{equation*}}
\!\!and we want to emphasize that all our 
subsequent computations will be expressed in terms of Lie derivatives
of the function $k$ by the vector fields $\LL$, $\KK$, $\Lb$, $\Kb$, hence
in terms of $F$.

From our construction, the four vector fields $\LL$, $\KK$, $\Lb$,
$\Kb$ constitute a basis of $T^{1,0}_pM \oplus T^{0,1}_pM$ at each
point $p$ of $M$.  It turns out that the vector field $\T$ defined by:
\begin{equation*}
\T:= i \, \big[ \LL ,\Lb \big]
\end{equation*}
 is linearly independant from $\LL$,
$\KK$, $\Lb$, $\Kb$. With the five vector fields $\LL$, $\KK$, $\Lb$,
$\Kb$ and $\T$, we have thus exhibited a local section from $M$ into
$\C \otimes F(M)$, the complexification of the bundle $F(M)$ of frames of
$M$, which is geometrically adapted to $M$ in the following sense:
\begin{enumerate}
\item{the line bundle generated by $\KK$ is
 the kernel of the Levi form of $M$,}
\item{ $\LL$ and $\KK$  constitute a basis of $T^{1,0}M$,}
\item{$\T$ is defined by the formula  $\T:= i \, \big[ \LL ,\Lb \big]$.}
\end{enumerate}

Then we define the coframe of  $1$-forms:
\begin{equation*}
\left(\rho_{0}, \kappa_{0}, \zeta_{0}, \ov{\kappa}_{0}, \ov{\zeta}_0 \right)
\end{equation*}
which is the dual coframe of the frame: 
\begin{equation*}
\left(\LL, \KK, \Lb, \Kb, \T \right). 
\end{equation*}

The computation of the exterior derivatives of 
$\rho_{0}$, $\kappa_{0}$, $\zeta_{0}$, $\ov{\kappa}_{0}$, $\ov{\zeta}_0$, 
which constitute the so-called structure equations of the coframe, involves another important function on $M$, that we denote by $P$ in the sequel.
We give here the expression of $P$ in terms of the graphing function $F$ because, as with the function $k$, all our subsequent computations will involve
terms expressed as derivatives of $P$ by the fundamental vector fields $\LL$, $\KK$, $\Lb$, $\Kb$, $\T$, namely:
\begin{equation*}
P = \frac{ l_{z_1} + A^1 \, l_v - l \, A^1_v}{l},
\end{equation*}
where:
\begin{equation*}
A_1 = 2 \, \frac{F_{z_1}}{1 + i \, F_v}
,\end{equation*}
and where:
\begin{equation*}
l: = i \left(\ov{A^1_{z_1}} - A^1_{\ov{z_1}} + A^1 \ov{A^1_{v}} - \ov{A^1} A^1_{v} \right).
\end{equation*}

Then in terms of $P$ and $k$, the structure equations enjoyed by $\rho_{0}$, $\kappa_{0}$, $\zeta_{0}$, $\ov{\kappa}_{0}$, $\ov{\zeta}_0$, are the following:
\begin{equation*}
\begin{aligned}
d \rho_{0} & = P \left. \rho_0 \wedge \kappa_0 \right. - \Lk \left. \rho_0 \wedge \zeta_0 \right. + \ov{P} \left. \rho_0 \wedge \ov{\kappa_0} \right.
- \Lbkb  \left. \rho_0 \wedge \ov{\zeta_0} \right. + i \left.  \kappa_0 \wedge \ov{\kappa_0} \right., \\
d \kappa_{0} & = - \mathcal{T}(k) \, \left.  \rho_0 \wedge \zeta_0 \right. - \Lk \, \left. \kappa_0 \wedge \zeta_0 \right. 
+ \Lbk \, \left. \zeta_0 \wedge \ov{\kappa}_0 \right. 
, \\
d \zeta_0 & = 0, \\
d \ov{\kappa}_0 & = - \Tkb \, \left.  \rho_0 \wedge \ov{\zeta_0} \right.  
- \Lkb \, \left. \kappa_0 \wedge  \ov{\zeta}_0 \right. - \Lbkb \, \left. \ov{\kappa_0} \wedge \ov{\zeta}_0 \right. 
,\\
d \ov{ \zeta_{0}} & = 0.
\end{aligned}
\end{equation*}
The fact that $M$ is $2$-nondegenerate is expressed by the 
(biholomorphically invariant, see \cite{MPS}) assumption that: 
\begin{equation*}
\Lbk \, \, \, \text{\em vanishes nowhere on $M$};
\end{equation*}
notice here that $\Lbk$ appears as the coefficient of $\zeta_0 \wedge 
\ov{\kappa_0}$ in $d \kappa_0$.

The end of section \ref{section:setup} is devoted to reinterpret the equivalence problem under biholomophisms of such hypersurfaces as an equivalence
problem between $G$-structures. We recall that if
$G \subset GL(n,\R)$ is a Lie group, a $G$-structure on a manifold $M$ of dimension $n$ is a subbundle of 
the bundle $F(M)$ of frames of $M$, which is a principal $G$-bundle. The fact that we can express the equivalence problem in 
terms of equivalences between $G$-structures comes from the following observation: 
if $\phi$ is a local biholomorphism of $\C^3$  such that $\phi(M) = M$,
then the restriction $\phi_{M}$ of $\phi$ to  $M$ is a local smooth diffeomorphism of $M$ which satisfies
the additional two conditions:
\begin{enumerate}
\item{$\phi_{M}$ stabilizes the bundle $T^{1,0}(M)$;}
\item{$\phi_{M}$ stabilizes the kernel of the Levi form of $M$.}
\end{enumerate}
As a result, there are three functions $\ff$, $\cc$ and $\ee$ on $M$ such that :
\begin{equation*}
\phi_{M*}(\KK) = \ff \, \KK,
\end{equation*}
and
\begin{equation*}
\phi_{M*}(\LL) = \cc \, \LL + \ee \, \KK.
\end{equation*}
Of course, as $\phi_{M}$ is a real diffeomorphism, we shall also have :
\begin{equation*}
\phi_{M*}(\Kb) = \ov{\phi_{M*}(\KK)} =  \fb \, \Kb,
\end{equation*}
and
\begin{equation*}
\phi_{M*}(\Lb) = \ov{\phi_{M*}(\LL)} = \cb \, \Lb + \eb \, \Kb.
\end{equation*}
It is then easy to show that the matrix Lie group which encodes suitably the problem is the $10$ dimensional Lie group $G_1$ given by the matrices of the form:
\begin{equation*}
g := \begin{pmatrix}
{\sf c}\overline{\sf c} & 0 & 0 & 0 & 0 \\
{\sf b} & {\sf c} & 0 & 0 & 0 \\
{\sf d} & {\sf e} & {\sf f} & 0 & 0 \\
\overline{\sf b} & 0 & 0 & \overline{\sf c} & 0 \\
\overline{\sf d} & 0 & 0 & \overline{\sf e} & \overline{\sf f}
\end{pmatrix},
\end{equation*}
where $\cc$ and $\ff$ are non-zero complex numbers, while $\bb$, $\dd$, $\ee$ are arbitrary complex numbers.

The rest of our article is devoted to the implementation of Cartan's
equivalence method to reduce this $G_1$-equivalence problem to an
absolute parallelism.  We use \cite{Olver-1995} and \cite{Sternberg}
as standard references on Cartan's equivalence method.  We develope
the parametric version of Cartan's equivalence method, that is we
perform all the computations and give explicit expressions of the
functions involved in the normalizations of the group parameters,
because we need to control carefully the expressions of these
functions: some of them might indeed vanish identically on $M$, which
is of crucial importance when deciding whether a potential
normalization might be allowed or not. Our computations involves only
terms which are derivatives of the functions $k$ and $P$ by the
fundamental vector fields $\LL$, $\KK$, $\Lb$, $\Kb$, $\T$, and they
become ramified by the fact that some relations exists between these
derivatives: those that follow simply from the Jacobi identities, and
those that follow from the fact that the Levi form of $M$ is of rank
$1$ everywhere. We give a sum up of the relations that we use at the
end of subsection \ref{subsection:lie}.  These relations imply
important simplifications in the formulae we obtain for the torsion
coefficients, and shall not be missed if one keeps in mind that we
usually want to control whether these coefficients do vanish or not on
$M$, which is a delicate task, even with the help of a computer
algebra system.

We find in section \ref{section:step1} that the first normalization of the group parameters is: 
\begin{equation*}
{\sf f} = \frac{\cc}{\cb} \, \overline{ \mathcal{L}_1}(k).
\end{equation*}
This enables us
to reduce $G_1$ to a new matrix Lie group $G_2$, which is $8$-dimensional and whose elements $g$ take the form:
\begin{equation*}
g = \begin{pmatrix}
{\sf c} \overline{\sf c} & 0 &0 &0 &0 \\
{\sf b} & {\sf c} & 0 & 0 & 0 \\
{\sf d} & {\sf e} & \frac{\sf c}{\overline{\sf c}} & 0 & 0 \\
{\sf \overline{b}} & 0 & 0 & {\sf \overline{c}} & 0 \\
0 & 0 & {\sf \overline{d}} & {\sf \overline{e}} & \frac{\sf c}{\sf \overline{c}}
\end{pmatrix}.
\end{equation*}
We then 
perform a second loop in Cartan's equivalence method in section \ref{section:step2}, which yields the normalization:
\begin{equation*}
\bb = - i \, \cb \ee + i \, \frac{\cc}{3} \left( \frac{\ov{\mathcal{L}_1}\left( \Lbk \right)}{\Lbk} - \ov{P} \right)
,\end{equation*} 
and which therefore leads to a $G_3$-equivalence problem,
where $G_3$ is the $6$-dimensional matrix Lie group whose elements are of the form: 
\begin{equation*}
g = \begin{pmatrix}
\cc \cb & 0 & 0 & 0 & 0 \\
- i\, \ee \cb & \cc & 0 & 0 &0 \\
\dd & \ee & \frac{\cc}{\cb} & 0 & 0 \\
i\, \eb \cc & 0 & 0 & \cb & 0 \\
\db & 0 & 0 & \eb & \frac{\cb}{\cc} 
\end{pmatrix}
.
\end{equation*}
The third loop is done in section \ref{section:step3} and it gives us a normalization of the parameter $\dd$ as:
\begin{multline*}
\dd = - i\, \frac{1}{2} \,  \frac{\ee^2 \cb}{ \cc} + i \, \frac{2}{9} \, \frac{\cc}{\cb} \, \frac{\Lb \left( \Lbk \right)^2}{\Lbk^2} + i \,
\frac{1}{18} \,   \frac{\cc}{\cb} \, \frac{\Lb \left(\Lbk \right) \ov{P}}{\Lbk} \\ - i \, \frac{1}{9} \, \frac{\cc}{\cb} \, \ov{P}^2  +  i \, \frac{1}{6} \, 
\frac{\cc}{\cb} \, \Lb \left( \ov{P} \right) - i \,  \frac{1}{6} \, \frac{\cc}{\cb} \, \frac{\Lb \left( \Lb \left( \Lbk \right) \right)}{\Lbk}.
\end{multline*}
This therefore reduces $G_3$ to the $4$-dimensional group $G_4$, whose elements are of the form:
\begin{equation*}
g= \begin{pmatrix}
\cc \cb & 0 & 0 & 0 & 0 \\
-i \, \ee \cb & \cc & 0 & 0 & 0 \\
- \frac{i}{2} \, \frac{\ee^2 \cb}{\cc}  & \ee & \frac{\cc}{\cb} & 0 & 0 \\
i \,  \eb \cc & 0 & 0 &  \cb & 0  \\
 \frac{i}{2} \, \frac{\eb^2 \cc}{\cb}  &0 & 0 &  \eb & \frac{\cb}{\cc}  \\
\end{pmatrix}
.\end{equation*}

The fourth loop of Cartan's method, which is done in section
\ref{section:step4},
leads to a more advanced analysis than the three previous ones. The
normalizations of the group parameters that are suggested at this 
stage depend on the
vanishing or the non-vanishing of two functions, $J$ and $W$, which 
appear to be two fundamental invariants of the problem. The
expressions of $J$ and $W$ are given below:
\begin{dmath*}
J = \frac{5}{18} \, \frac{\LL \left( \Lkb \right)^2}{\Lkb^2} \, P + \frac{1}{3} \, P \, \LL \left( P \right) -
\frac{1}{9} \,  \frac{\LL \left( \Lkb \right)}{\Lkb} \, P^2 \\ +  \frac{20}{27} \, \frac{ \LL \left( \Lkb \right)^3}{\Lkb^3} 
 - \frac{5}{6} 
\, \frac{ \LL \left( \Lkb \right) \, \LL \left( \LL \left( \Lkb \right) \right)}{\Lkb^2} \\  + \frac{1}{6} \, \frac{\LL \left( \Lkb \right) \, \LL(P)}{\Lkb}
- \frac{1}{6} \, \frac{\LL \left( \LL \left( \Lkb \right) \right)}{\Lkb} \, P  \\
- \frac{2}{27} \, P^3 - 
\frac{1}{6} \, \LL \left( \LL \left( P \right) \right)  + \frac{1}{6} \,  \frac{\LL \left( \LL \left( \LL \left( \Lkb \right) \right) \right)}{\Lkb} 
,\end{dmath*}
and
\begin{multline*}
W:=  \frac{2}{3} \, 
 \frac{\LL \left( \Lbk \right)}{\Lbk} + \frac{2}{3} \, \frac{\LL 
\left( \Lkb \right)}{\Lkb} \\
+ \frac{1}{3} \frac{\Lb \left( \Lbk \right) \KK \left( \Lbk \right)}{\Lbk^3}  -
\frac{1}{3} \frac{ \KK \left( \Lb \left( \Lbk \right) \right)}{\Lbk^2} + \frac{i}{3} \, \frac{\Tk}{\Lbk}
.
\end{multline*}
We thus observe a branching phenomenon at that point: if $J$ and $W$ are both identically vanishing on $M$, 
then no further reductions of the group parameters are allowed 
by Cartan's method. However, if $J$ is non-vanishing we can normalize the parameter $\cc$ by
\begin{equation*}
\cc = J^{\frac{1}{3}},
\end{equation*} 
whereas if $W$ is non vanishing we can perform the normalization 
\begin{equation*}
\cc = W.
\end{equation*}

We notice here that we are not treating the cases where one of the two
invariants $J$ or $W$ might vanish somewhere on $M$ without beeing
identically vanishing on $M$, that is we are making a genericity
assumption $M$, which is a standard process when using Cartan's
technique. To be fully precise, we also suppose in section
\ref{section:W} that the function $\Kb(W)$ is generic on $M$, that is
it is either identically $0$ or non-vanishing on $M$, in order to
establish the results of this section.  This motivates the following
definition:
\begin{definition} \label{definition:generic}
A five dimensional CR-submanifold of $\C^3$ of CR-dimension $2$ which is $2$-non degenerate, and whose Levi form is of constant rank $1$ is said to be generic
if the functions $J$, $W$ and  $\Kb(W)$ are either $0$ or non-vanishing on $M$.
\end{definition}

Section \ref{section:J} is devoted to show that in the case $J \neq
0$, one can normalize the last group parameter $\ee$, thus reducing
the equivalence problem to the study of a five dimensional $\{e
\}$-structure. Section \ref{section:W} deals with the same issue in
the case $W \neq 0$.  To this end, we show that $W \neq 0$ implies
$\Kb(W) \neq 0$ under the genericity assumption (this is the purpose
of Lemma \ref{lemma:KbW}).  In both cases $J \neq 0$ and $W \neq 0$,
the final $\{e \}$-structure that we obtain on $M$ contains terms
which are derivatives of the graphing function $F$ up to order $8$.
Thus the results of these sections only require that $M$ is 
$\mathcal{C}^8$-smooth.

Finally, in section \ref{section:prolongation}, we show that when both
$J$ and $W$ vanish identically on $M$, we can reduce the equivalence
problem to a $10$-dimensional $\{e \}$-structure after performing two
suitable prolongations. The structure equations that we obtain are the
same as those enjoyed by the tube over the future the light cone:
\begin{equation*}
\left( {\sf Re} \,z_1 \right)^2 -\left( {\sf Re} \,z_2 \right)^2 
- \left( {\sf Re} \,z_3 \right)^2
=
0, \qquad \qquad {\sf Re} \,z_1 > 0,
\end{equation*}
which is locally biholomorphic (see \cite{Fels-Kaup-2007, Merker-2003}) to
the graphed hypersurface:
\begin{equation*}
u = \frac{z_1 \ov{z_1} + \frac{1}{2} z_1^2 \ov{z_2} + \frac{1}{2} \ov{z_1^2} z_2}{1 - z_2 \ov{z_2}}
.\end{equation*}
This proves the fact that when $J$ and $W$ are both vanishing, $M$ is locally biholomorphic to the tube over the light cone.
We summarize these results in the following theorem:

\begin{theorem} \label{thm:pocchiola}
Let $M \subset \C^3$
be a $\mathcal{C}^8$-smooth $5$-dimensional 
hypersurface of $CR$-dimension $2$, which is $2$-non degenerate, whose Levi form is of constant rank $1$ and which is generic
in the sense of definition \ref{definition:generic}. Then 
\begin{enumerate}
\smallskip\item{if $W \neq 0$ or if $J \neq 0$ on $M$,
 then the local equivalence problem for $M$ reduces to the equivalence problem for a five dimensional $\{e \}$-structure.}
\smallskip\item{if $W = 0$ and $J=0$ identically on $M$, then $M$ is locally biholomorphic to the tube over the light cone.}
\end{enumerate}
\end{theorem}

Granted that the functions $k$ and $P$ are expressed in terms of partial derivatives of order $\leq 3$ of the graphing function $F$,
and that the two main invariants $J$ and $W$ are explicit in terms of $k$ and $P$,
we stress that the local biholomorphic equivalence to the light cone is explicitely characterised in terms of $F$.

It is well-known (see, for example, \cite{Kobayashi}) that the group
of automorphisms $\mathcal{U}$ of an $\{ e \}$-structure on a $\mathcal{C}^{\infty}$ manifold
$N$ is a Lie transformation group such that
$ \text{dim} \, \mathcal{U} \leq \text{dim} \, N.$
As a result of theorem \ref{thm:pocchiola}, we thus have:

\begin{corollary}
Let $M \subset \C^3$ be a $\mathcal{C}^{\infty}$ $CR$-manifold satisfying the hypotheses of theorem \ref{thm:pocchiola}.
If $M$ is not locally equivalent to the tube over the light cone at a point $p \in M$, 
then the dimension of the Lie algebra of germs of $CR$-automorphisms of $M$ at $p$ is bounded by $5$.
\end{corollary}

We now give a slight extension of theorem \ref{thm:pocchiola}.
If $M$ is a $5$-dimensional abstract $CR$-manifold of $CR$ dimension $2$ then there exist a subbundle $L$ of $\C \otimes TM$ of dimension $2$ such that
\begin{enumerate}
\item{$L \cap \ov{L}$ = \{0\}}
\item{$L$ is formally integrable.}
\end{enumerate}
It is then well-known that there exist local coordinates $(x_1,x_2,x_3,x_4,v)$ on $M$ and two local sections $\LL$ and $\mathcal{L}_2$ of $L$, such that:
\begin{equation*}
\LL = \frac{\partial}{\partial z_1} + A^1 \, \frac{\partial}{\partial v},
\end{equation*} and
\begin{equation*}
\mathcal{L}_2 = \frac{\partial}{\partial z_2} + A^2 \, \frac{\partial}{\partial v},
\end{equation*}
where $A^1$ and $A^2$ are two locally defined functions on $M$, and where the vector fields 
$\frac{\partial}{\partial z_1}$ and $\frac{\partial}{\partial z_2}$ are defined by the usual formulae:
\begin{equation*}
\frac{\partial}{\partial z_1} = \frac{1}{2} \left( \frac{\partial}{\partial x_1} - i \, \frac{\partial}{\partial x_2} \right)
,\end{equation*}
and
\begin{equation*}
\frac{\partial}{\partial z_2} = \frac{1}{2} \left( \frac{\partial}{\partial x_3} - i \, \frac{\partial}{\partial x_4} \right).
\end{equation*}
As a result, we can define the functions $k$ and $P$ together with the four vector fields $\KK$, $\Lb$, $\Kb$ and $\T$
in terms of the fundamental functions $A^1$ and $A^2$  as in the embedded case, and
all the subsequent structure equations at each step of Cartan's method
are unchanged. Theorem \ref{thm:pocchiola} remains thus valid 
in the more general setting of abstract $CR$-manifolds.

Finally, the $G$-structures that we introduce at each step are in fact globally defined on $M$ (as subbundles of $\C \otimes TM$).  
As a result, the first part of theorem \ref{thm:pocchiola} has the following global counterpart:

\begin{theorem}
Let $M$ be an abstract $CR$-manifold satisfying the hypotheses of theorem  \ref{thm:pocchiola}. Then $J$ and $W$ are globally defined on
$M$. If $J$ does not vanish on $M$ or if $W$ does not vanish on $M$, then there exist an absolute parallelism on $M$.
\end{theorem}

\medskip
{\em Acknowledgments.}
I wish to thank Professor Alexander Isaev for insightful suggestions that provided
improvements, e.g. the abstract and global counterparts of theorem \ref{thm:pocchiola}.

\section{Geometric and analytic set up} \label{section:setup}

\subsection{Shape of the initial coframe}
Let $M \subset \C^3$ be a local real analytic hypersurface passing through the origin of $\C^3$.
We recall that $M$ can be represented 
as a graph over the $5$-dimensional real hyperplane
$\C_{ z_1} \times \C_{ z_2} \times \R_v$:
\[
u
=
F\big(z_1,z_2,\overline{z_1},\overline{z_2},v\big),
\]
where $F$ is a local real analytic function depending
on $5$ arguments.
We make the assumption that $M$ is a $CR$-submanifold of $CR$ dimension $2$, that is the bundle $T^{1,0}M$ is of complex dimension $2$.
Let us look for a frame of 
$T^{1,0}M$ constituted of two vectors field of the form:

\begin{align*}
\mathcal{L}_1
&
=
\frac{\partial}{\partial z_1}
+
A_1\,
\frac{\partial}{\partial w},
\\
\mathcal{L}_2
&
=
\frac{\partial}{\partial z_2}
+
A_2\,
\frac{\partial}{\partial w},
\end{align*}
with two unknown functions $A_1$ and $A_2$.
As $M$ is the zero set of the function $G:= u - F$, the condition that $\LL$ and $\mathcal{L}_2$ belong to $T^{1,0}M$ take the form:
\begin{equation*}
d \, G\left( \LL \right)= 0 \qquad \text{and} \qquad 
d \, G \left( \mathcal{L}_2 \right) = 0.
\end{equation*}
As we have:
\begin{equation*}
dG = du - F_{z_1} \, dz_1 - F_{z_2} \, dz_2 - F_{\ov{z_1}} \, d\ov{z_1} - F_{\ov{z_2}} \, d \ov{z_{2}} - F_v \, dv
\end{equation*}
and
\begin{equation*}
\partial_{w} = \frac{1}{2} \left( \partial_u - i \, \partial_v \right),
\end{equation*}
these two conditions read as: 
\begin{align*}
F_{z_j} - \frac{1}{2} \, A_j - \frac{i}{2} \, F_v \, A_j  =0,  && j=1,2
,\end{align*} 
which lead to:
\begin{align*}
A_j &= 2 \, \frac{F_{z_j}}{1 + i \, F_v} && j=1,2.
\end{align*}

If $\pi$ denotes the canonical projection $\C^3 \longrightarrow \C^2 \times \R$ which sends the variables $(z_1,z_2,w)$ on $(z_1,z_2,v)$, the fact that $M$ 
is a graph over the hyperplane $\C_{ z_1} \times \C_{ z_2} \times \R_v$ makes the restriction of $\pi$ to $M$ a local diffeomorphism 
$M \longrightarrow \C^2 \times \R$, that is a local chart on $M$. All the subsequent computations will be made in coordinates $(z_1,z_2,v)$, 
which means that they will be made through this local chart provided by $\pi$. 
The (extrinsic) vector fields $\mathcal{L}_j$ are mapped by $\pi$ onto the (intrinsic) vector fields 
$\pi_{*} \left( \mathcal{L}_j \right)$. As $\pi_{*}\left( \partial_w \right) = - \frac{i}{2} \, \partial_v$, 
we have:
\begin{align*}
\pi_{*} \left( \mathcal{L}_j \right)& = \partial_{z_j} + A^j  \, \partial_{v} && j=1,2 
\end{align*}
where
\begin{align*}
A^j &:= - i \, \frac{F_{z_j}}{1 + i \, F_v} && j =1,2.
\end{align*}
In order to simplify the notations, we will still denote $p_{*} \left( \mathcal{L}_j \right)$ by $\mathcal{L}_j$ in the sequel.
If $\sigma$ is a $1$-form on $M$ whose kernel at each point $p$ is $\left. T^{1,0}_pM \oplus T^{0,1}_pM \right. $, we identify the projection 
\begin{equation*}
\C \otimes T_pM \longrightarrow \left. \C \otimes T_pM \right.\big/ \left. T^{1,0}_pM \oplus T^{0,1}_pM \right. 
\end{equation*}
with the map $\sigma_p$: \, $\C \otimes T_pM \longrightarrow \C$. An example of such a $1$-form $\sigma$ is given by:
\begin{equation*}
\sigma := dv - A^1 \, dz_1 - A^2 \, dz_2 - \ov{A^1} \, d\ov{z_1} - \ov{A^2} \, d \ov{z_2}.
\end{equation*}
As an identification of $T^{1,0}_pM$ with $\C^2$ is also provided by the basis of vector fields $\LL$ and $\mathcal{L}_2$ ,
the Levi form of $M$ can be viewed as the skew-symmetric hermitian form on $\C^2$ given by the matrix:
\begin{equation*}
LF:= 
\begin{pmatrix}
\sigma_p \left(i \, \big[ \LL , \Lb \big] \right)& \sigma_p \left( i \, \big[ \mathcal{L}_2, {\Lb} \big] \right)\\
\sigma_p \left( i \, \big[ \LL , \ov{\mathcal{L}_2} \big] \right) & \sigma_p \left( i \,  \big[ \mathcal{L}_2, \ov{\mathcal{L}_2} \big] \right) \\
\end{pmatrix}
.\end{equation*}
The computation of the Lie bracket $\big[ \mathcal{L}_1 , \ov{\mathcal{L}_1} \big]$ gives:
\begin{align*}
\label{eq:=eql}
\big[ \mathcal{L}_1 , \ov{\mathcal{L}_1} \big]  & = \big[ \partial_{z_1} + A^1 \, \partial_{v} , \partial_{\ov{z_1}} + \ov{A^1} \, \partial_{v} \big] \\
                                               & =  \left( \ov{A^1}_{z_1} - A^1_{\ov{z_1}} + A^1 \ov{A^1_{v}} - \ov{A^1} A^1_{v} \right) \, \partial_v.
\end{align*}
Similar computations of $\big[ \mathcal{L}_1 , \ov{\mathcal{L}_2} \big]$, $\big[ \mathcal{L}_2 , \ov{\mathcal{L}_1} \big]$ and
 $\big[ \mathcal{L}_2 , \ov{\mathcal{L}_2} \big]$ give that 
\begin{equation*}
\big[ \mathcal{L}_1 , \ov{\mathcal{L}_2} \big]=  \big[ \mathcal{L}_2 , \ov{\mathcal{L}_1} \big] = \big[ \mathcal{L}_2 , \ov{\mathcal{L}_2} \big] = 0 \qquad 
\mod{\partial_v}.
\end{equation*}
In the sequel, we make the assumption that $M$ is Levi degenerate of rank $1$. There is therefore a function $k$ defined on $M$ such that 
{\tiny $ \begin{pmatrix} k \\ 1 \end{pmatrix}$} gives a basis of the kernel of $LF$.
As a result of the definition of $k$ and $LF$, the four Lie brackets $\big[ \mathcal{L}_1 , \ov{\mathcal{L}_1} \big]$, $
\big[ \mathcal{L}_1 , \ov{\mathcal{L}_2} \big]$, $\big[ \mathcal{L}_2 , \ov{\mathcal{L}_1} \big]$ and
 $\big[ \mathcal{L}_2 , \ov{\mathcal{L}_2} \big]$
enjoy the following two relations:
\begin{equation} \label{eq:elim}
\left\{
\begin{aligned}
k \,  \big[ \mathcal{L}_1 , \ov{\mathcal{L}_1} \big] +  \big[ \mathcal{L}_2 , \ov{\mathcal{L}_1} \big] & =0 \\
k \,  \big[ \mathcal{L}_1 , \ov{\mathcal{L}_2} \big] +  \big[ \mathcal{L}_2 , \ov{\mathcal{L}_2} \big] & =0.
\end{aligned}
\right.
\end{equation}
Moreover, the vector field 
$\mathcal{K}:= \mathcal{L}_2 + k \, \mathcal{L}_1$ gives a basis of the kernel of the Levi form on $M$ and the four vectors 
fields $\LL$, $\mathcal{K}$, $\Lb$ and $\ov{\mathcal{K}}$ give a basis of $T^{1,0}M \oplus T^{0,1}M$.
Let us introduce the fifth vector field
\begin{equation*}
\mathcal{T} := i \, \big[ \mathcal{L}_1 , \ov{\mathcal{L}_1} \big]. 
\end{equation*}
As $\mathcal{T}$ lies in the line bundle generated by $\partial_v$, the five vector fields $\mathcal{T}$, $\LL$, $\Lb$, $\mathcal{K}$ and $\ov{\mathcal{K}}$
give a basis of $\C \otimes_{\R} TM$.
\subsection{Lie bracket structure} \label{subsection:lie}
Let us explore the Lie bracket relations satisfied by this basis of $\C \otimes_{\R} TM$.
We start with the computation of $ \big[ \LL , {\mathcal{L}_2} \big]$.

\begin{align*}
\big[ \LL , {\mathcal{L}_2} \big] & = \big[ \partial_{z_1} + A^1 \partial_v , \partial_{z_2} + A^2 \partial_v \big] \\
 & \equiv 0 \qquad \text{mod $\partial_v$}, 
\end{align*}
which means that $ \big[ \LL , {\mathcal{L}_2} \big]$ belongs to the line bundle generated by $\partial_v$. On the other hand, as $\LL$ and $\mathcal{L}_2$ 
both belong to $T^{1,0}M$, 
and as it is a well known fact that $T^{1,0}M$ is involutive, $ \big[ \LL , {\mathcal{L}_2} \big]$  belongs to  $T^{1,0}M$, whose intersection with 
$\C \cdot \partial_v$ is reduced to zero. We thus have:
\begin{equation*}
\big[ \LL , {\mathcal{L}_2} \big] = 0.
\end{equation*}
As a result, we can compute $ \big[ \mathcal{K} , \mathcal{L}_1 \big]$. Indeed we have:

\begin{equation*}
\big[ \mathcal{K} , \mathcal{L}_1 \big]  = \big[ k \, \LL + \mathcal{L}_2, \LL \big] = - \LL(k) \, \LL.
\end{equation*}

We now turn our attention on the computation of the bracket $ \big[ \mathcal{K} , \ov{\mathcal{L}_1} \big]$. Using the relation $(\ref{eq:elim})$, we get:
\begin{align*}
\big[ \mathcal{K} , \ov{\mathcal{L}_1} \big] & =  \big[ k \, \LL + \mathcal{L}_2, \Lb \big] \\
                                             & = k \, \big[ \LL, \Lb \big] + \big[ \mathcal{L}_2, \ov{\mathcal{L}_1} \big] - \Lbk \, \LL \\
                                             & = - \Lbk \, \LL.
\end{align*}

To compute further brackets, we need to determine the value of $\mathcal{K}(\ov{k})$.

Taking the Lie bracket between $\mathcal{K}$ and the complex conjugate of the first equation of $(\ref{eq:elim})$ gives:
\begin{equation*}
\mathcal{K}(\ov{k}) \, \big[ \LL , \Lb \big] + \ov{k} \, \big[ \mathcal{K} , \big[ \LL , \Lb \big] \big] + \big[ \mathcal{K} , \big[ \LL , 
\ov{\mathcal{L}_2} \big] \big] = 0.
\end{equation*}
As $\mathcal{K}(\ov{k}) \, \big[ \LL , \Lb \big]$ belongs to $\C \cdot \partial_v$, the vector field 
\begin{equation*}
S :=  \ov{k} \, \big[ \mathcal{K} , \big[ \LL , \Lb \big] \big] + \big[ \mathcal{K} , \big[ \LL , 
\ov{\mathcal{L}_2} \big] \big]
\end{equation*}
is equal to its projection on $\C \cdot \partial_v$. It is thus sufficient to perform its computation  mod $T^{1,0}M$.
The Jacobi identity gives:
\begin{equation*}
S  = \ov{k} \, \big[ \big[ \mathcal{K}, \LL \big] , \Lb \big] + \ov{k} \, \big[ \LL, \big[ \mathcal{K}, \Lb \big] \big] + \big[ \big[ \mathcal{K}, \LL \big], 
\ov{\mathcal{L}_2} \big] + \big[ \LL, \big[ \mathcal{K}, \ov{\mathcal{L}_2} \big] \big]. 
\end{equation*}
As $\big[ \mathcal{K} , \Lb \big]  =  - \Lb(k) \, \LL$, we have $\big[ \LL, \big[ \mathcal{K}, \Lb \big] \big] \equiv 0$ mod $T^{1,0}M$.
Similarly we have $\big[ \mathcal{K} ,\ov{ \mathcal{L}_2} \big]  =  - \ov{\mathcal{L}_2}(k) \, \LL$, from which we deduce that $\big[ \LL, \big[ \mathcal{K}, 
\ov{\mathcal{L}_2} \big] \big] \equiv 0$ mod $T^{1,0}M$. 
We thus have:
\begin{align*}
S & \equiv \ov{k} \, \big[ \big[ \mathcal{K}, \LL \big] , \Lb \big] + \big[ \big[ \mathcal{K}, \LL \big], 
\ov{\mathcal{L}_2} \big] && \text{mod $T^{1,0}M$} \\
  & \equiv  \big[ \big[ \mathcal{K}, \LL \big] ,  \ov{k} \, \Lb \big] + \big[ \big[ \mathcal{K}, \LL \big], 
\ov{\mathcal{L}_2} \big] && \text{mod $T^{1,0}M$} \\
  & \equiv \big[ \big[ \mathcal{K}, \LL \big] , \ov{\mathcal{K}} \big] && \text{mod $T^{1,0}M$}.
\end{align*}
The involutivity of the bundle $T^{1,0}M$ implies that $\big[ \mathcal{K}, \LL \big]$ belongs to $T^{1,0}M$. As $\mathcal{K}$ has been choosen to belong 
to the kernel of the Levi form of $M$,  $\big[ \big[ \mathcal{K}, \LL \big] , \ov{\mathcal{K}} \big]$ belongs to  $T^{1,0}M$. We thus have $S \equiv 0$ mod $T^{1,0}M$,
from which we deduce:
\begin{equation}
\label{eq:Kb}
\mathcal{K}(\ov{k}) = 0.
\end{equation}

We are now ready to compute $ \big[ \mathcal{K} , \ov{\mathcal{K}} \big]$:
\begin{align*}
\big[ \mathcal{K} , \ov{\mathcal{K}} \big]  &  =  \big[ k \, \LL + \mathcal{L}_2, \ov{k} \, \ov{\LL} + \ov{\mathcal{L}_2}  \big] \\ 
                                        &  =  k \, \ov{k} \, \big[ \LL , \Lb \big] + k \, \big[ \mathcal{L}_1 \ov{\mathcal{L}_2} \big] + 
\ov{k} \, \big[\mathcal{L}_2, \LL \big] + \ov{k} \, \big[ \mathcal{L}_2 , \Lb \big] + \big[ \mathcal{L}_2 , \ov{\mathcal{L}_2} \big]\\
& \quad +  k \, \Lkb \, \Lb + \mathcal{L}_2(\ov{k}) \, \ov{\mathcal{L}_1} - \ov{\mathcal{L}_2}(k) \, \LL - \ov{k} \, \Lbk \, \LL \\
                                         &  =  k \left( \ov{k} \, \big[ \LL , \Lb \big] + \big[ \LL \mathcal{L}_2 \big] \right) + 
\left( \ov{k} \, \big[ \mathcal{L}_2 , \Lb \big] + \big[ \mathcal{L}_2 , \ov{\mathcal{L}_2} \big] \right) +  \mathcal{K}(\ov{k}) \, \Lb - 
\ov{\mathcal{K}}(k) \, \LL  \\
& =  0 \qquad  \text{by $(\ref{eq:elim})$ and $(\ref{eq:Kb})$}.
\end{align*}
We now compute $\big[\mathcal{L}_1, \mathcal{T} \big]$.
We recall that from the definition of $\mathcal{T}$ we have $\mathcal{T} = l \, \partial_v $, where the function $l$ is defined by 
\begin{equation*}
l: = i \left(\ov{A^1_{z_1}} - A^1_{\ov{z_1}} + A^1 \ov{A^1_{v}} - \ov{A^1} A^1_{v} \right).
\end{equation*}
 We thus have:
\begin{align*}
\big[\mathcal{L}_1, \mathcal{T} \big] & = \big[ \partial_{z_1} + A^{1} \, \partial_v , l \, \partial_v \big] \\
& = \left( l_{z_1} + A^1 \, l_v - l \, A^1_v \right) \partial_v \\
& = P \, \mathcal{T}.
\end{align*}
where $P$ is the function defined on $M$ by
\begin{equation*}
P = \frac{ l_{z_1} + A^1 \, l_v - l \, A^1_v}{l}.
\end{equation*}
The last bracket that we need to compute is $\big[ \mathcal{K}, \mathcal{T} \big]$. Using the Jacobi identity, we get:
\begin{align*}
\big[\mathcal{K}, \mathcal{T} \big] & = i \big[ \mathcal{K} , \big[ \LL, \ov{\mathcal{L}_1} \big] \big] \\
& = i \,  \big[ \big[ \mathcal{K}, \LL \big], \Lb \big] + i \,  \big[ \LL , \big[ \mathcal{K}, \Lb \big] \big]  \\
& = i \, \big[ -\Lk \, \LL, \Lb \big] + i \,  \big[ \LL, - \Lbk \LL \big] \\
& = - \Lk \, \mathcal{T} + i \, \Lb \left( \Lk \right) \LL - i \, \LL \left( \Lbk \right) \LL \\
& = - \Lk \, \mathcal{T} - i \, \big[\LL, \Lb \big] (k) \, \LL \\
& = - \Lk \, \mathcal{T} - \mathcal{T}(k) \, \LL.
\end{align*}
The Jacobi identity actually implies other relations between the functions $P$, $k$ and their derivatives with respect to the five vector fields
$\mathcal{T}$, $\LL$, $\Lb$, $\mathcal{K}$ and $\ov{\mathcal{K}}$. The following computation of $\big[ \mathcal{K} , \big[ \mathcal{T} , \LL \big] \big] $ 
aims to determine an expression of $\mathcal{K}(P)$.
\begin{align*}
\big[ \mathcal{K} , \big[ \mathcal{T} , \LL \big] \big] & = - \big[ \mathcal{K} , P \, \mathcal{T} \big] \\
                                   & = - \mathcal{K}(P) \, \mathcal{T} - P \, \big[ - \Lk \, \mathcal{T} - \mathcal{T}(k) \, \LL \big] \\
                                   & = - \mathcal{K}(P) \, \mathcal{T} + P \, \Lk \, \mathcal{T} + P \, \mathcal{T}(k) \, \LL.
\end{align*}
On the other hand, the Jacobi identity gives:
\begin{align*}
\big[ \mathcal{K} , \big[ \mathcal{T} , \LL \big] \big] & =  \big[ \big[ \mathcal{K} , \T \big] , \LL \big] + 
\big[ \mathcal{T} , \big[ \mathcal{K} , \LL \big] \big] \\
                                   & = \big[ - \Lk \, \T - \Tk \, \LL , \LL \big] + \big[\T, - \Lk \, \LL \big] \\                    
                                   & = \LL \left( \Tk \right) \, \LL - \Lk \, \big[ \T , \LL \big] + \LL \left( \Lk \right) \, \T \\
                                   & \quad -  \Lk \, \big[ \T , \LL \big] - \T \left( \Lk \right) \, \LL \\
                                   & = \big[ \LL , \T \big](k) \, \LL + 2 \, \Lk \, \big[\LL, \T \big] + \LL \left( \Lk \right) \, \T \\
                                   & = P \, \Tk \, \LL + 2 \, \Lk \, P \, \T + \LL \left( \Lk \right) \, \T \\
                                   & = P \, \Tk \, \LL + \big(   2 \, \Lk \, P + \LL \left( \Lk \right) \big) \, \T .
\end{align*}
By identification of both results, we have:
\begin{equation*}
- \KK(P) + P \, \Lk =  2 \, \Lk \, P + \LL \left( \Lk \right),
\end{equation*}
that is:
\begin{equation*}
\mathcal{K}(P) = - P \, \Lk - \LL \left( \Lk \right).
\end{equation*}
We compute $\KK(\ov{P})$ in a similar way. We start with a direct computation of $\big[ \KK, \big[ \T, \Lb \big] \big]$:
\begin{align*}
\big[ \KK, \big[ \T, \Lb \big] \big]   & = - \big[ \mathcal{K} , \ov{P} \, \mathcal{T} \big] \\
                                   & = - \mathcal{K}(\ov{P}) \, \mathcal{T} - \ov{P} \, \big[ - \Lk \, \mathcal{T} - \mathcal{T}(k) \, \LL \big] \\
                                   & = - \mathcal{K}(\ov{P}) \, \mathcal{T} + \ov{P} \, \Lk \, \mathcal{T} + \ov{P} \, \mathcal{T}(k) \, \LL.
\end{align*}
The computation using the Jacobi identity gives:
\begin{align*}
\big[ \mathcal{K} , \big[ \mathcal{T} , \Lb \big] \big] & =  \big[ \big[ \mathcal{K} , \T \big] , \Lb \big] + 
\big[ \mathcal{T} , \big[ \mathcal{K} , \Lb \big] \big] \\
                                   & = \big[ - \Lk \, \T - \Tk \, \LL , \Lb \big] + \big[\T, - \Lb \, \LL \big] \\                    
                                   & = \Lb \left( \Lk \right) \, \T - \Lk \, \big[ \T , \Lb \big] + \Lb \left( \Tk \right) \, \LL \\
                                   & \quad - \Tk \big[ \LL, \Lb \big]  - \T \left( \Lbk \right) \, \LL - \Lbk \big[\T, \LL \big] \\
                                   &  = \Lb \left( \Lk \right) \, \T + \ov{P} \, \Lk \, \T + \big[\Lb, \T \big](k) \, \LL + i \, \Tk \, \T + P \, \Lbk \, \T \\
                                   &  = \left( \Lb \left( \Lk \right) + \ov{P} \, \Lk + P \, \Lbk + i \, \Tk \right) \T + \ov{P} \, \Tk \, \LL.
\end{align*}
Identification of both results gives:
\begin{equation*}
\KK(\ov{P}) = -P \, \Lbk  - \Lb \left( \Lk \right) - i \, \Tk.
\end{equation*}

Let us summarize the results that we have obtained so far. 
The five vector fields $\mathcal{T}$, $\LL$, $\Lb$, $\mathcal{K}$ and $\ov{\mathcal{K}}$ enjoy the following Lie bracket structure:
\begin{equation}
\label{eq:ls}
\begin{aligned}
\big[ \mathcal{T} , \LL \big] & = - P \, \mathcal{T}, \\
\big[ \mathcal{T} , \Lb \big] & = - \ov{P} \, \mathcal{T}, \\
\big[ \mathcal{T} , \mathcal{K} \big] & =  \Lk \, \mathcal{T} + \mathcal{T}(k) \,  \LL,  \\
\big[ \mathcal{T} , \ov{\mathcal{K}} \big] & =  \Lbkb \, \mathcal{T} + \mathcal{T}(\ov{k}) \,  \Lb,  \\
\big[ \LL, \Lb \big] & = - i \, \mathcal{T}, \\
\big[ \LL, \mathcal{K}  \big]& = \Lk \, \LL, \\
\big[ \LL, \ov{\mathcal{K}}  \big]&  = \Lkb \, \Lb, \\
\big[ \Lb , \mathcal{K} \big] &= \Lbk \, \LL, \\
\big[ \Lb, \ov{\mathcal{K}} \big] &= \Lbkb \, \Lb, \\
\big[ \mathcal{K}, \ov{\mathcal{K}} \big] & =0,
\end{aligned}
\end{equation}
where $P$ is a function defined on $M$.
The Jacobi identity implies the following two additional relations:
\begin{equation*}
\mathcal{K}(P) = - P \, \Lk - \LL \left( \Lk \right),
\end{equation*}
and
\begin{equation*}
\mathcal{K}(\ov{P}) =  -P \, \Lbk  - \Lb \left( \Lk \right) - i \, \Tk.
\end{equation*}

\subsection{Structure equations of the initial coframe}
From the formula
\begin{equation*}
d \omega(X,Y)=X \left( \omega(Y) \right) - Y \left(\omega(X) \right) - \omega \left( \big[ X, Y \big] \right) 
,\end{equation*}
where $X$ and $Y$ are two arbitrary vector fields and $\omega$ is a 1-form, we deduce from equation (\ref{eq:ls})
the structure equations enjoyed by the base coframe $(\rho_{0}, \kappa_{0}, \zeta_{0}, \ov{\kappa}_{0}, \ov{\zeta}_0)$, that is:
\begin{equation}
\label{eq:streq}
\begin{aligned}
d \rho_{0} & = P \left. \rho_0 \wedge \kappa_0 \right. - \Lk \left. \rho_0 \wedge \zeta_0 \right. + \ov{P} \left. \rho_0 \wedge \ov{\kappa_0} \right.
- \Lbkb  \left. \rho_0 \wedge \ov{\zeta_0} \right. + i \left.  \kappa_0 \wedge \ov{\kappa_0} \right., \\
d \kappa_{0} & = - \mathcal{T}(k) \, \left.  \rho_0 \wedge \zeta_0 \right. - \Lk \, \left. \kappa_0 \wedge \zeta_0 \right. 
+ \Lbk \, \left. \zeta_0 \wedge \ov{\kappa}_0 \right. 
, \\
d \zeta_0 & = 0, \\
d \ov{\kappa}_0 & = - \Tkb \, \left.  \rho_0 \wedge \ov{\zeta_0} \right.  
- \Lkb \, \left. \kappa_0 \wedge  \ov{\zeta_0} \right. - \Lbkb \, \left. \ov{\kappa_0} \wedge \ov{\zeta_0} \right. 
,\\
d \ov{ \zeta_{0}} & = 0.
\end{aligned}
\end{equation}
\subsection{Equivalence under biholomorphisms}
Let $\phi$ be a local biholomorphism of $\C^3$ such that $\phi(0)=0$ which preserves $M$, i.e. such that $\phi(M) = M$.
Then the restriction $\phi_{M}$ of $\phi$ to  $M$ is a local real analytic diffeomorphism of $M$ which satisfies the 
following two additional conditions:
\begin{enumerate}
\item{$\phi_{M}$ stabilizes the bundle $T^{1,0}M$.}
\item{$\phi_{M}$ stabilizes the kernel of the Levi form of $M$.}
\end{enumerate}
As a result, there are three functions $\ff$, $\cc$ and $\ee$ on $M$ such that:
\begin{equation*}
\phi_{M*}(\KK) = \ff \, \KK,
\end{equation*}
and
\begin{equation*}
\phi_{M*}(\LL) = \cc \, \LL + \ee \, \KK.
\end{equation*}
Of course, as $\phi_{M}$ is a real diffeomorphism, we shall also have:
\begin{equation*}
\phi_{M*}(\Kb) = \ov{\phi_{M*}(\KK)} =  \fb \, \Kb,
\end{equation*}
and
\begin{equation*}
\phi_{M*}(\Lb) = \ov{\phi_{M*}(\LL)} = \cb \, \Lb + \eb \, \Kb.
\end{equation*}
On the other hand there is a priori no special condition that shall be
satisfied by $\phi_{M*}(\T)$, except the fact that it shall be a real
vector field, because $\T$ is real. There are thus a real function
${\sf a}$ and two complex valued functions $\bb$ and $\dd$ such that:
\begin{equation*}
\phi_{M*}(\T) = 
{\sf a} \, \T + \bb \, \LL + \dd \, \KK + \bbb \, \Lb + \db \, \Kb.
\end{equation*}
We sum up these relations with the following matrix notation:
\begin{equation*}
\phi_{M*} \left( 
\begin{matrix}
\T \\
\LL \\
\KK \\
\Lb \\
\Kb
\end{matrix}
\right)
=  
\begin{pmatrix}
{\sf a} & {\sf b} & \dd & \bbb & \db \\
0 & {\sf c} & \ee & 0 & 0 \\
0 & 0 & {\sf f} & 0 & 0 \\
0 & 0 & 0 & \overline{\sf c} & \eb \\
0 & 0 & 0 & 0 & \overline{\sf f}
\end{pmatrix}
\cdot
\begin{pmatrix}
\T \\
\LL \\
\KK \\
\Lb \\
\Kb
\end{pmatrix}.
\end{equation*}  
As $\phi_{M*}$ is invertible, the functions ${\sf a}$, $\cc$ and $\ff$ shall not vanish on $M$.
The relation between the coframe $(\rho_0 , \kappa_0, \zeta_0, \ov{\kappa_0} , \ov{\zeta_0})$ and the coframe
$ \phi_M^* \left( \rho_0 ,\kappa_0, \zeta_0,\ov{\kappa_0} ,  \ov{\zeta_0} \right)$ is thus given by a plain transposition 
of the previous equation, that is:
\begin{equation*}
\phi_M^* \left( 
\begin{matrix}
\rho_0 \\
\kappa_0 \\
\zeta_0 \\
\ov{\kappa}_0 \\
\ov{\zeta}_0
\end{matrix}
\right)
=  
 \begin{pmatrix}
{\sf a} & 0 & 0 & 0 & 0 \\
{\sf b} & {\sf c} & 0 & 0 & 0 \\
{\sf d} & {\sf e} & {\sf f} & 0 & 0 \\
\overline{\sf b} & 0 & 0 & \overline{\sf c} & 0 \\
\overline{\sf d} & 0 & 0 & \overline{\sf e} & \overline{\sf f}
\end{pmatrix}
\cdot
\begin{pmatrix}
\rho_0 \\
\kappa_0 \\
\zeta_0 \\
\ov{\kappa_0} \\
\ov{\zeta_0}
\end{pmatrix}.
\end{equation*}  
In fact the function ${\sf a}$ shall satisfy another condition.
As $\T = i \, \big[ \LL , \Lb \big]$, we have
\begin{align*}
\phi_{M*}(\T) & = i \,  \big[ \phi_{M*}(\LL) , \phi_{M*}(\Lb) \big] \\
              & = i \,  \big[  \cc \, \LL + \ee \, \KK , \cb \, \Lb + \eb \, \Kb \big] \\
             & \equiv \cc \, \cb \, \T && \mod{T^{1,0}M},
\end{align*}
On the other hand we have from the definition of ${\sf a}$ that $\phi_{M*}(\T) \equiv {\sf a} \, \T  \mod{T^{1,0}M}$, which implies:
\begin{equation*}
{\sf a} = \cc \, \cb.
\end{equation*}
\subsection{Initial $G$-structure}
Let $G_1$ be the  $10$ dimensional real matrix Lie group whose elements are of the form:
\begin{equation*}
g := \begin{pmatrix}
{\sf c}\overline{\sf c} & 0 & 0 & 0 & 0 \\
{\sf b} & {\sf c} & 0 & 0 & 0 \\
{\sf d} & {\sf e} & {\sf f} & 0 & 0 \\
\overline{\sf b} & 0 & 0 & \overline{\sf c} & 0 \\
\overline{\sf d} & 0 & 0 & \overline{\sf e} & \overline{\sf f}
\end{pmatrix},
\end{equation*}
where $\cc$ and $\ff$ are non-zero complex numbers whereas $\bb$, $\dd$ and $\ee$ are arbitrary complex numbers.

Following~\cite{ Olver-1995}, let us introduce $5$ new one-forms
$\rho$, $\kappa$, $\zeta$, $\overline{ \kappa}$, $\overline{ \zeta}$
in accordance with the shape of the ambiguity matrix related to local
biholomorphic equivalences of such kinds of hypersurfaces:
\[
\begin{pmatrix}
\rho
\\
\kappa
\\
\zeta
\\
\overline{\kappa}
\\
\overline{\zeta}
\end{pmatrix}
:=
g \cdot 
\begin{pmatrix}
\rho_0
\\
\kappa_0
\\
\zeta_0
\\
\overline{\kappa_0}
\\
\overline{\zeta_0}
\end{pmatrix}
,\]
that is to say, in expanded form:
\[
\aligned
\rho
&
:=
{\sf c}\overline{\sf c}\,\rho_0,
\\
\kappa
&
:=
{\sf b}\,\rho_0+{\sf c}\,\kappa_0,
\\
\zeta
&
:=
{\sf d}\,\rho_0+{\sf e}\,\kappa_0+{\sf f}\,\zeta_0,
\\
\overline{\kappa}
&
:=
\overline{\sf b}\,\rho_0+\overline{\sf c}\,\overline{\kappa_0},
\\
\overline{\zeta}
v&
:=
\overline{\sf d}\,\rho_0+\overline{\sf e}\,\overline{\kappa_0}
+\overline{\sf f}\,\overline{\zeta_0}.
\endaligned
\]
By inverting the matrix:\[
\left(\!\!
\begin{array}{c}
\rho_0
\\
\kappa_0
\\
\zeta_0
\\
\overline{\kappa_0}
\\
\overline{\zeta_0}
\end{array}
\!\!\right)
=
\left(\!\!
\begin{array}{ccccc}
\frac{1}{{\sf c}\overline{\sf c}} & 0 & 0 & 0 & 0
\\
\frac{-{\sf b}}{{\sf c}^2\overline{\sf c}}
& \frac{1}{\sf c} & 0 & 0 & 0
\\
\frac{{\sf b}{\sf e}-{\sf c}{\sf d}}{{\sf c}^2
\overline{\sf c}{\sf f}} & -\,\frac{\sf e}{{\sf c}{\sf f}}
& \frac{1}{\sf f} & 0 & 0
\\
\frac{-\overline{\sf b}}{{\sf c}\overline{\sf c}^2} &
0 & 0 & \frac{1}{\overline{\sf c}} & 0
\\
\frac{\overline{\sf b}\overline{\sf e}
-\overline{\sf c}\overline{\sf d}}{{\sf c}
\overline{\sf c}^2 \overline{\sf f}} &
0 & 0 & -\frac{\overline{\sf e}}{\overline{\sf c}
\overline{\sf f}} &
\frac{1}{\overline{\sf f}}
\end{array}
\!\!\right)
\left(\!\!
\begin{array}{c}
\rho
\\
\kappa
\\
\zeta
\\
\overline{\kappa}
\\
\overline{\zeta}
\end{array}
\!\!\right),
\]
we find how the $\{\}_0$-indexed forms express in terms of the lifted
complete forms:

\begin{equation}
\label{eq:change}
\begin{aligned}
\rho_0
&
=
\frac{1}{{\cc} \overline{\sf c}}\,\rho,
\\
\kappa_0
&
=
-\,\frac{\sf b}{{\sf c}^2 \cb}\,\rho
+
\frac{1}{\sf c}\,\kappa,
\\
\zeta_0
&
=
\frac{{\sf b}{\sf e}-{\sf c}{\sf d}}{\cc^2 \cb
{\sf f}}\,\rho
-
\frac{{\sf e}}{{\sf c}{\sf f}}\,\kappa
+
\frac{1}{\sf f}\,\zeta,
\\
\overline{\kappa_0}
&
=
-\,\frac{\overline{\sf b}}{{\sf c}\cb^2}\,
\rho
+
\frac{1}{\cb}\,\overline{\kappa},
\\
\overline{\zeta_0}
&
=
\frac{\overline{\sf b}{\overline{\sf e}}
-\overline{\sf c}\overline{\sf d}}{{\sf c}
\cb^2 \overline{\sf f}}\,\rho
-
\frac{\overline{\sf e}}{\overline{\sf c}\overline{\sf f}}\
\overline{\kappa}
+
\frac{1}{\overline{\sf f}}\,\overline{\zeta}.
\end{aligned}
\end{equation}

\section{Absorption of torsion and normalization:
first loop} \label{section:step1}
\subsection{Lifted structure equations}
We apply the Cartan's method as explained in ~\cite{ Olver-1995}.
The first step is to compute the structure equations for the lifted coframe.
With the matrix notations 
\begin{alignat*}{1}
\omega_{0} :=
\begin{pmatrix}
\rho_0 \\ \kappa_{0} \\ \zeta_{0} \\ \ov{\kappa}_{0} \\ \ov{\zeta}_0
\end{pmatrix}, & \qquad
\omega:=
\begin{pmatrix}
\rho \\ \kappa \\ \zeta \\ \ov{\kappa} \\ \ov{\zeta}
\end{pmatrix},
\end{alignat*}
we have 
\begin{equation*}
\omega = g \cdot \omega_{0}.
\end{equation*}
As a result, the structure equations for the lifted coframe are related to those of the base coframe by the relation:
\begin{equation} 
\label{eq:str}
d \omega =  dg \cdot g^{-1} \wedge \omega + g \cdot d \omega_{0}.
\end{equation}
The term $ dg \cdot g^{-1} \wedge \omega$ depends only on the structure equations of $G_1$ and is expressed through its Maurer-Cartan forms.
The term $ g \cdot d \omega_{0}$  contains the so-called torsion coefficients of the $G_1$-structure. It is computed easily in terms of the forms 
$\rho$, $\kappa$, $\zeta$, $\ov{\kappa}$, $\ov{\zeta}$, by 
applying the linear change (\ref{eq:change}) in the expression of $d \omega_0$, which is given by the set of equations (\ref{eq:streq}), 
and a matrix multiplication by $g$.

We start with the expression of the Maurer-Cartan forms of $G_1$. They are given by the
linear independant entries of the matrix $dg \cdot g^{-1}$.
An easy computation gives:
\begin{equation*}
dg \cdot g^{-1} =
\begin{pmatrix}
\alpha^1 + \ov{\alpha^1} & 0 & 0 & 0 & 0 \\
\alpha^2 & \alpha^1  & 0 &0 & 0 \\
\alpha^3 & \alpha^4 & \alpha^5 & 0 & 0 \\
\ov{\alpha^2} & 0 & 0 & \ov{\alpha^1} & 0\\
\ov{\alpha^3} & 0 & 0 & \ov{\alpha^4} & \ov{\alpha^5} \\
\end{pmatrix}
,
\end{equation*}
where
\begin{equation*}
\begin{aligned}
\alpha^1  &: = \frac{ d \cc}{\cc}, \\
\alpha^2 &:= \frac{ d \bb}{\cc \cb} - \frac{\bb \, d \cc}{\cc^2}{\cb}, \\
\alpha^3  &:= \frac{ d \dd}{\cc \cb} - \frac{\bb \,  d \ee}{\cc^2 \cb} + \frac{\left( -\dd \cc + \ee \bb \right) d {\sf f}}{\cc^2 \cb {\sf f}}, \\
\alpha^4 &:= \frac{d \ee}{\cc} - \frac{\ee \, d {\sf f}}{\cc \sf f}, \\
\alpha^5 &:= \frac{d {\sf f}}{\sf f}.
\end{aligned}
\end{equation*}

The next step is to express the structure equations of the lifted coframe from equation (\ref{eq:str}) as explained above.
Rather lenghty but straigtforward computations give:
\begin{dgroup*}
\begin{dmath*}
d \rho  =
\alpha^1 \wedge \rho + \overline{\alpha^{1}} \wedge \rho \\
+
T^{\rho}_{\rho \kappa} \, \rho \wedge \kappa
+
T^{\rho}_{\rho \zeta} \, \rho \wedge \zeta
+
T^{\rho}_{\rho \overline{\kappa}} \, \rho \wedge \overline{\kappa}
+
T^{\rho}_{\rho \overline{\zeta}} \, \rho \wedge \overline{\zeta}
+
i \, \kappa \wedge \overline{\kappa}
,\end{dmath*}
\begin{dmath*}
d \kappa
=
\alpha^{1} \wedge \kappa + \alpha^{2} \wedge \rho \\
+
T^{\kappa}_{\rho \kappa} \, \rho \wedge \kappa
+
T^{\kappa}_{\rho \zeta} \, \rho \wedge \zeta
+
T^{\kappa}_{\rho \overline{\kappa}} \, \rho \wedge \overline{\kappa}
+
T^{\kappa}_{\rho \overline{\zeta}} \, \rho \wedge \overline{\zeta}
+
T^{\kappa}_{\kappa \zeta} \, \kappa \wedge \zeta 
+
T^{\kappa}_{\kappa \overline{\kappa}} \, \kappa \wedge \overline{\kappa}
+
T^{\kappa}_{\zeta \overline{\kappa}} \, \zeta \wedge \overline{\kappa},
\end{dmath*}
\begin{dmath*}
d \zeta
=
\alpha^{3} \wedge \rho + \alpha^{4} \wedge \kappa + \alpha^{5} \wedge \zeta \\
+
T^{\zeta}_{\rho \kappa} \, \rho \wedge \kappa
+
T^{\zeta}_{\rho \zeta} \, \rho \wedge \zeta
+
T^{\zeta}_{\rho \overline{\kappa}} \, \rho \wedge \overline{\kappa}
+
T^{\zeta}_{\rho \overline{\zeta}} \,\left. \rho \wedge \overline{\zeta} \right.
+
T^{\zeta}_{\kappa \zeta} \, \kappa \wedge \zeta 
+
T^{\zeta}_{\kappa \overline{\kappa}} \, \kappa \wedge \overline{\kappa}
+
T^{\zeta}_{\zeta \overline{\kappa}} \, \zeta \wedge \overline{\kappa},
\end{dmath*}
\end{dgroup*}
where the expressions of the torsion coefficients $T^{\smallbullet}_{\smallbullet \smallbullet}$ are given by the following equations: 
\[
T^{\rho}_{\rho \kappa} = i\, {\frac { \overline{\sf b}}{{\sf c}\overline{ \sf c}}} + {\frac {\sf e  }{{\sf c}{ \sf f}}}\, 
\Lk + {\frac {P}{\sf c}},
\]

\[
T^{\rho}_{\rho \zeta} = -{\frac {\Lk }{\sf f}},
\]

\[
T^{\rho}_{\rho \overline{\kappa}}
=
-i \, {\frac {\sf b}{{\sf c} \overline{ \sf c}}}
+
{\frac {\overline{\sf e}}{\overline{ \sf c}\overline{\sf f}}} \, \Lbkb + {\frac {\overline{P}
}{\overline{ \sf c}}},
\]

\[
T^{\rho}_{\rho \overline{\zeta}} = -\frac {\Lbkb }{\overline{\sf f}}
,\]

\[
T^{\kappa}_{\rho \kappa}
=
-
\frac {\ee }{\cc \cb \ff} \, \Tk
-
\frac {
\overline{\sf b} \ee }{{\sf c} \overline{ \sf c}^2{\sf f}} \, \Lbk
-
\frac {{\sf d} }{{\sf c}\overline{ \sf c}{\sf f}}\, \Lk +
 i\, {\frac
{{\sf b}\overline{\sf b}}{{\sf c}^{2} \overline{ \sf c} ^{2}}}
+
{\frac {{\sf b} {\sf e}}{{\sf c}^{2}\overline{ \sf c}{\sf f}}}\, \Lk
 +
 {\frac {{\sf b}}{{\sf c}^{2}
\overline{ \sf c}}}\,P
,\]

\[
T^{\kappa}_{\rho \zeta}
=
{\frac {\overline{\sf b}}{
\overline{ \sf c} ^{2}{\sf f} }}\,\Lbk - \frac{1}{\cb \ff} \, \Tk 
,\]

\[
T^{\kappa}_{\rho \overline{\kappa}}
=
-\frac {\sf d }{\overline{ \sf c}
  ^2 \ff }\,\Lbk
 +
 \frac {{\sf b} {\sf e}}{{\sf c} 
\overline{ {\sf c}} ^2 \ff}\,\Lbk 
-
i\, {\frac {{\sf b}^{2}}{{\sf c}^{2} 
\overline{ \sf c} ^{2}}}
+
{\frac {{\sf b}\overline{\sf e} }{{\sf c} \overline{ \sf c} ^{2}\overline{\sf f}}\, \Lbkb
}
+
\frac {{\sf b}}{{\sf c} \overline{ \sf c} ^2}\,\overline{P}
,\]

\[
T^{\kappa}_{\rho \overline{\zeta}}
=
-\frac {  \sf b}{{\sf c}\overline{ \sf c}
\overline{\sf f}} \,\Lbkb \]

\[
T^{\kappa}_{\kappa \zeta}
=
-{\frac {\Lk }{\sf f}}
,\]

\[
T^{\kappa}_{\kappa \overline{\kappa}}
=
-{\frac {\sf e}{\overline{ \sf c} \ff }}\,\Lbk 
+
i\, {\frac {\sf b}{
{\sf c}\overline{ \sf c}}}
,\]

\[
T^{\kappa}_{\zeta \overline{\kappa}}
=
\frac {{\sf c}}{\overline{ \sf c} \ff}\,\Lbk 
,\]

\[
T^{\zeta}_{\rho \kappa}
=
-
\frac {\ee^2  }{{\sf c}^2 \overline{ \sf c} \ff} \Tk
-
\frac { \bb {\sf 
e}^{2}}{{\sf c}^{2}
\overline{ \sf c}  ^2 \ff} \, \Lbk 
+
i\, {\frac {\overline{\sf b} \dd}{{\sf c}^{2} 
\overline{ \sf c} ^2}}
+
{\frac {\sf d}{{\sf c}^{2}\overline{ \sf c}}}\,P
,\]

\[
T^{\zeta}_{\rho \zeta}
=
- \frac{\ee}{\cc \cb \ff} \, \Tk
+
{\frac { {\sf b}\overline{\sf e}}{{\sf c} 
\overline{ \sf c} ^{2} \ff}}\, \Lbk 
+
{\frac { {\sf b} {\sf e}}{{\sf c}^{2}
\overline{ \sf c}{\sf f}}} \, \Lk
-
{\frac {{ \sf d} }{{\sf c}\overline{ \sf c}{\sf f}}}\, \Lk
,\]

\[
T^{\zeta}_{\rho \overline{\kappa}}
=
-
{\frac { {\sf d}{ \sf e}}{{\sf c}  \overline{ \sf c}
 ^{2}\ff}}\, \Lbk
 +
{\frac { \bb {\sf e}^{2}}{{\sf c}^{2
} \overline{ \sf c}  ^{2} \ff}} \, \Lbk
-
i\,{\frac {{\sf b}{\sf d}}{{\sf c}^{2} 
\overline{ \sf c}  ^{2}}}
+
{\frac {{ \sf d}\overline{\sf e}}{{\sf c}  \overline{ \sf c}^{2}\overline{\sf f}} \, \Lbkb
}
+
{\frac {{ \sf d}}{{\sf c}  \overline{ \sf c}  ^{2}}}\overline{P}
,\]

\[
T^{\zeta}_{\rho \overline{\zeta}}
=
-{\frac {{ \sf d}}{{\sf c}\overline{ \sf c}
\overline{\sf f}}}\,\Lbkb
,\]

\[
T^{\zeta}_{\kappa \zeta}
=
-{\frac {{ \sf e}}{{\sf c}\overline{ \sf c}
\overline{\sf f}}}\,\Lk 
,\]

\[
T^{\zeta}_{\kappa \overline{\kappa}}
=
-
{\frac {{\sf e}^{2}}{{\sf c}\overline{ \sf c}{\sf f}}}\,\Lbk
+
{
i\,\frac {\sf d}{{\sf c}\overline{ \sf c}}}
,\]

\[
T^{\zeta}_{\zeta \overline{\kappa}}
=
{\frac { {\sf e}}{\overline{ \sf c} \ff}}\,\Lbk
.\]

\subsection{Normalization of the group parameter $\ff$}
We now proceed with the absorption step of Cartan's method. 
We introduce the modified Maurer-Cartan forms $\tilde{\alpha}^i$, 
which are a related to the $1$-forms $\alpha^i$ by the relations:
\begin{equation*}
\tilde{\alpha}^i := \alpha^i -  x_{\rho}^i \, \rho \, -  x_{\kappa}^i \, \kappa - x_{\zeta}^i \, \zeta \, - \, x_{\overline{\kappa}}^i \, \overline{\kappa} \, - \, 
x_{\overline{\zeta}}^i \, \overline{\zeta},
\end{equation*}
where $x^1$, $x^2$, $x^3$, $x^4$ and $x^5$ are arbitrary complex-valued functions.
The previously written structure equations take the new form:
\begin{dgroup*}
\begin{dmath*}
 d \rho = \tilde{\alpha}^1 \wedge \rho + \overline{\tilde{\alpha}^1} \wedge \rho\\
+ \left( T_{\rho \kappa}^{\rho}- x_\kappa^1 - x_{\overline{\kappa}}^1 \right)  \rho \wedge \kappa \, + \, \left(T_{\rho \zeta}^{\rho} - x_{\kappa}^1 - 
\overline{  x_{\overline{\zeta}}^1} \right)  \rho \wedge \zeta  \,  + \, \left(T_{\rho \overline{\kappa}}^{\rho} - x_{\overline{\kappa}}^1-    
 \overline{ x_{\kappa}^1} \right)  \rho \wedge \overline{\kappa}  \, + \, \left( T_{\rho \overline{\zeta}}^{\rho} - x_{\zeta}^1 - x_{\overline{\zeta}}^1 \right)
 \rho \wedge \overline{\zeta}  + i \, \kappa \wedge \overline{\kappa},
\end{dmath*}
\begin{dmath*}
d \kappa = \tilde{\alpha}^{1} \wedge \kappa + \tilde{\alpha}^{2} \wedge \rho \\
+ \, \left(T_{\rho \kappa}^{\kappa} - x_\kappa^2 + x_\rho^1     \right)  \rho \wedge \kappa \, + \, \left(T_{\rho \zeta}^{\kappa} - x_{\kappa}^2 \right) \,
\rho \wedge \zeta \,  + \, \left( T_{\rho \overline{\kappa}}^{\kappa} - x_{\overline{\kappa}}^2 \right)  \rho \wedge \overline{\kappa}  + 
\left( T^{\kappa}_{\rho \overline{\zeta}} - x_{\overline{\zeta}}^2 \right)  \rho \wedge \overline{\zeta} \, + \, \left( T_{\kappa \zeta}^{\kappa} + x_{\zeta}^1 
\right)\, \kappa \wedge \zeta \,  + \, \left( T_{\kappa \overline{\kappa}}^{\kappa} - x_{\overline{\kappa}}^1 \right)  \kappa \wedge \overline{\kappa} \, + \, 
T_{\zeta \overline{\kappa}}^{\kappa} \,  \zeta \wedge \overline{\kappa} \, + \, \left( T_{\kappa \overline{\zeta}}^1 - x_{\kappa \overline{\zeta} }^1 \right) \,
\kappa \wedge \zeta,
\end{dmath*}
\begin{dmath*}
 d \zeta = \tilde{\alpha}^3 \wedge \rho + \tilde{\alpha}^4 \wedge \kappa + \tilde{\alpha}^5 \wedge \zeta  \\
+ \left( T_{\rho \kappa}^{\zeta}  - x_{\kappa}^3 + x_{\rho}^4 \right)  \rho \wedge \kappa  +  \left( T_{\rho \zeta}^{\zeta} - x_{\zeta}^3 + x_{\rho}^5 
\right)  \rho \wedge \zeta  + \left(T_{\rho \kappa}^{\zeta} - x_{\overline{\kappa}}^3  \right) \left. \rho \wedge \overline{\kappa} \right.  + 
\left( T_{\rho \overline{\zeta}}^{\zeta} -x_{\overline{\zeta}}^3 \right)  \rho \wedge \overline{\zeta}   \,  + \, \left( T_{\kappa \overline{\kappa}}^{\zeta} - 
x_{\overline{\kappa}}^4 \right)  \kappa \wedge \overline{\kappa}  + 
\left( T_{\zeta 
\overline{\kappa}}^{\zeta} - x_{\overline{\kappa}}^5 \right)  \zeta \wedge \overline{\kappa} + \left( x_{\kappa}^5 - x_{\zeta}^4 \right)  \kappa \wedge \zeta - x_
{\overline{\kappa}}^4  \,
 \kappa \wedge \overline{\kappa} \, + \, \left(x_{\overline{\kappa}}^5 - x_{\overline{\zeta}}^4 \right)  \overline{\kappa} \wedge \zeta - x_{\overline{\zeta}}^5 
 \, \zeta \wedge \overline{\zeta}.
\end{dmath*}
\end{dgroup*}

We then choose  $x^1$, $x^2$, $x^3$, $x^4$ and $x^5$ in a way that eliminate as many torsion coefficients as possible.
We easily see that the only coefficient which can not be absorbed is the one in front of $\zeta \wedge \ov{\kappa}$ in $d \kappa$, because it does not depend 
on the $x^i$'s.
We choose the normalization 
\begin{equation*}
T_{\zeta \overline{\kappa}}^{\kappa} = 1,
\end{equation*}
which yields to:
\begin{equation*}
{\sf f} = \frac{\cc}{\cb} \, \overline{ \mathcal{L}_1}(k).
\end{equation*}
We notice that the absorbed structure equations take the form:
\begin{align*}
d \rho & = \tilde{\alpha}^1 \wedge \rho + \overline{\tilde{\alpha}^1} \wedge \rho + i \, \kappa \wedge \overline{\kappa}, \\
d \kappa & =  \tilde{\alpha}^{1} \wedge \kappa + \tilde{\alpha}^{2} \wedge \rho  + \zeta \wedge \overline{\kappa}, \\
d \zeta &  = \tilde{\alpha}^3 \wedge \rho + \tilde{\alpha}^4 \wedge \kappa + \tilde{\alpha}^5 \wedge \zeta. 
\end{align*}
As a preliminary step towards the second loop of the algorithm, we return to the expression of the lifted coframe. The normalization of ${\sf f}$ 
gives the new relation:
\begin{equation}
\begin{pmatrix}
\rho \\
\kappa \\
\zeta \\
\overline{\kappa}\\
\overline{\zeta} 
\end{pmatrix}
=
\begin{pmatrix}
{\sf c} \overline{\sf c} & 0 &0 &0 &0 \\
{\sf b} & {\sf c} & 0 & 0 & 0 \\
{\sf d} & {\sf e} & \frac{\sf c}{\overline{\sf c}} \overline{ \mathcal{L}_{1}}(k) & 0 & 0 \\
{\sf \overline{b}} & 0 & 0 & {\sf \overline{c}} & 0 \\
0 & 0 & {\sf \overline{d}} & {\sf \overline{e}} & \frac{\sf c}{\sf \overline{c}} \overline{ \mathcal{L}_{1}}(k)
\end{pmatrix}
\cdot
\begin{pmatrix}
\rho_0 \\
\kappa_0 \\
\zeta_0 \\
\overline{\kappa}_0 \\
\overline{\zeta}_0
\end{pmatrix}
.\end{equation}

Let us interpret this in the framework of $G$-structures.  We introduce the new one-form 
\begin{equation}
\hat{\zeta}_{0} =  \overline{ \mathcal{L}_{1}}(k) \cdot \zeta_0
,\end{equation}
such that the previous equation rewrites:
\begin{equation}
\begin{pmatrix}
\rho \\
\kappa \\
\zeta \\
\overline{\kappa}\\
\overline{\zeta} 
\end{pmatrix}
=
\begin{pmatrix}
{\sf c} \overline{\sf c} & 0 &0 &0 &0 \\
{\sf b} & {\sf c} & 0 & 0 & 0 \\
{\sf d} & {\sf e} & \frac{\sf c}{\overline{\sf c}} & 0 & 0 \\
{\sf \overline{b}} & 0 & 0 & {\sf \overline{c}} & 0 \\
0 & 0 & {\sf \overline{d}} & {\sf \overline{e}} & \frac{\sf c}{\sf \overline{c}}
\end{pmatrix}
\cdot
\begin{pmatrix}
\rho_0 \\
\kappa_0 \\
\hat{\zeta}_0 \\
\overline{\kappa}_0 \\
\overline{\hat{\zeta}}_0
\end{pmatrix}.
\end{equation}

We thus have reduced the $G_{1}$ equivalence problem to a $G_2$ equivalence problem, where $G_2$ is the $8$ dimensional real matrix 
Lie group whose elements are of the form
\begin{equation*}
g = \begin{pmatrix}
{\sf c} \overline{\sf c} & 0 &0 &0 &0 \\
{\sf b} & {\sf c} & 0 & 0 & 0 \\
{\sf d} & {\sf e} & \frac{\sf c}{\overline{\sf c}} & 0 & 0 \\
{\sf \overline{b}} & 0 & 0 & {\sf \overline{c}} & 0 \\
0 & 0 & {\sf \overline{d}} & {\sf \overline{e}} & \frac{\sf c}{\sf \overline{c}}
\end{pmatrix}.
\end{equation*}
The last task that we need to perform before the second loop of the algorithm is to compute the new structures equations enjoyed by the base coframe 
$(\rho_0, \kappa_0, \hat{\zeta}_0, \kappa_0, \overline{\hat{\zeta}_0})$.
We easily get:
\begin{dgroup*}
\begin{dmath*}
d \rho_{0} = P \, \rho_{0} \wedge \kappa_{0} - \frac{\mathcal{L}_{1}(k)}{\overline{\mathcal{L}_1}(k)} \, \rho_{0} \wedge \hat{\zeta}_0 +\,  \overline{P} \, 
\rho_0 \wedge \overline{\kappa_0} -\frac{ \overline{\mathcal{L}_1}(\overline{k})}{\mathcal{L}_1(\overline{k})} \, \rho_0 \wedge \overline{\hat{\zeta}_{0}} +
 i\, \kappa_0 \wedge \overline{\kappa_0},
\end{dmath*}
\begin{dmath*}
d\kappa_{0} =- \frac{\Tk}{\Lbk} \, \rho_0 \wedge \hat{\zeta}_0 - \frac{\Lk}{\Lbk} \, \kappa_0 \wedge \hat{\zeta}_0  + \hat{\zeta}_{0} \wedge \ov{\kappa_{0}} 
,\end{dmath*}

\begin{dmath*}
d \hat{\zeta}_0 = \frac{\mathcal{T} \left( \Lbk \right)}{\Lbk} \, \rho_0 \wedge \hat{\zeta}_0 + \frac{\LL \left( \Lbk \right)}{\Lbk} \, \kappa_0 \wedge 
\hat{\zeta}_0 - \frac{\LL \left( \Lbk \right)}{\Lbk} \, \hat{\zeta}_0 \wedge \ov{\kappa_0
}+ \frac{\Lbkb}{\Lkb}\,  \hat{\zeta}_0 \wedge \ov{\hat{\zeta}_0}
\end{dmath*}.
\end{dgroup*}

\section{Absorption of torsion and normalization: second loop} \label{section:step2}
\subsection{Lifted structure equations}
The Maurer forms of the $G_{2}$ are given by the independant entries of the matrix $d g \cdot g^{-1}$. A straightforward computation gives
\begin{equation*}
 dg \cdot g^{-1}
= 
\begin{pmatrix}
\beta^1 + \ov{\beta^1} & 0 &0 &0 &0 \\
\beta^2 & \beta^1 & 0 & 0& 0 \\
\beta^3 & \beta^4 & \beta^1 - \ov{\beta^1} & 0 &0 \\
\ov{\beta^2} & 0 & 0 & \ov{\beta^1} & 0 \\
\ov{\beta^3} & 0 & 0 &  \ov{\beta^4} & - \beta^1 + \ov{\beta^1}
\end{pmatrix}
,\end{equation*}

where the forms $\beta^1$, $\beta^2$, $\beta^3$ and $\beta^4$ are defined by
\begin{gather*}
\beta^1 : = \frac{d {\sf c}}{\sf c}, \\
\beta^2 := \frac{d \bb}{\cc \cb}- \frac{\bb d \cc}{\cc^2  \cb}, \\
\beta^3 := \frac{\left( - \dd \cc + \ee \bb \right) d \cc}{\cc^3 \cb} - \frac{\left(- \dd \cc + \ee \bb \right) d \cb}{\cc^2 \cb^2} + \frac{d \dd}{\cc \cb} -
 \frac{\bb d \ee}{\cc^2 \cb}, \\
\beta^4 := - \frac{\ee d \cc}{\cc^2} + \frac{\ee d \cb}{\cb \cc} + \frac{d \ee}{\cc }.
\end{gather*}
Using formula (\ref{eq:str}), we get the structure equations for the lifted coframe $(\rho, \kappa, \zeta, \ov{\kappa}, \ov{\zeta})$ from those of the base coframe
$(\rho_0, \kappa_0, \hat{\zeta}_0, \ov{\kappa_0}, \ov{\hat{\zeta}_0})$ by a matrix multiplication and a linear change of coordinates, as in the first loop:
\begin{dgroup*}
\begin{dmath*}
d \rho
=
\beta^{1} \wedge \rho + \overline{\beta^{1}} \wedge \rho \\
+
U^{\rho}_{\rho \kappa} \, \rho \wedge \kappa
+
U^{\rho}_{\rho \zeta} \, \rho \wedge \zeta
+
U^{\rho}_{\rho \overline{\kappa}} \, \rho \wedge \overline{\kappa}
+
U^{\rho}_{\rho \overline{\zeta}} \, \rho \wedge \overline{\zeta}
+
i \, \kappa \wedge \overline{\kappa}
,\end{dmath*}

\begin{dmath*}
d \kappa
=
\beta^{1} \wedge \kappa + \beta^{2} \wedge \rho \\
+
U^{\kappa}_{\rho \kappa} \, \rho \wedge \kappa
+
U^{\kappa}_{\rho \zeta} \, \rho \wedge \zeta
+
U^{\kappa}_{\rho \overline{\kappa}} \, \rho \wedge \overline{\kappa}
\\ 
+
U^{\kappa}_{\rho \overline{\zeta}} \, \rho \wedge \overline{\zeta}
+
U^{\kappa}_{\kappa \zeta} \, \kappa \wedge \zeta
+
U^{\kappa}_{\kappa \overline{\kappa}} \, \kappa \wedge \overline{\kappa}
+
\zeta \wedge \overline{\kappa}
,\end{dmath*}

\begin{dmath*}
d \zeta
=
\beta^{3} \wedge \rho + \beta^{4} \wedge \kappa + \beta^{1} \wedge \zeta - \overline{\beta^{1}} \wedge \zeta \\
+
U^{\zeta}_{\rho \kappa} \, \rho \wedge \kappa
+
U^{\zeta}_{\rho \zeta} \, \rho \wedge \zeta
+
U^{\zeta}_{\rho \overline{\kappa}} \, \rho \wedge \overline{\kappa}
+
U^{\zeta}_{\rho \overline{\zeta}} \, \rho \wedge \overline{\zeta}
+
U^{\zeta}_{\kappa \zeta} \, \kappa \wedge \zeta
+
U^{\zeta}_{\kappa \overline{\kappa}} \, \kappa \wedge \overline{\kappa}
+
U^{\zeta}_{\kappa \overline{\zeta}} \, \kappa \wedge \overline{\zeta}
+
U^{\zeta}_{\zeta \overline{\kappa}} \, \zeta \wedge \overline{\kappa}
+
U^{\zeta}_{\zeta \overline{\zeta}} \, \zeta \wedge \overline{\zeta}
.\end{dmath*}
\end{dgroup*}

The torsion coefficients $U^{\smallbullet}_{\smallbullet \smallbullet}$ are given by:
\[
U^{\rho}_{\rho \kappa} = i\,  {\frac {\overline{\sf b}}{{\sf c}\overline{ \sf c}}}
+
{\frac { {\sf e}\overline{ \sf c}}{{\sf c}^{2}}}\, \frac{\Lk }{\Lbk }
 +
 {\frac {P}{\sf c}}
,\]

\[
U^{\rho}_{\rho \zeta} = -
{\frac {\overline{ \sf c}}{{\sf c}}}\, \frac{\Lk }{\Lbk}
,\]

\[
U^{\rho}_{\rho \overline{\kappa}}
=
-i {\frac {{\sf b}}{{\sf c}\overline{ \sf c}}}
+
{\frac {\overline{\sf e}{\sf c}}{  \overline{ \sf c}^2}} \, \frac{\Lbkb}{\Lkb }
+
{\frac {\overline{P}}{\overline{ \sf c}}}
,\]

\[
U^{\rho}_{\rho \overline{\zeta}} = 
-
{\frac {{\sf c}}{\overline{ \sf c}}}\, \frac{\Lbkb }{\Lkb }
,\]

\[
U^{\kappa}_{\rho \kappa}
=
-
\frac{\ee}{\cc^2} \, \frac{\Tk}{\Lbk}
-
{\frac { {\sf e}
\overline{\sf b}}{{\sf c}^{2}\overline{ \sf c}}}
-
{\frac {{ \sf d}}{{\sf c}^
{2}}} \, \frac{\Lk}{\Lbk }
+
i\, {\frac {{\sf b}\overline{\sf b}}{{\sf c}^{2}
  \overline{ \sf c}^2}}
  +
  {\frac {{\sf b} {\sf e}}{{\sf c}
^{3}}}\, \frac{\Lk }{\Lbk }
+
{\frac {\sf b}{{\sf c}^{2}\overline{ \sf c}}}\, P
,\]

\[
U^{\kappa}_{\rho \zeta}
=
{\frac {\overline{\sf b}}{{\sf c}\overline{ \sf c}}}
- \frac{1}{\cc} \, \frac{\Tk}{\Lbk},\]

\[
U^{\kappa}_{\rho \overline{\kappa}}
=
-{\frac {\sf d}{{\sf c}\overline{ \sf c}}}
+
{\frac { {\sf e}{\sf b}}{{\sf c}^{2}\overline{ \sf c}}}
-
i\, {\frac {
{\sf b}^{2}}{{\sf c}^{2}  \overline{ \sf c}^2}}
+
{\frac {{\sf b}
\overline{\sf e}}{  
\overline{ \sf c} ^{3}}}\, \frac{\Lbkb}{\Lkb }
+
{\frac {\sf b
}{{\sf c}  \overline{ \sf c}^2}}\, \overline{P}
,\]

\[
U^{\kappa}_{\rho \overline{\zeta}}
=
-{\frac {\sf b}{ 
\overline{ \sf c}^{2}}}\, \frac{\Lbkb}{\Lkb }
,\]

\[
U^{\kappa}_{\kappa \zeta}
=
-{\frac {\overline{ \sf c}}{{\sf c}}}\, \frac{\Lk }{\Lbk }
,\]

\[
U^{\kappa}_{\kappa \overline{\kappa}}
=
-{\frac {\sf e}{\sf c}}+i\, {\frac {{\sf b}}{{\sf c}\overline{ \sf c}}}
,\]

\begin{dmath*}
U^{\zeta}_{\rho \kappa}
=
{\frac {{ \sf d}}{{\sf c}^{2}
\overline{ \sf c}}} \, \frac{\LL \left( \Lbk  \right) }{\Lbk }
-
{\frac { {\sf e}\overline{\sf d}
}{{\sf c} \overline{ \sf c}
 ^{2}}}\, \frac{\overline{\mathcal{L}_{1}} (\overline{k}) }{\mathcal{L}_{1} (\overline{k}) }
 +
 {\frac { {\sf e}\overline{\sf e}
\overline{\sf b}}{ 
\overline{ \sf c}  ^{3}{\sf c}}}\, \frac{\overline{\mathcal{L}_{1}}
 (k) }{\mathcal{L}_{1} (\overline{k}) }
+
{\frac { {\sf e}
\overline{\sf b}}{{\sf c}^{2}  \overline{ \sf c}^2}}\, \frac{\overline{\mathcal{L}_{1}} \left( \overline{\mathcal{L}_{1}} (k) 
 \right) }{\overline{\mathcal{L}_{1}}
 (k) }
 -
 {\frac { {\sf e} }{{\sf c}^{2}\overline{ \sf c}} \, \frac{\mathcal{T}\left( \overline{\mathcal{L}_{1}} \left( k
 \right)  \right) }{\overline{\mathcal{L}_{1}} \left( k \right) }
}
+
\frac{\ee^2}{\cc^3} \, \frac{\Tk}{\Lbk} 
-
{\frac {
{\sf e}^{2}\overline{\sf b}}{\overline{ \sf c}{\sf c}^{3}}}
+
i\, {\frac {{ \sf d}\overline{\sf b}}{{\sf c}
^{2}  \overline{ \sf c}^2}}
+
{\frac {\sf d}{{\sf c}^{2}
\overline{ \sf c}}} \, P
,\end{dmath*}

\begin{multline*}
U^{\zeta}_{\rho \zeta}
=
{\frac {\overline{\sf d}}{
  \overline{ \sf c}^2}} \, \frac{\overline{\mathcal{L}_{1}} (\overline{k}) }{\mathcal{L}_{1} (\overline{k}) }
  -
  {
\frac {\overline{\sf e} \overline{\sf b}}{  \overline{ \sf c}^{3}}}\, \frac{\overline{\mathcal{L}_{1}} (\overline{k}) }{\mathcal{L}_{1}( \overline{k}) }
 -
 {\frac {\overline{\sf b}}{{\sf c}  \overline{ \sf c}^2
}} \, \frac{\overline{\mathcal{L}_{1}} \left( \overline{\mathcal{L}_{1}}
 (k)  \right) }{\overline{\mathcal{L}_{1}} (k) }
-
{\frac {{\sf b}}{{\sf c}^{2}\overline{ \sf c}}} \frac{\mathcal{L}_{1} \left( \overline{\mathcal{L}_{1}}
 (k)  \right) }{\overline{\mathcal{L}_{1}} \left( k
 \right) }
\\ +
 {\frac{1 }{{\sf c}\overline{ \sf c}}} \, \frac{\T \left( \Lbk \right) }{\Lbk}
-
\frac{\ee}{\cc^2} \, \frac{\Tk}{\Lbk}
 +
 {\frac { {\sf e}
\overline{\sf b}}{{\sf c}^{2}\overline{ \sf c}}}
+
{\frac {{\sf b}{\sf e}}{{\sf c}
^{3}}}\, \frac{\mathcal{L}_{1} (k)}{\overline{\mathcal{L}_{1}} (k) }
-
{\frac {{ \sf d}}{{\sf c}
^{2}}}\, \frac{\mathcal{L}_{1} (k)}{\overline{\mathcal{L}_{1}} (k) }
,\end{multline*}

\begin{multline*}
U^{\zeta}_{\rho \overline{\kappa}}
=
2\,{\frac {\overline{\sf e}{ \sf d}}{
  \overline{ \sf c}^{3}}} \, \frac{\overline{\mathcal{L}_{1}} (\overline{k}) }{\mathcal{L}_{1} (\overline{k}) }
  -
  {\frac { {\sf e}\overline{\sf e} {\sf b}}{
  \overline{ \sf c}^{3}{\sf c}}} \,\frac{\overline{\mathcal{L}_{1}} (\overline{k}) }{\mathcal{L}_{1} (\overline{k}) }
 +
  {\frac {{ \sf d}}{{\sf c}
  \overline{ \sf c}^2}} \, \frac{\overline{\mathcal{L}_{1}} \left( \overline{\mathcal{L}_{1}} (k)  \right) }{\overline{\mathcal{L}_{1}} (k) } 
 -
  {\frac { {\sf e}{\sf b}}{
{\sf c}^{2}  \overline{ \sf c}^2 
}} \, \frac{\overline{\mathcal{L}_{1}} \left( \overline{\mathcal{L}_{1}} (k)  \right) }{\overline{\mathcal{L}_{1}} (k)}
\\ -
{\frac { {\sf e}{\sf d}}{{\sf c}^{2}\overline{ \sf c}}}
+ 
{\frac {{\sf e}^{2}{\sf b}}{\overline{ \sf c}{\sf c}
^{3}}}
-
i\, {\frac {{\sf d}{\sf b}}{{\sf c}^{2}  \overline{ \sf c}^2}}
+
{\frac 
{{ \sf d}}{{\sf c}  \overline{ \sf c}^2}} \, \overline{P}
,\end{multline*}

\[
U^{\zeta}_{\rho \overline{\zeta}}
=
-2\,{\frac { {\sf d} }{ 
\overline{ \sf c}^{2}}} \, \frac{\overline{\mathcal{L}_{1}} (\overline{k}) }{\mathcal{L}_{1} (\overline{k}) }
+
{\frac {{\sf e} {\sf b}
}{{\sf c } \overline{ \sf c}^{2}}} \, \frac{\overline{\mathcal{L}_{1}} (\overline{k}) }{\mathcal{L}_{1} (\overline{k}) }
,\]

\[
U^{\zeta}_{\kappa \zeta}
=
{\frac {1 }{\sf c
}} \, \frac{\mathcal{L}_{1} \left( \overline{\mathcal{L}_{1}} (k)  \right)}{\overline{\mathcal{L}_{1}} (k) }
-
{\frac { {\sf e}\overline{ \sf c}}{{\sf c}^{2}}} \, \frac{\mathcal{L}_{1} \left( k
 \right) }{\overline{\mathcal{L}_{1}} (k) }
,\]

\[
U^{\zeta}_{\kappa \overline{\kappa}}
=
{\frac { {\sf e}\overline{\sf e} }{ \overline{ \sf c}^2
 }} \,\frac{\overline{\mathcal{L}_{1}} (\overline{k})}{\mathcal{L}_{1} (\overline{k}) }
  +
{\frac { {\sf e}}{
\overline{ \sf c}{\sf c}}}\,\frac{\overline{\mathcal{L}_{1}} \left( \overline{\mathcal{L}_{1}} (k)  \right) }{\overline{\mathcal{L}_{1}} (k) }
-
{\frac {{\sf e}^{2}}{{\sf c}^{2}
}}
+
i\, {\frac {\sf d}{{\sf c}\overline{ \sf c}}}
,\]

\[
U^{\zeta}_{\kappa \overline{\zeta}}
=
- {\frac { {\sf e}}{\overline{ \sf c}}}\,\frac{\overline{\mathcal{L}_{1}} (\overline{k}) }{\mathcal{L}_{1}
 (\overline{k}) }
,\]

\[
U^{\zeta}_{\zeta \overline{\kappa}}
=
- {\frac {\overline{\sf e}{\sf c}}{
  \overline{ \sf c}^2}} \,\frac{\overline{\mathcal{L}_{1}} (\overline{k}) }{\mathcal{L}_{1} (\overline{k}) }
-
{\frac {1}{
\overline{ \sf c}}}\,\frac{\overline{\mathcal{L}_{1}} \left( \overline{\mathcal{L}_{1}} (k)  \right) }{\overline{\mathcal{L}_{1}} (k) }
+
{\frac {\sf e}{\sf c}}
,\]

\[
U^{\zeta}_{\zeta \overline{\zeta}}
=
{\frac {{\sf c}}{\overline{ \sf c}}}\, \frac{\overline{\mathcal{L}_{1}} (\overline{k}) }{\mathcal{L}_{1}
 (\overline{k}) }
.\]
\subsection{Normalization of the group parameter $\bb$}
We can now perform the absorption step. As for the first loop, we introduce the modified Maurer-Cartan forms $\tilde{\beta}^i$ which differ from the $\beta^i$ by 
a linear combination of the $1$-forms $\rho$, $\kappa$, $\zeta$, $\overline{\kappa}$, 
$\overline{\zeta}$, i.e. that is:
\begin{equation*}
\tilde{\beta}^i= \beta^i - y_{\rho}^i \, \rho \, -  y_{\kappa}^i \, \kappa - y_{\zeta}^i \, \zeta \, - \, y_{\overline{\kappa}}^i \, \overline{\kappa} \, - \, 
y_{\overline{\zeta}}^i \, \overline{\zeta}.
\end{equation*}
The structure equations rewrite:
\begin{dgroup*}
\begin{dmath*}
d \rho = \tilde{\beta}^1 \wedge \rho + \ov{\tilde{\beta}^1} \wedge \rho \\  +
 \left( U_{\rho \kappa}^{\rho}  - y_{\kappa}^1 - \ov{y}^1_{\ov{\kappa}} \right) \rho \wedge \kappa + 
\left( U_{\rho \zeta}^{\rho}  - y_{\zeta}^1 - \ov{y}^1_{\ov{\zeta}} \right) \rho \wedge \zeta +
 \left( U_{\rho \ov{\kappa}}^{\rho}  - y_{\ov{\kappa}}^1 - \ov{y}^1_{\kappa} \right) \rho \wedge \ov{\kappa} +
\left( U_{\rho \zeta}^{\rho}  - y_{\ov{\zeta}}^1 - \ov{y}^1_{\zeta} \right) \rho \wedge \zeta + i\, \kappa \wedge \ov{\kappa},
\end{dmath*}
\begin{dmath*}
d \kappa = \tilde{\beta}^1 \wedge \kappa + \tilde{\beta}^2 \wedge \rho \\
 + \left( U_{\rho \kappa}^{\kappa} + y^1_{\rho} - y^2_{\kappa} \right) \rho \wedge \kappa +  
\left( U_{\rho \zeta}^{\kappa}  - y^2_{\zeta} \right) \rho \wedge \zeta +  \left( U_{\rho \ov{\kappa}}^{\kappa}  - y^2_{\ov{\kappa}} \right) \rho \wedge \ov{\kappa} 
\\+
\left( U_{\rho \ov{\zeta}}^{\kappa}  - y^2_{\ov{\zeta}} \right) \rho \wedge \ov{\zeta} +
\left( U_{\kappa \zeta}^{\kappa}  - y^1_{\zeta} \right) \kappa \wedge \zeta +  
\left( U_{\kappa \ov{\kappa}}^{\kappa}  - y^1_{\ov{\kappa}} \right) \kappa \wedge \ov{\kappa} - y^1_{\ov{\zeta}} \, \kappa \wedge \ov{\zeta} + 
\zeta \wedge \ov{\kappa}
,\end{dmath*}
\begin{dmath*}
d \zeta = \tilde{\beta}^{3} \wedge \rho + \tilde{\beta}^{4} \wedge \kappa + \tilde{\beta}^{1} \wedge \zeta - \overline{\tilde{\beta}^{1}} \wedge \zeta \\
+ \left( U_{\rho \kappa}^{\zeta} - y^3_{\kappa} + y^4_{\rho} \right) \rho \wedge \kappa +  \left( U_{\rho \zeta}^{\zeta} - 
y^3_{\zeta} + y^1_{\rho} - \ov{y}^1_{\rho} \right) \rho \wedge \zeta +  \left( U_{\rho \ov{\kappa}}^{\zeta} - y^3_{\ov{\kappa}} \right) 
\left.\rho \wedge \ov{\kappa}\right.  +   
\left( U_{\kappa \zeta}^{\zeta} - y^4_{\zeta} + y^1_{\kappa}-\ov{y}^1_{\kappa} \right) \kappa \wedge \zeta +  \left( U_{\kappa \ov{ \kappa}}^{\zeta} -
 y^4_{\ov{\kappa}} \right) \kappa \wedge \ov{\kappa} + \left( U_{\kappa \ov{\zeta}}^{\zeta} - y^4_{\ov{\zeta}} \right) \kappa \wedge \ov{\zeta} +  
\left( U_{\zeta \ov{ \kappa}}^{\zeta} - y^1_{\ov{\kappa}} + \ov{y}^1_{\kappa} \right) \zeta \wedge \ov{\kappa} + 
\left( U_{\zeta \ov{\zeta}}^{\zeta} - y^1_{\ov{\zeta}} + \ov{y}^1_{\zeta} \right) \zeta \wedge \ov{\zeta}   
.\end{dmath*}
\end{dgroup*}

We get the following absorption equations:
\begin{alignat*}{3}
y_{\kappa}^1 + \ov{y}^1_{\ov{\kappa}} &= U_{\rho \kappa}^{\rho},  
& \qquad \qquad  y_{\zeta}^1 + \ov{y}^1_{\ov{\zeta}} &=  U_{\rho \zeta}^{\rho},
& \qquad \qquad   y_{\ov{\kappa}}^1 + \ov{y}^1_{\kappa} &=  U_{\rho \ov{\kappa}}^{\rho}, \\
y_{\ov{\zeta}}^1 + \ov{y}^1_{\zeta} & = U_{\rho \zeta}^{\rho}, 
& \qquad \qquad - y^1_{\rho} + y^2_{\kappa}& = U_{\rho \kappa}^{\kappa},  & 
\qquad \qquad   y^2_{\zeta} & =  U_{\rho \zeta}^{\kappa}, \\
 y^2_{\ov{\kappa}} &=U_{\rho \ov{\kappa}}^{\kappa},  & \qquad \qquad  
 y^2_{\ov{\zeta}}& = U_{\rho \ov{\zeta}}^{\kappa}, & \qquad \qquad 
 y^1_{\zeta} &= U_{\kappa \zeta}^{\kappa},  \\  
 y^1_{\ov{\kappa}} & =  U_{\kappa \ov{\kappa}}^{\kappa}, & \qquad \qquad
 y^1_{\ov{\zeta}} &=0, & \qquad \qquad
 y^3_{\kappa} - y^4_{\rho} & = U_{\rho \kappa}^{\zeta},  \\ 
y^3_{\zeta} - y^1_{\rho} + \ov{y}^1_{\rho} &= U_{\rho \zeta}^{\zeta}, & \qquad \qquad 
 y^3_{\ov{\kappa}} & = U_{\rho \ov{\kappa}}^{\zeta}, &\qquad \qquad   
 y^4_{\zeta} - y^1_{\kappa} + \ov{y}^1_{\kappa} &= U_{\kappa \zeta}^{\zeta}, \\
 y^4_{\ov{\kappa}} &= U_{\kappa \ov{ \kappa}}^{\zeta}, &\qquad \qquad 
 y^4_{\ov{\zeta}} &= U_{\kappa \ov{\zeta}}^{\zeta}, &\qquad \qquad   
y^1_{\ov{\kappa}} - \ov{y}^1_{\kappa} &= U_{\zeta \ov{ \kappa}}^{\zeta}, \\ 
y^1_{\ov{\zeta}} - \ov{y}^1_{\zeta} &= U_{\zeta \ov{\zeta}}^{\zeta}.    
\end{alignat*} 
Eliminating the $y^{\smallbullet}_{\smallbullet}$ among these equations leads to the following relations between the torsion coefficients:
\begin{align*}
U_{\rho \ov{\kappa}}^{\rho} & = \ov{U_{\rho \kappa}^{\rho}}, \\
U_{\rho \ov{\zeta}}^{\rho} & = \ov{U_{\rho \zeta}^{\rho}}, \\
U_{\rho \zeta}^{\rho}&  = U_{\kappa \zeta}^{\kappa},\\
U_{\zeta \ov{ \zeta}}^{\zeta}&  = - U_{\rho\ov{ \zeta}}^{\rho},   \\
2 \, U_{\kappa \ov{\kappa}}^{\kappa} & =  U_{\zeta \ov{\kappa}}^{\zeta} + U_{\rho \ov{\kappa}}^{\rho}. 
\end{align*}
We verify easily that the first four equations do not depend on the group coefficients and are already satisfied. However, the last one does depend on the group 
coefficients. It gives us the normalization of $\bb$ as it rewrites:
\begin{equation*}
\label{eq:nb}
\bb = - i \, \cb \ee + i \, \frac{\cc}{3} \left( \frac{\ov{\mathcal{L}_1}\left( \Lbk \right)}{\Lbk} - \ov{P} \right)
.\end{equation*} 
We now look at the new relation between the coframe $(\rho_0, \kappa_0 ,\hat{ \zeta_0} , \ov{\kappa_0}, \ov{\hat{\zeta}_0})$ and the lifted coframe 
$(\rho, \kappa, \zeta, \ov{\kappa}, \ov{\zeta})$, when one takes into account the normalization (\ref{eq:nb}).
Indded we have:
\begin{align*}
\rho & = \cc \cb \,\rho_0 \\
\kappa& = -i \, \ee \cb \, \rho_0 +  \cc \left( \kappa_0 + \frac{i}{3} \left( \frac{\ov{\mathcal{L}_1}\left( \Lbk \right)}{\Lbk} - \ov{P} \right) \rho_0 \right)\\
\zeta &= \dd \, \rho_0 + \ee \, \kappa_0 + \frac{\cc}{\cb} \, \hat{\zeta}_0.
\end{align*}
As in the first loop of the method, we modify the base coframe to get an interpretation of these equations as a $G$-structure.
Let us introduce: 
\begin{equation*}
\hat{\kappa}_0 : = \kappa_0 +  \frac{i}{3} \left( \frac{\ov{\mathcal{L}_1}\left( \Lbk \right)}{\Lbk} - \ov{P} \right) \rho_0.
\end{equation*}
The first two equations become 
\begin{equation*}
\rho = \cc \cb \,\rho_0 \quad \text{and} \quad  \kappa = -i \, \ee \cb \, \rho_0 + \cc \, \hat{\kappa}_0,
\end{equation*}
while the third one rewrites:
\begin{equation*}
\zeta = \left[ \dd - i \, \frac{\ee}{3} \left( \frac{\ov{\mathcal{L}_1}\left( \Lbk \right)}{\Lbk} - \ov{P} \right) \right]  \rho_0 + \ee \, \hat{\kappa}_0 + 
\frac{\cc}{\cb} \, \hat{\zeta}_0.
\end{equation*}
Let us introduce the new group parameter $\dd' := \dd -  i \, \frac{\ee}{3} \left( \frac{\ov{\mathcal{L}_1}\left( \Lbk \right)}{\Lbk} - \ov{P} \right)$.
We note that $\dd'$ describes $\mathbb{C}$ when $\dd$ describes $\mathbb{C}$.
We have thus reduced the problem to an equivalence of $G_3$-structure, 
described by the coframe 
$(\rho, \hat{\kappa}, \hat{\zeta}, \hat{\kappa}, \ov{\hat{\zeta}})$ and the relations:
\begin{equation*}
\begin{pmatrix}
\rho  \\
\kappa \\
\zeta \\
\ov{\kappa} \\
\ov{\zeta} \\
\end{pmatrix}
=
\begin{pmatrix}
\cc \cb & 0 & 0 & 0 & 0 \\
- i\, \ee \cb & \cc & 0 & 0 &0 \\
\dd^' & \ee & \frac{\cc}{\cb} & 0 & 0 \\
i\, \eb \cc & 0 & 0 & \cb & 0 \\
\db^' & 0 & 0 & \eb & \frac{\cb}{\cc} 
\end{pmatrix}
\cdot
\begin{pmatrix}
\rho_0 \\
\hat{\kappa}_0 \\
\hat{\zeta}_0 \\
\ov{\hat{\kappa}_0} \\
\ov{\hat{\zeta}_0}
\end{pmatrix}
.\end{equation*}
To simplify the notations, we simply drop the $'$ and write $\dd$ instead of  $\dd'$ in the sequel.
$G_3$ is the matrix Lie group whose elements are of the form 
$$
{\sf g} = \begin{pmatrix}
\cc \cb & 0 & 0 & 0 & 0 \\
- i\, \ee \cb & \cc & 0 & 0 &0 \\
\dd & \ee & \frac{\cc}{\cb} & 0 & 0 \\
i\, \eb \cc & 0 & 0 & \cb & 0 \\
\db & 0 & 0 & \eb & \frac{\cb}{\cc} 
\end{pmatrix}
.$$
It is a six dimensional real Lie group. We compute its Maurer Cartan forms with the usual formula
\[
d {\sf g} \cdot {\sf g}^{-1}
=
\begin{pmatrix}
\gamma^1 + \ov{\gamma}^1 & 0 & 0 & 0 &0 \\
\gamma^2 & \gamma^1 & 0 & 0 & 0 \\
\gamma^3 & i\, \gamma^2 & \gamma^1 - \ov{\gamma}^1 & 0 & 0 \\
\ov{\gamma}^2 & 0 & 0 & \ov{\gamma}^1 & 0 \\
- \gamma^3 & 0 & 0 &- i\, \ov{\gamma}^2 & - \gamma^1 + \ov{\gamma}^1 \\
\end{pmatrix}
\]

where 
\begin{equation*}
\gamma^1 := \frac{d \cc}{\cc},
\end{equation*}
\begin{equation*}
\gamma^2 := i \, \ee \frac{d \cc}{\cc^2} - i\, \frac{\ee \, d \cb}{\cc \cb} - i\, \frac{d \ee}{\cc}
\end{equation*}
and 
\begin{equation*}
\gamma^3 := \left(\frac{\dd \cc + i\, \ee^2 \cb }{\cc^2 \cb}\right) \left(\frac{d \cb}{\cb} - \frac{ d \cc }{\cc} \right) + \frac{d \dd}{\cc \cb} + 
i \, \frac{\ee d \ee}{\cc^2}
.\end{equation*}

As a preliminary step before the third loop of absorption and normalization, we compute the structure equations for the coframe 
$(\rho_0, \hat{\kappa}_0 ,\hat{ \zeta_0} , \ov{\hat{\kappa}_0}, \ov{\hat{\zeta}_0})$.
From the formula :
{\small 
\begin{multline*}
d  \left( \frac{\ov{\mathcal{L}_1 }\left( \Lbk \right)}{\Lbk} - \ov{P} \right)  
= 
\left(- \mathcal{T}\left( \ov{P} \right) - \frac{\Lb \left( \Lbk \right) \mathcal{T}\left( \Lbk \right)}{\Lbk^2} + 
\frac{\mathcal{T} \left( \Lb \left( \Lbk \right) \right)}{\Lbk} \right) \rho_0 
\\ + 
\left(
\frac{\Lb \left( \Lbk \right) \LL \left( \Lbk \right) }{\Lbk^2} + \frac{ \LL\left( \Lb \left( \Lbk \right) \right) }{\Lbk} - \LL \left( \ov{P} \right) 
\right) \kappa_{0} 
\\ +
\left( \frac{\Lb \left( \Lbk \right) \mathcal{K}\left( \Lbk \right)}{\Lbk^2} - \mathcal{K} \left( \ov{P} \right) + 
\frac{\mathcal{K}\left( \Lb \left( \Lbk \right) \right)}{\Lbk} \right) \zeta_{0}
 \\ +
\left(-\frac{\Lb \left( \Lbk \right)^2}{\Lbk^2} + \frac{ \Lb \left( \Lb \left( \Lbk \right) \right) }{\Lbk} - \Lb \left( \ov{P} \right) 
\right) \ov{\kappa_{0}} 
 \\ +
\left( - \frac{\Lbkb \LL \left( \Lbk\right)}{\Lbk} + \Lbkb \ov{P}
 \right) \ov{\zeta_{0}}
,\end{multline*}
}
we get:
\begin{multline*}
d \rho_0 
=
\left( \frac{1}{3} \frac{\LL \left( \Lkb \right)}{\Lkb} + \frac{2}{3} P \right) \rho_0 \wedge \hat{\kappa}_0
-
\frac{\Lk}{\Lbk} \rho_0 \wedge \hat{\zeta}_0 
\\ +
\left(
\frac{1}{3} \frac{\Lb \left( \Lbk \right)}{\Lbk} + \frac{2}{3} \ov{P}
\right)
\rho_0 \wedge \ov{\hat{\kappa}_0}
-
\frac{\Lbkb}{\Lkb} \rho_0 \wedge \ov{\hat{\zeta}_0} 
+ i \, \hat{\kappa}_0 \wedge \ov{\hat{\kappa}_0},
\end{multline*}

\begin{dmath*}
d \hat{\kappa}_0 = \left( \frac{i}{9} \frac{\Lb \left( \Lbk \right) \LL \left( \Lkb \right)}{\Lbk \Lkb} + i \, \frac{2}{9} \frac{\Lb \left( \Lbk \right)}{\Lbk} P-
 \frac{i}{9} \frac{\LL \left( \Lkb \right)}{\Lkb} \ov{P} 
- i\, \frac{2}{9} P \ov{P} + \frac{i}{3} \LL \left(\ov{ P }\right)
 -
\frac{i}{3} \frac{\LL \left( \Lb \left( \Lbk \right) \right)}{\Lbk} +
 \frac{i}{3} \frac{\Lb \left( \Lbk \right) \LL \left( \Lbk \right)}{\Lbk^2} \right)  \left.  \rho_0 \wedge \hat{\kappa}_0 \right.
+
\left(-\frac{i}{3} \, \frac{\LL \left( \Lbk \right)}{\Lbk} + \frac{i}{3} \frac{\Lb \left( \Lbk \right) \mathcal{K}\left( \Lbk \right)}{\Lbk^3} -
\frac{i}{3} \frac{\mathcal{K} \left( \Lb \left( \Lbk \right)\right) }{\Lbk^2} - 
\frac{i}{3} \frac{\LL \left( \Lkb \right)}{\Lkb} - \frac{\Tk}{\Lbk} \right) \left. \rho_0 \wedge \hat{\zeta}_0 \right.
+
\left(
i \, \frac{4}{9} \frac{\Lb \left( \Lbk \right) ^2}{\Lbk^2} + \frac{i}{9} \frac{\Lb \left( \Lbk \right) \ov{P}}{\Lbk}
- i \, \frac{2}{9} \, \ov{P}^2 + i \, \frac{1}{3} \frac{\Lb \left( \Lb \left( \Lbk \right) \right)}{\Lbk}
\right) \left. \rho_0 \wedge \ov{\hat{\kappa}_0} \right.
-\frac{\Lk}{\Lbk}  \hat{\kappa}_0 \wedge \hat{\zeta}_0
 +
\left(\frac{1}{3} \, \ov{P} - \frac{1}{3} \frac{\Lb \left( \Lbk \right)}{\Lbk} \right)  \hat{\kappa}_0 \wedge \ov{\hat{\kappa}_0} 
 + \hat{\zeta}_0 \wedge 
\ov{\hat{\kappa}_0} 
,\end{dmath*}

and

{\small
\begin{dmath*}
d \hat{\zeta}_0 = 
\left( \frac{i}{3} \frac{\Lb \left( \Lbk \right) \LL \left( \Lkb \right)}{\Lbk \Lkb} - \frac{i}{3} \frac{\Lb \left( \Lbk \right)}{\Lbk} P - \frac{i}{3}
 \frac{\Lb \left( \Lbk \right) \LL \left( \Lbk \right)}{\Lbk^2} + \frac{i}{3} \frac{\LL \left( \Lbk \right)}{\Lbk} \ov{P} + 
\frac{\mathcal{T}\left(\Lbk\right)}{\Lbk} \right)
\left.\rho_0 \wedge \hat{\kappa}_0 \right. + \frac{\LL \left( \Lbk\right)}{\Lbk} \hat{\kappa}_0 \wedge \hat{\zeta}_0
 - \frac{\Lb \left( \Lbk \right)}{\Lbk} \,  \left. \hat{\zeta}_0 \wedge 
\ov{\hat{\kappa}_0} \right. + \frac{\Lbkb}{\Lkb} \, \hat{\zeta}_0 \wedge \ov{\hat{\zeta}_0}
.\end{dmath*}
}

\section{Absorption of torsion and normalization: third loop} \label{section:step3}
\subsection{Lifted structure equations}
We are now ready to perform the third loop of Cartan's method. We begin with the structure equations for the lifted coframe. We have:
\begin{dgroup*}
\begin{dmath*}
d \rho
=
\gamma^{1} \wedge \rho +  \overline{\gamma^{1}} \wedge \rho \\
+
V^{\rho}_{\rho \kappa} \, \rho \wedge \kappa
+
V^{\rho}_{\rho \zeta} \, \rho \wedge \zeta
+
V^{\rho}_{\rho \overline{\kappa}} \, \rho \wedge \overline{\kappa} 
+
V^{\rho}_{\rho \overline{\zeta}} \, \rho \wedge \overline{\zeta}
+
i \, \kappa \wedge \overline{\kappa}
,\end{dmath*}

\begin{dmath*}
d \kappa
=
\gamma^{1} \wedge \kappa + \gamma^{2} \wedge \rho \\
+
V^{\kappa}_{\rho \kappa} \, \rho \wedge \kappa
+
V^{\kappa}_{\rho \zeta} \, \rho \wedge \zeta
+
V^{\kappa}_{\rho \overline{\kappa}} \, \rho \wedge \overline{\kappa}
+
V^{\kappa}_{\rho \overline{\zeta}} \, \rho \wedge \overline{\zeta}
+
V^{\kappa}_{\kappa \zeta} \, \kappa \wedge \zeta
+
V^{\kappa}_{\kappa \overline{\kappa}} \, \kappa \wedge \overline{\kappa}
+
\zeta \wedge \overline{\kappa}
,\end{dmath*}
\begin{dmath*}
d \zeta
=
\gamma^{3} \wedge \rho + i \, \gamma^{2} \wedge \kappa + \gamma^{1} \wedge 
\zeta - \overline{\gamma^{1}} \wedge \zeta \\
+
V^{\zeta}_{\rho \kappa}  \, \rho \wedge \kappa
+
V^{\zeta}_{\rho \zeta} \, \rho \wedge \zeta
+
V^{\zeta}_{\rho \overline{\kappa}} \, \rho \wedge \overline{\kappa}
+
V^{\zeta}_{\rho \overline{\zeta}} \, \rho \wedge \overline{\zeta}
+
V^{\zeta}_{\kappa \zeta} \, \kappa \wedge \zeta
+
V^{\zeta}_{\kappa \overline{\kappa}} \, \kappa \wedge \overline{\kappa}
+
V^{\zeta}_{\kappa \overline{\zeta}} \, \kappa \wedge \overline{\zeta}
+
V^{\zeta}_{\zeta \overline{\kappa}} \, \zeta \wedge \overline{\kappa}
+
V^{\zeta}_{\zeta \overline{\zeta}} \, \zeta \wedge \overline{\zeta}
,\end{dmath*}
\end{dgroup*}

where
\begin{dmath*}
V^{\rho}_{\rho \kappa}
=
-
\frac {\overline {\sf e}} {\overline{\sf c}}
+
\frac {1 }{3{\sf c}} \, \frac{\mathcal{L}_{1} \left( \mathcal{L}_{1} \left( 
\overline{k} \right)  \right)} {\mathcal{L}_{1} \left( \overline{k}\right)}
+
\frac{2}{3} \, {\frac {P} {{\sf c}}} 
+
{\frac {{\sf e} \, \overline{\sf c} \, } {{{\sf c}}^{2} 
}} \, \frac{\mathcal{L}_{1}(k) }{\overline{\mathcal{L}_{1}} (k)}  
,\end{dmath*}

\begin{dmath*}
V^{\rho}_{\rho \zeta} 
=
-
{\frac {{\overline{\sf  c}}  }{{\sf c} } \frac{\mathcal{L}_{1} (k)}{\overline{\mathcal{L}_{1}} \left( 
k \right) }}
,\end{dmath*}

\begin{dmath*}
V^{\rho}_{\rho \overline{\kappa}} 
=
-{\frac {{\sf e}} {{\sf c}}}
+
\frac{1}{3{\ov{\sf c}}} \, {\frac {\overline{\mathcal{L}_{1}} \left( 
\overline{\mathcal{L}_{1}} (k)  \right) }{ \overline{\mathcal{L}_{1}}
 (k) }}
 +
 \frac{2}{3} \, \frac{\overline{P}}{{\overline{\sf c}}}
 +
 {\frac {{\overline{\sf e}} {\sf c} }{{{\overline{\sf c}
}}^{2}}} \, \frac{\overline{\mathcal{L}_{1}} (\overline{k})}{\mathcal{L}_{1} (\overline{k}) } 
,\end{dmath*}

\begin{dmath*}
V^{\rho}_{\rho \overline{\zeta}}
=
-\frac{\sf c}{\overline{\sf c}} 
  \, \frac{\overline{\mathcal{L}_{1}} (\overline{k})}{\mathcal{L}_{1} (\overline{k})}
,\end{dmath*}

\begin{dmath*}
V^{\kappa}_{\rho \kappa} = \frac{i}{3} \, \frac{\ee}{\cc^2} \frac{\KK \left( \Lb \left( \Lbk \right) \right)}{\Lbk^2} - \frac{i}{3} \, \frac{\ee}{\cc^2} \, 
\frac{\Lb \left( \Lbk \right) \KK \left( \Lbk \right)}{\Lbk^3}  - \frac{i}{3} \, \frac{\eb}{\cb^2} \, \frac{\Lb \left( \Lbk \right)}{\Lbk} + \frac{2}{9} \, 
\frac{i}{\cc \cb} \, \frac{\Lb \left( \Lbk \right) P}{\Lbk} - \frac{2i}{3} \, \frac{\ee}{\cc^2} \, P + \frac{i}{3} \, \frac{\eb}{\cb^2} \, \ov{P}
+ \frac{1}{3} \, \frac{i}{\cc \cb } \, \LL(\ov{P}) - \frac{2}{9} \, \frac{i}{\cc \cb} \, P \ov{P} - i \, \frac{\cb \ee^2}{\cc^3} \, \frac{\Lk}{\Lbk}
+ \frac{1}{9} \, \frac{i}{\cc \cb} \, \frac{\Lb\left( \Lbk \right) \LL \left( \Lkb \right)}{\Lbk \, \Lkb} - \frac{1}{9} \, \frac{i}{\cc \cb} \, 
\frac{\LL \left(\Lkb \right) \ov{P}}{\Lkb} + \frac{i}{3} \, \frac{\ee}{\cc^2} \, \frac{\LL \left( \Lbk \right)}{\Lbk} + \frac{1}{3} \, 
\frac{\ee}{\cc^2} \, \frac{\Tk}{\Lbk} -
 \frac{\dd}{\cc^2} \, \frac{\Lk}{\Lbk} + \frac{1}{3} \, \frac{i}{\cc \cb} \, \frac{\Lb \left( \Lbk \right) \LL \left( \Lbk \right)}{\Lbk^2} - \frac{1}{3} \, 
\frac{i}{\cc \cb} \, \frac{\LL \left( \Lb \left( \Lbk \right)\right)}{\Lbk}
\end{dmath*}

\begin{dmath*}
V^{\kappa}_{\rho \zeta} 
=
\frac{i}{3 \cc} \, \frac{\overline{\mathcal{L}_{1}} \left( \overline{\mathcal{L}_{1}} (k) 
\right) \mathcal{K} \left( \overline{\mathcal{L}_{1}} (k)  \right) }{  \left( \overline{\mathcal{L}_{1}}  (k)  \right) ^{3}}
-
\frac{i}{3\cc} \, \frac{\LL \left( \Lbk \right)}{\Lbk}
-
 \frac{i}{3{\sf c}} \, \frac {\mathcal{K}
 \left( \overline{\mathcal{L}_{1}} \left( \overline{\mathcal{L}_{1}} (k)  \right) 
 \right) }{ \left( \overline{\mathcal{L}_{1}} (k)  \right) ^{2}}
-
\frac{1}{3\cc} \, \frac{\Tk}{\Lbk}
+
i \, \frac{\overline{\sf e}}{\overline{\sf c}}
-
\frac{i}{3{\sf c}} \, \frac{\mathcal{L}_{1} \left( \mathcal{L}_{1} (\overline{k})  \right) }{\mathcal{L}_{1}
(\overline{k})}
,\end{dmath*}

\begin{dmath*}
V^{\kappa}_{\rho \overline{\kappa}} 
=
-  \frac{2i}{3} \, \frac {\sf e }{\overline{\sf c}\,{\sf c}} \, \frac{\overline{\mathcal{L}_{1}} \left( \overline{\mathcal{L}_{1}} \left( k
 \right)  \right) }{\overline{\mathcal{L}_{1}} (k)}
+
\frac{4i}{9} \, {\frac {\left( \overline{\mathcal{L}_{1}} \left( \overline{\mathcal{L}_{1}} \left( k
 \right)  \right)  \right) ^{2}}{{{\overline{\sf c}}}^{2} \left( \overline{\mathcal{L}_{1}}
 (k)  \right) ^{2}}}
 +
 \frac{i}{9 \overline{\sf c}^{2} } \, {\frac {\overline{P}
\overline{\mathcal{L}_{1}} \left( \overline{\mathcal{L}_{1}} (k)  \right) }{\overline{\mathcal{L}_{1}} (k) }}
-
\frac{i}{3} \, {\frac {\overline{P}{\sf 
e}}{{\overline{\sf c}} {\sf c}}}
-
\frac{2i}{9} \,{\frac {\overline{P}^
{2}}{{{\overline{\sf c}}}^{2}}}
+
\frac{i}{3} \, {\frac {\overline{\mathcal{L}_{1}} \left( \overline{P}
 \right) }{{{\overline{\sf c}}}^{2}}}
 -
 \frac{i}{3{{\overline{\sf c}}}^{2}} \, {\frac {\overline{\mathcal{L}_{1}} \left( 
\overline{\mathcal{L}_{1}} \left( \overline{\mathcal{L}_{1}} (k)  \right)  \right) }{
\overline{\mathcal{L}_{1}} (k) }}
-
{\frac {{\sf d}}{{\sf 
c} {\cb}}}
-
i \, \, \frac{{\sf e} \overline{\sf e} }{\overline{\sf c}^{2}}
\frac{\overline{\mathcal{L}_{1}} \left( 
\overline{k} \right)}{\mathcal{L}_{1} (\overline{k}) }
,\end{dmath*}

\begin{dmath*}
V^{\kappa}_{\rho \overline{\zeta}}
=
i \,  \frac {{\sf e} }{\overline{\sf c}} \, \frac{\overline{\mathcal{L}_{1}} (\overline{k}) }{\mathcal{L}_{1} (\overline{k})}
,\end{dmath*}

\begin{dmath*}
V^{\kappa}_{\kappa \zeta} 
=
- \frac{\overline{\sf c}}{\sf c} \, \frac{\mathcal{L}_{1} (k) }{\overline{\mathcal{L}_{1}} \left( 
k \right) }
,\end{dmath*}

\begin{dmath*}
V^{\kappa}_{\kappa \overline{\kappa}}
=
-\frac{1}{3 {\overline{\sf c}}} \, {\frac {\overline{\mathcal{L}_{1}} \left( \overline{\mathcal{L}_{1}} (k) 
 \right) }{\overline{\mathcal{L}_{1}} (k) }}
 +
 \frac{1}{3} \, {\frac {
\overline{P}}{{\overline{\sf c}}}}
,\end{dmath*}

\begin{dmath*}
V^{\zeta}_{\rho \kappa}
=
\frac{2i}{3} \, {\frac {{\sf e}{\overline{\sf e}}}{{\sf c}{{\overline{\sf c}}}^{2}} \, \frac{\overline{\mathcal{L}_{1}} \left( \overline{\mathcal{L}_{1}}
 (k)  \right)}{\overline{\mathcal{L}_{1}}
 (k) }}
+
\frac{i}{3} \, \frac {{\sf e}{\overline{\sf e}}}{{\sf c}\,{{\overline{\sf c}}}^{2}} \, \overline{P}
-
\frac{i}{3} \, \frac {{\sf e}} {{\sf c}^{2}{\overline{\sf c}}} \, \frac{\overline{P} \mathcal{L}_{1} \left( \overline{\mathcal{L}_{1}} \left( k
\right)  \right)}{\overline{\mathcal{L}_{1}} \left( k
\right) } 
 +
\frac{i}{3} \, \frac{\ee^2}{\cc^3} \, \frac{\LL \left( \Lbk \right)}{\Lbk}
+
\frac{ {\sf d}}{ {\sf c}^{2} \overline{\sf c}
} \, \frac{\mathcal{L}_{1} \left( \overline{\mathcal{L}_{1}} (k)  \right) }{\overline{\mathcal{L}_{1}} (k)}
-
\frac{{\sf e} \overline{\sf d}}{{\sf c} {{\overline{\sf c}}}^{2}} \, \frac{
\overline{\mathcal{L}_{1}} (\overline{k})}{\mathcal{L}_{1}
(\overline{k}) }
+
\frac{2i}{3} \, \frac{{\sf e}}{{{\sf c}}^{2}{\overline{\sf c}}} \,  \frac{\mathcal{L}_{1} \left( 
\overline{\mathcal{L}_{1}} (k)  \right) \overline{\mathcal{L}_{1}} \left( 
\overline{\mathcal{L}_{1}} (k)  \right)}{
\overline{\mathcal{L}_{1}} (k) ^{2}}
-
\frac{\sf e} { {\sf c}
^{2} {\overline{\sf c}} }  \, \frac{\mathcal{T} \left( \overline{\mathcal{L}_{1}} (k)  \right)} {\overline{\mathcal{L}_{1}} (k)}
+
\frac{i}{3} \, \frac{\ee}{\cc^2 \cb} \, \Lb(\ov{P})
+
\frac{5i}{9} \, \frac{\sf e}{{\sf c}^{2} \overline{\sf c}} \frac{P \overline{\mathcal{L}_{1}} \left( \overline{\mathcal{L}_{1}} (k) 
 \right) } {\overline{\mathcal{L}_{1}} (k)}
 -
\frac{i}{3} \, \frac{ \sf e} { {\sf c}^{2} \overline{\sf c} } \, \frac{ \mathcal{L}_{1} \left( \overline{\mathcal{L}_{1}} \left( \overline{\mathcal{L}_{1}}
 (k)  \right)  \right) }{
\overline{\mathcal{L}_{1}} (k) }
+
\frac{1}{3} \, \frac {\ee^2 }{\cc^3} \, \frac{\Tk}{\Lbk}
+
 \frac{i}{3} \, \frac{{\sf e}^{2}}{{\sf c}^{3}} \, \frac{\mathcal{L}_{1}
\left( \mathcal{L}_{1} (\overline{k})  \right)} {\mathcal{L}_{1}
(\overline{k}) }
-
{\frac {{\sf d}\,{\overline{\sf e}}}{{\sf c}\,
{{\overline{\sf c}}}^{2}}}
+
\frac{2}{3} \, {\frac {{\sf d}}{{{\sf c}}^{2}{\overline{\sf c}}}} \, P
-
\frac{i}{9} \, \frac{\sf e}{{\sf c}^{2} \overline{\sf c} } \, \frac{\overline{P}\mathcal{L}_{1} \left( \mathcal{L}_{1} \left( \overline{k}
 \right)  \right) } {\mathcal{L}_{1} \left( \overline{k}
 \right) }
 -
\frac{i}{3} \, \frac{ {\sf e}^{2}}{{\sf c}^{3}} \, \frac{\overline{\mathcal{L}_{1}} \left( 
\overline{\mathcal{L}_{1}} (k)  \right) \mathcal{K} \left( \overline{\mathcal{L}_{1}} \left( k
\right)  \right)}{ \left( \overline{\mathcal{L}_{1}} \left( k
\right)  \right) ^{3}}
+
 \frac{i}{3} \, \frac{{\sf e}^{2}} {{\sf c}^{3}} \frac{\mathcal{K} \left( 
\overline{\mathcal{L}_{1}} \left( \overline{\mathcal{L}_{1}} (k)  \right)  \right) }{
 \left( \overline{\mathcal{L}_{1}} (k)  \right) ^{2}}
-
\frac{2i}{9} \, \frac{\sf e}{ {\sf c}^{2} \overline{\sf c}}  \, \overline{P} P
-
\frac{2i}{9} \, \frac{\sf e} {{\sf c}^
{2} {\overline{\sf c}}} \,\frac{ \overline{\mathcal{L}_{1}} \left( \overline{\mathcal{L}_{1}} (k) 
 \right) \mathcal{L}_{1} \left( \mathcal{L}_{1} (\overline{k})  \right) } 
{\overline{\mathcal{L}_{1}} (k) \mathcal{L}_{1} (\overline{k}) }
+
i \, \frac {{\sf e} \overline{\sf e}^{2}} {\overline{\sf c}^{3}} \frac{\overline{\mathcal{L}_{1}} \left( 
\overline{k} \right) }{\mathcal{L}_{1} (\overline{k}) }
-
i \, \frac{{\sf e}^{2} \overline{\sf e}} {{\sf c}^{2}{\overline{\sf c}}}
+
\frac{1}{3} \, \frac{\sf d  }{{\sf c}^{2} \overline{\sf c}} \, \frac{\mathcal{L}_{1} \left( \mathcal{L}_{1} 
(\overline{k})  \right)}{\mathcal{L}_{1} (\overline{k}) }
,\end{dmath*}

\begin{dmath*}
V^{\zeta}_{\rho \zeta} 
=
- \frac{1}{3} \,
\frac{\ee}{\cc^2} \, \frac{\Tk}{\Lbk}
-
\frac{\overline{\sf d}}{\overline{\sf c}^{2}} \, \frac{\overline{\mathcal{L}_{1}} (\overline{k})}{\mathcal{L}_{1} (\overline{k})}
+
\frac{i}{3} \, \frac{1}{ {\sf c} \overline{\sf c}} \, {\frac {\overline{P}\mathcal{L}_{1}
 \left( \overline{\mathcal{L}_{1}} (k)  \right) }{\overline{\mathcal{L}_{1}} (k) }}
+
i \, {\frac {{\sf e}\,{\overline{\sf e}}}{{\sf c}
\,{\overline{\sf c}}}}
-
i \, \frac{\cb \ee^2} {\cc^3} \frac{\mathcal{L}_{1} (k)}{\overline{\mathcal{L}_{1}} (k) }
 -
 \frac{i}{3} \, \frac {{\sf e} } {{{\sf c}}^{2}} \frac {\mathcal{K} \left( \overline{\mathcal{L}_{1}} \left( 
\overline{\mathcal{L}_{1}} (k)  \right)  \right)}{
 \overline{\mathcal{L}_{1}} (k)^{2}} 
 - 
\frac{i}{3} \, \frac{1}{{\sf c}\,{\overline{\sf c}}} \, \frac{P
\overline{\mathcal{L}_{1}} \left( \overline{\mathcal{L}_{1}} (k)  \right)}{ \overline{\mathcal{L}_{1}} (k) }
-
i \, \frac{ \overline{\sf e} } { {\overline{\sf c}}^{2} } \, \frac{
\overline{\mathcal{L}_{1}} \left( \overline{\mathcal{L}_{1}} (k)  \right) }{
\overline{\mathcal{L}_{1}} (k) }
+
\frac{1}{ {\sf c} {\overline{\sf c}} } \frac{\mathcal{T} \left( 
\overline{\mathcal{L}_{1}} (k)  \right) }{
\overline{\mathcal{L}_{1}} (k) }
-
\frac{i}{3} \, \frac {\sf e}{{\sf c}^{2}} \, \frac{ \mathcal{L}_{1} \left( \mathcal{L}_{1}
 (\overline{k})  \right) }{\mathcal{L}_{1} \left( 
\overline{k} \right) }
-
\frac{i}{3} \, \frac{1}{{\sf c} {\overline{\sf c}}} \,  \frac {\mathcal{L}_{1} \left( \overline{\mathcal{L}_{1}} \left( k
 \right)  \right) \overline{\mathcal{L}_{1}} \left( \overline{\mathcal{L}_{1}} (k) 
 \right) }{ \overline{\mathcal{L}_{1}} (k) ^{2}}
 +
 \frac{i}{3} \, \frac{1}{ {\sf c} \overline{\sf c}} \, {\frac {\overline{\mathcal{L}_{1}} \left( \overline{\mathcal{L}_{1}} (k)  \right) \mathcal{L}_{1}
 \left( \mathcal{L}_{1} (\overline{k})  \right) }{
\overline{\mathcal{L}_{1}} (k) \mathcal{L}_{1} (\overline{k}) }}
+
\frac{i}{3} \, \frac{\sf e}{{\sf c}^
{2}}  \, \frac{\overline{\mathcal{L}_{1}} \left( \overline{\mathcal{L}_{1}} (k) 
 \right) \mathcal{K} \left( \overline{\mathcal{L}_{1}} (k)  \right) }{\overline{\mathcal{L}_{1}} (k)^{3}}
+
i \,\frac{2}{3} \,  \frac{
\sf e}{{\sf c}^
{2}} \frac{\mathcal{L}_{1} \left( \overline{\mathcal{L}_{1}} (k)  \right) }{\overline{\mathcal{L}_{1}} (k) }
-
i \, {\frac {{\sf c} {{\overline{\sf e}}}^{2}
}{{{\overline{\sf c}}}^{3}}} \, \frac{\overline{\mathcal{L}_{1}} (\overline{k})}{\mathcal{L}_{1} \left( 
\overline{k} \right)}
-
\frac{\sf d}{{\sf c}^{2}} \, \frac{\mathcal{L}_{1} (k) }{\overline{\mathcal{L}_{1}} (k) }
,\end{dmath*} 

\begin{dmath*}
V^{\zeta}_{\rho \overline{\kappa}} 
=
2 \, \frac{{\sf d} \overline{\sf e}} {{\overline{\sf c}}^{3}} \, \frac{\overline{\mathcal{L}_{1}} \left( \overline{k}
 \right) }{\mathcal{L}_{1} (\overline{k})}
 + 
i \, \frac{\overline{\sf e} {\sf e}^{2}} {
\overline{\sf c}^{2}{\sf c}} \, \frac{\overline{\mathcal{L}_{1}} (\overline{k}) }{\mathcal{L}_{1} (\overline{k})}
+
\frac{4}{3} \, \frac{
\sf d}{\overline{\sf c}^{2} {\sf c}} \, \frac{\overline{\mathcal{L}_{1}} \left( \overline{\mathcal{L}_{1}} (k)  \right) }
{\overline{\mathcal{L}_{1}} (k)}
+
\frac{2i}{3} \, {\frac {
{{\sf e}}^{2}}{{\overline{\sf c}}\,{{\sf c}}^{2}}} \, \frac{\overline{\mathcal{L}_{1}} \left( \overline{\mathcal{L}_{1}} (k) 
 \right) }{\overline{\mathcal{L}_{1}} (k) }
 +
 \frac{4i}{9} \, {\frac {{\sf e} }{{{\overline{\sf c}}}^{2}{\sf c}}} \frac{\left( \overline{\mathcal{L}_{1}}  \left( \overline{\mathcal{L}_{1}}
 (k)  \right) \right) ^{2}}{
 \overline{\mathcal{L}_{1}} (k) ^{2}}
 +
 \frac{i}{9} \frac{
\sf e}{\overline{\sf c}^{2} {\sf c}} \, \frac{\overline{P}\overline{\mathcal{L}_{1}} \left( \overline{\mathcal{L}_{1}} \left( k
 \right)  \right) }{\overline{\mathcal{L}_{1}} \left( k
 \right) }
 +
\frac{i}{3} \, {\frac {{{\sf e}}^{2}}{{\overline{\sf c}} {{\sf 
c}}^{2}}} \, \overline{P}
-
\frac{2i}{9} \, {\frac {{\sf e}}{
{{\overline{\sf c}}}^{2}{\sf c}}} \,  \overline{P}^{2}
+
\frac{i}{3} \, {\frac {{\sf e}}{{{\overline{\sf c}}}^{2}{\sf c}}} \, \overline{\mathcal{L}_{1}} \left( 
\overline{P} \right) 
-
\frac{i}{3} \, {\frac {{\sf e}} {{{\overline{\sf c}}}^{2}{\sf c}}} \frac{\overline{\mathcal{L}_{1}} \left( \overline{\mathcal{L}_{1}} 
\left( \overline{\mathcal{L}_{1}} \left( k
 \right)  \right)  \right) }{\overline{\mathcal{L}_{1}}
 (k) } 
 -
2 \, {\frac {{\sf e} {\sf d}}{{\overline{\sf c}} {{\sf c
}}^{2}}}
-
i \, {\frac {{{\sf e}}^{3}}{{{\sf c}}^{3}}}
+
\frac{2}{3} \, {\frac {{\sf 
d} }{{{\overline{\sf c}}}^{2}{\sf c}}} \,  \overline{P}
,\end{dmath*}

\begin{dmath*}
V^{\zeta}_{\rho \overline{\zeta}} 
=
-2 \, {\frac {{\sf d}}{{{
\overline{\sf c}}}^{2}}} \frac{\overline{\mathcal{L}_{1}} (\overline{k}) }{\mathcal{L}_{1} \left ( \overline{k} \right) } 
- 
i \, {\frac {{{\sf e}}^{2}
}{{\sf c}\,{\overline{\sf c}}}} \, \frac{\overline{\mathcal{L}_{1}} (\overline{k}) }{\mathcal{L}_{1}
 (\overline{k}) }
,\end{dmath*}

\begin{dmath*}
V^{\zeta}_{\kappa \zeta} 
=
\frac{1}{\sf c} \,  {\frac {\mathcal{L}_{1} \left( \overline{\mathcal{L}_{1}} (k)  \right) }{
\overline{\mathcal{L}_{1}} (k) }}
-
{\frac {{\sf e}\,{\overline{\sf c}}}{{{\sf c}}^{2}}} \, \frac{\mathcal{L}_{1} \left( 
k \right)}{\overline{\mathcal{L}_{1}} (k) }
,\end{dmath*}

\begin{dmath*}
V^{\zeta}_{\kappa \overline{\kappa}}
=
{\frac {{\sf e}{\overline{\sf e}}}
{{{\overline{\sf c}}}^{2}}} \, \frac{\overline{\mathcal{L}_{1}} (\overline{k})}{\mathcal{L}_{1} (\overline{k})}
+
\frac{2}{3} \, \frac {{\sf e} {\sf c
}}{\overline{\sf c}} \,  \frac{\overline{P}} {\overline{\mathcal{L}_{1}} (k) }
+
\frac{1}{3} \, {\frac{\sf e} {{\sf c}\,{\overline{\sf c}}}} \, \overline{P}
-
{\frac {{{\sf e}}^{2}}{{{
\sf c}}^{2}}} 
+
i \, {\frac {{\sf d}}{{\sf c}\,{\overline{\sf c}}}}
,\end{dmath*}

\begin{dmath*}
V^{\zeta}_{\kappa \overline{\zeta}} 
=
- \frac{\sf e}{{\overline{\sf c}}} \frac{ \overline{\mathcal{L}_{1}} (\overline{k}) }{
\mathcal{L}_{1} (\overline{k}) }
,\end{dmath*}

\begin{dmath*}
V^{\zeta}_{\zeta \overline{\kappa}}
=
-\frac{ \overline{\sf e} {\sf c}} { \overline{\sf c}^{2} } \, \frac{\overline{\mathcal{L}_{1}} (\overline{k}) 
}{\mathcal{L}_{1} (\overline{k}) }
-
\frac{1}{\ov{\sf c}} \, \frac {\overline{\mathcal{L}_{1}}
 \left( \overline{\mathcal{L}_{1}} (k)  \right) } {
\overline{\mathcal{L}_{1}} (k) }
+
{\frac {{\sf e}}{{\sf c}}}
,\end{dmath*}

\begin{dmath*}
V^{\zeta}_{\zeta \overline{\zeta}}
=
{\frac{\sf c}{\overline{\sf c}}} \, \frac{\overline{\mathcal{L}_{1}} (\overline{k}) }{\mathcal{L}_{1} (\overline{k})}
.\end{dmath*}

\subsection{Normalization of the group parameter $\dd$} 
As for the previous steps, we now start the absorption step.
We introduce:
\[
\tilde{\gamma}^i := \gamma^i -  z^i_{\rho} \,  \rho - z^i_{\kappa}  \,  \kappa - z^i_{\zeta} \,  \zeta - z^i_{\ov{\kappa}} \,  \ov{\kappa} - z^i_{\ov{\zeta}} \, \ov{\zeta}
.\]
The structure equations are modified accordingly:
\begin{dgroup*}
\begin{dmath*}
d \rho = \tilde{\gamma}^{1} \wedge \rho +  \overline{\tilde{\gamma}^{1}} \wedge \rho \\
+ \left( V_{\rho \kappa}^{\rho} - z^1_{\kappa}- \ov{z^1_{\ov{\kappa}}} \right) \left. \rho \wedge \kappa \right. + 
\left( V_{\rho \zeta}^{\rho} - z^1_{\zeta} - \ov{z^1_{\ov{\zeta}}} \right) \left. \rho \wedge \zeta \right.
+ \left( V_{\rho \ov{\kappa}}^{\rho} - z^1_{\ov{\kappa}} -\ov{z^1_{\kappa}} \right) \left. \rho \wedge \ov{\kappa} \right.
+ \left( V^{\rho}_{\rho \ov{\zeta}} - \ov{z^1_{\zeta}} - z^1_{\ov{\zeta}} \right) \left. \rho \wedge \ov{\zeta} \right.
+  i\, \kappa \wedge \ov{\kappa}
,\end{dmath*} 

\begin{dmath*}
d \kappa = \tilde{\gamma}^{1} \wedge \kappa + \tilde{\gamma}^{2} \wedge \rho \\
+ \left( V_{\rho \kappa}^{\kappa} - z^2_{\kappa} + 
z^{1}_{\rho} \right) \left. \rho \wedge \kappa \right. + \left( V_{\rho \zeta}^{\kappa}
 - z^2_{\zeta} \right) \left. \rho \wedge \zeta \right. 
+ \left( V^{\kappa}_{\rho \ov{\kappa}} - z^2_{\ov{\kappa}} \right) \left. \rho \wedge \ov{\kappa} \right.
+ \left( V^{\kappa}_{\rho \ov{\zeta}}- z^2_{\ov{\zeta}} \right) \left. \rho \wedge \ov{\zeta} \right. 
+ \left( V_{\kappa \zeta}^{\kappa} - z^1_{\zeta} \right) 
\left. \kappa \wedge \zeta \right. 
+ \left( V^{\kappa}_{\kappa \ov{\kappa}} 
- z^1_{\ov{\kappa}} \right) \left. \kappa \wedge \ov{\kappa} \right. 
+ \zeta \wedge \ov{\kappa} 
- z^1_{\ov{\zeta}} \left. \kappa \wedge \ov{\zeta} \right.
,\end{dmath*}
\end{dgroup*}
and 
\begin{dmath*}
d \zeta = \tilde{\gamma}^{3} \wedge \rho 
+ i \, \tilde{\gamma}^{2} \wedge \kappa + \tilde{\gamma}^{1} \wedge 
\zeta - \overline{\tilde{\gamma}^{1}} \wedge \zeta \\
+ \left( V^{\zeta}_{\rho \kappa} - z^3_{\kappa} + i \, z_{\rho}^2 \right) \left. \rho \wedge \kappa \right. 
+ \left( V^{\zeta}_{\rho \zeta} + z^1_{\rho} - z^3_{\zeta}- \ov{z^1_{\rho}} \right) \left. \rho \wedge \zeta \right.
+ \left( V^{\zeta}_{\rho \ov{\kappa}} - z^3_{\ov{\kappa}} \right) \left. \rho \wedge \ov{\kappa} \right. +
\left( V^{\zeta}_{\rho \ov{\zeta}} - z^3_{\ov{\zeta}} \right) \left. \rho \wedge \ov{\zeta} \right. +
\left( V^{\zeta}_{\kappa \zeta} - i \, z^2_{\zeta} + z^1_{\kappa} - \ov{z^1_{\ov{\kappa}}} \right)
 \left. \kappa \wedge \zeta \right. +
\left( V^{\zeta}_{\kappa \ov{\kappa}} - i \, z^2_{\ov{\kappa}} \right) \left. \kappa \wedge \ov{\kappa} \right.
+ \left( V^{\zeta}_{\kappa \ov{\zeta}} - i \, z^2_{\ov{\zeta}} \right) \left. \kappa \wedge \ov{\zeta} \right.
+ \left( V^{\zeta}_{\zeta \ov{\kappa}} 
- z^1_{\ov{\kappa}} + \ov{z^1_{\kappa}} \right) \left. \zeta \wedge \ov{\zeta} \right.
+ \left( V^{\zeta}_{\zeta \ov{\zeta}} - z^1_{\ov{\zeta}} + \ov{z^1_{\ov{\zeta}}} \right) 
\left.\zeta \wedge \ov{\zeta} \right.
.\end{dmath*}
We thus want to solve the system of linear equations:

\begin{alignat*}{3}
z^1_{\kappa} + \ov{z^1_{\ov{\kappa}}}  &= V^{\rho}_{\rho \kappa},  
& \qquad \qquad z^1_{\ov{\kappa}} + \ov{z^1_{\kappa}}  &= V^{\rho}_{\rho \ov{\kappa}}, 
& \qquad \qquad z^1_{\zeta} + \ov{z^1_{\ov{\zeta}}}  &= V^{\rho}_{\rho \zeta},  \\
\ov{z^1_{\zeta}} + z^1_{\ov{\zeta}} &= V^{\rho}_{\rho \ov{\zeta}}, 
& \qquad \qquad z^2_{\kappa} - z^1_{\rho} & = V^{\kappa}_{\rho \zeta}, 
& \qquad \qquad z^2_{\ov{\kappa}} &= V^{\kappa}_{\rho \ov{\kappa}}, \\
z^2_{\zeta}& = V^{\kappa}_{\rho \zeta}, 
& \qquad \qquad z^2_{\ov{\zeta}} &= V^{\kappa}_{\rho \ov{\zeta}}, 
& \qquad \qquad z^1_{\zeta}& = V^{\kappa}_{\kappa \zeta}, \\
z^1_{\ov{\zeta}} &= 0, 
& \qquad \qquad z^1_{\ov{\kappa}} &= V^{\kappa}_{\kappa \ov{\kappa}}, 
& \qquad \qquad z^3_{\kappa} - i \, z^2_{\rho} &= V^{\zeta}_{\rho \kappa}, \\
- z^1_{\rho} + \ov{z^1_{\rho}} + z^3_{\zeta} &= V^{\zeta}_{\rho \zeta}, 
& \qquad \qquad z^1_{\kappa} - \ov{z^1_{\ov{\kappa}}} - i\, z^2_{\zeta} &= - V^{\zeta}_{\kappa \zeta}, 
& \qquad \qquad i \, z^2_{\ov{\kappa}} &= V^{\zeta}_{\kappa \ov{\kappa}}, \\
z^3_{\ov{\kappa}}& = V^{\zeta}_{\rho \ov{\kappa}}, 
& \qquad \qquad z^3_{\ov{\zeta}} &= V^{\zeta}_{\rho \ov{\zeta}}, 
& \qquad \qquad i \, z^2_{\ov{\zeta}} &= V^{\zeta}_{\kappa \ov{\zeta}}, \\
z^1_{\ov{\kappa}} - \ov{z^1_{\kappa}} &= V^{\zeta}_{\zeta \ov{\kappa}},
& \qquad \qquad z^1_{\ov{\zeta}} - \ov{z^1_{\zeta}} &= V^{\zeta}_{\zeta \ov{\zeta}}.
\end{alignat*}

This is easily done as:
\begin{equation*}
\left\{
\begin{aligned}
z^1_{\kappa}  & = \ov{V^{\rho}_{\rho \ov{\kappa}}} - \ov{V^{\kappa}_{\kappa \ov{\kappa}}}, \\
z^1_{\ov{\kappa}} & = V^{\kappa}_{\kappa \ov{\kappa}}, \\
z^1_{\zeta} & = V^{\rho}_{\rho \zeta}, \\
z^1_{\ov{\zeta}} & = 0, \\
z^2_{\ov{\kappa}} & = V^{\kappa}_{\rho \ov{\kappa}}, \\
z^2_{\ov{\zeta}} & = V^{\kappa}_{\rho \ov{\zeta}}, \\
z^2_{\zeta} & = V^{\kappa}_{\rho \zeta}, \\
z^3_{\ov{\kappa}} & = V^{\zeta}_{\rho \ov{\kappa}}, \\
z^3_{\ov{\zeta}} & = V^{\zeta}_{\rho \ov{\zeta}}, \\
z^3_{\zeta} & = V^{\zeta}_{\rho \zeta} + z^1_{\rho} - z^1_{\rho}, \\
z^3_{\kappa} & = V^{\zeta}_{\rho \kappa} + i\, z^2_{\rho}, \\
z^2_{\kappa} & = V^{\kappa}_{\rho \zeta} + z^1_{\rho}, 
\end{aligned}
\right.
\end{equation*}
where $z^1_{\rho}$ and $z^2_{\rho}$ may be choosen freely. Eliminating the $z^{\smallbullet}_{\smallbullet}$
we get the following additional conditions on the $V^{\smallbullet}_{\smallbullet \smallbullet}$:
\begin{equation} \label{eq:31}
\left\{
\begin{aligned}
V^{\rho}_{\rho \ov{\kappa}} &= \ov{V^{\rho}_{\rho \kappa}}, \\
V^{\rho}_{\rho \ov{\zeta}}  &= \ov{V^{\rho}_{\rho \zeta}}, \\
V^{\rho}_{\rho \zeta} &= V^{\kappa}_{\kappa \zeta}, \\
i \, V^{\kappa}_{\rho \ov{\zeta}} &= V^{\zeta}_{\kappa \ov{\zeta}}, \\
V^{\rho}_{\rho \zeta} & = - \ov{V^{\zeta}_{\zeta \ov{\zeta}}}, \\
2 \, V^{\kappa}_{\kappa \ov{\kappa}} & = V^{\rho}_{\rho \ov{\kappa}} + V^{\zeta}_{\zeta \ov{\kappa}}.
\end{aligned}
\right.
\end{equation}
and
\begin{equation}
\left\{
\begin{aligned}
i \, V^{\kappa}_{\rho \ov{\kappa}} & = V^{\zeta}_{\kappa \ov{\kappa}},  \\
V^{\ov{\zeta}}_{\kappa \ov{\zeta}} +  V^{\zeta}_{\kappa \zeta} & =   i \, V^{\kappa}_{\rho \zeta}.
\end{aligned}
\right.
\end{equation}

We easily verify that the equations (\ref{eq:31}) are indeed satisfied. However the remaining two equations are not
and they provide two essential torsion coefficients, namely $i \, V^{\kappa}_{\rho \ov{\kappa}} - V^{\zeta}_{\kappa \ov{\kappa}}$ and 
$V^{\ov{\zeta}}_{\kappa \ov{\zeta}} +  V^{\zeta}_{\kappa \zeta} -   i \, V^{\kappa}_{\rho \zeta}$, which will give us at least one new normalization of the group coefficients.
Indeed we have 
\begin{multline*}
 i \, V^{\kappa}_{\rho \ov{\kappa}} - V^{\zeta}_{\kappa \ov{\kappa}}
=
- \frac{4}{9} \, \frac{1}{\cb^2} \, \frac{\Lb \left( \Lbk \right)^2}{ \Lbk^2} - \frac{1}{9} \, \frac{1}{\cb^2} \, \frac{\Lb \left( \Lbk \right) \ov{P}}{ \Lbk}
+ \frac{2}{9} \, \frac{\ov{P}^2}{\cb^2} \\ -   \frac{1}{3}\, \frac{\Lb \left( \ov{P} \right)}{\cb^2}  + 
\frac{1}{3} \, \frac{1}{\cb^2} \frac{\Lb \left( \Lb \left( \Lbk \right) \right)}{ \Lbk}
- 2i \, \frac{\dd}{\cc \cb} + \frac{\ee^2}{\cc^2}
.\end{multline*}
Setting this expression to $0$, we get the normalization of the parameter $\dd$:
\begin{multline*}
\dd = - i\, \frac{1}{2} \,  \frac{\ee^2 \cb}{ \cc} + i \, \frac{2}{9} \, \frac{\cc}{\cb} \, \frac{\Lb \left( \Lbk \right)^2}{\Lbk^2} + i \,
\frac{1}{18} \,   \frac{\cc}{\cb} \, \frac{\Lb \left(\Lbk \right) \ov{P}}{\Lbk} \\ - i \, \frac{1}{9} \, \frac{\cc}{\cb} \, \ov{P}^2  +  i \, \frac{1}{6} \, 
\frac{\cc}{\cb} \, \Lb \left( \ov{P} \right) - i \,  \frac{1}{6} \, \frac{\cc}{\cb} \, \frac{\Lb \left( \Lb \left( \Lbk \right) \right)}{\Lbk}.
\end{multline*}
The other equation gives the essential torsion coefficient:
\begin{dmath*}
\frac{1}{\cc}
 \left( \frac{2}{3} \, 
 \frac{\LL \left( \Lbk \right)}{\Lbk} + \frac{2}{3} \, \frac{\LL 
\left( \Lkb \right)}{\Lkb}
+ \frac{1}{3} \frac{\Lb \left( \Lbk \right) \KK \left( \Lbk \right)}{\Lbk^3} -
\frac{1}{3} \frac{ \KK \left( \Lb \left( \Lbk \right) \right)}{\Lbk^2} + \frac{i}{3} \, \frac{\Tk}{\Lbk}
\right)
\end{dmath*}.
In the sequel we define the functions $H$ and  $W$ on $M^{5}$ by:

\begin{multline*}
H:=   \frac{2}{9} \,  \frac{\Lb \left( \Lbk \right)^2}{\Lbk^2} + 
\frac{1}{18} \,   \frac{\Lb \left(\Lbk \right) \ov{P}}{\Lbk} \\  -  \frac{1}{9} \, \ov{P}^2  + \frac{1}{6} \, 
 \Lb \left( \ov{P} \right) -  \frac{1}{6} \,  \frac{\Lb \left( \Lb \left( \Lbk \right) \right)}{\Lbk}
\end{multline*}
and
\begin{multline}
\label{eq:W}
W:=  \frac{2}{3} \, 
 \frac{\LL \left( \Lbk \right)}{\Lbk} + \frac{2}{3} \, \frac{\LL 
\left( \Lkb \right)}{\Lkb} \\
+ \frac{1}{3} \frac{\Lb \left( \Lbk \right) \KK \left( \Lbk \right)}{\Lbk^3}  -
\frac{1}{3} \frac{ \KK \left( \Lb \left( \Lbk \right) \right)}{\Lbk^2} + \frac{i}{3} \, \frac{\Tk}{\Lbk}
.
\end{multline}

We do not use the normalization $\cc = W$ at the moment, because this is allowed only if $W$ does not vanish. We will deal with this discussion 
further during the fourth loop of the algorithm.
With these notations, we have
\begin{equation*}
 \dd = - \frac{i}{2} \frac{\ee^2 \cb}{\cc} + i \, \frac{\cc}{\cb} \, H.
\end{equation*}

As a result, the relations between the base coframe $(\rho_0 , \hat{\kappa}_0, \hat{\zeta}_0, \ov{\hat{\kappa}_0}, \ov{\hat{\zeta}_0})$ and the lifted coframe 
$(\rho, \kappa, \zeta, \ov{\kappa}, \ov{\zeta})$ take the form:
\begin{equation*}
\left\{
\begin{aligned}
& \rho = \cc \cb\,  \rho_0 \\
& \kappa = - i \, \ee \cb \, \rho_0 + \cb \,  \hat{\kappa}_0 \\
& \zeta = - i \, \frac{1}{2} \, \frac{\ee^2}{ \cb \cc} \, \rho_0 + \ee \, \hat{\kappa}_0 + \frac{\cc }{\cb} \left( \hat{\zeta}_0 + i \, H \, \rho_0 \right)
\end{aligned}
\right.
\end{equation*}
Here again we explicitly exhibit the new $G$-structure by letting 
\[
 \check{\zeta}_{0} := \hat{\zeta}_{0} + i \, H \, \rho_0 
.\]
With these notations, we have:
\begin{equation*}
\left\{
\begin{aligned}
& \rho = \cc \cb\,  \rho_0 \\
& \kappa = - i \, \ee \cb \, \rho_0 + \cb \,  \hat{\kappa}_0 \\
& \zeta = - i \, \frac{1}{2} \, \frac{\ee^2}{ \cb \cc} \, \rho_0 + \ee \, \hat{\kappa}_0 + \frac{\cc }{\cb} \, \check{\zeta}_0.
\end{aligned}
\right.
\end{equation*}
We have reduced the previous $G_3$-structure to a $G_4$-structure, where $G_4$ is the four dimensional matrix Lie group whose elements are of the form:
\begin{equation*}
\begin{pmatrix}
\cc \cb & 0 & 0 & 0 & 0 \\
-i \, \ee \cb & \cc & 0 & 0 & 0 \\
- \frac{i}{2} \, \frac{\ee^2 \cb}{\cc}  & \ee & \frac{\cc}{\cb} & 0 & 0 \\
i \,  \eb \cc & 0 & 0 &  \cb & 0  \\
 \frac{i}{2} \, \frac{\eb^2 \cc}{\cb}  &0 & 0 &  \eb & \frac{\cb}{\cc}  \\
\end{pmatrix}
\end{equation*}
The basis for the Maurer-Cartan forms of $G_4$ is provided by the four forms 
\begin{equation*}
\delta^1 := \frac{d \cc}{\cc} \quad, \quad \delta^2 :=  i \, \ee \frac{d \cc}{\cc^2} - i\, \frac{\ee \, d \cb}{\cc \cb} - i\, \frac{d \ee}{\cc} \quad ,
\quad \ov{\delta^1} \quad , \quad \ov{\delta^2}.
\end{equation*}

\section{Absorption of torsion and normalisation: fourth loop} \label{section:step4}
At this stage we could compute the structure equations enjoyed by the base coframe $(\rho_0, \hat{\kappa}_0, \check{\zeta}_0, \ov{\hat{\kappa}_0} , 
\ov{\check{\zeta}_0})$, but as this involves rather lenghty computations, we procceed slightly differently from here. We just substitute the parameter $\dd$ by its 
normalization in the set of structure equations at the third loop. We have to take into account the fact that $d \dd$ is modified accordingly. Indeed we have:
\begin{equation*}
d \dd = - i \ee \frac{\cb}{\cc} - \frac{i}{2} \, \frac{\ee^2 \cb}{\cc} \left( \frac{ d \cb}{\cb} - \frac{d \cc}{\cc} \right) + i \, H \, \frac{\cc}{\cb} \left(
\frac{d \cc}{\cc} - \frac{d \cb}{\cb} \right) + i \, \frac{\cc}{\cb} \, d H
\end{equation*}
The forms $\gamma^1$ and $\gamma^2$ are not modified as they do not involve terms in $d \dd$, but this is not the case for $\gamma^3$ which is transformed as:
\begin{align*}
\gamma^3  & = \frac{d \dd }{\cc \cb} + i \, \frac{\ee}{\cc^2} - \frac{\dd \, d \cc}{\cc^2 \cb^2} - i \, \ee^2 \frac{d \cc}{\cc^3} + \frac{\dd \, d \cb}{\cc \cb^2} + 
i \,\frac{\ee^2 \, d\cb}{\cb \cc^2} \\
          & = i \, \frac{d H}{\cb^2} 
\end{align*}
The expressions of $d \rho$ and $d \kappa$ are thus unchanged from the expressions given by the structure equations at the third step, 
except the fact that we shall replace $\dd$ by 
$ - \frac{i}{2} \frac{\ee^2 \cb}{\cc} + i \, \frac{\cc}{\cb} \, H$ 
in the expression of each torsion coefficient $V_{\smallbullet \smallbullet}^{\smallbullet}$ and the fact that the forms $\gamma^1$ and $\gamma^2$ shall be replaced by the forms
$\delta^1$ and $\delta^2$, that is:

\begin{dmath*}
d \rho
 =
\delta^1 \wedge \rho +  \overline{\delta^{1}} \wedge \rho \\
+
V^{\rho}_{\rho \kappa} \, \rho \wedge \kappa
+
V^{\rho}_{\rho \zeta} \, \rho \wedge \zeta 
+
V^{\rho}_{\rho \overline{\kappa}} \, \rho \wedge \overline{\kappa}
+
V^{\rho}_{\rho \overline{\zeta}} \, \rho \wedge \overline{\zeta}
+
i \, \kappa \wedge \overline{\kappa},
\end{dmath*}

and
\begin{dmath*}
d \kappa
 =
\delta^{1} \wedge \kappa + \delta^{2} \wedge \rho \\
+
V^{\kappa}_{\rho \kappa} \, \rho \wedge \kappa
+
V^{\kappa}_{\rho \zeta} \, \rho \wedge \zeta
+
V^{\kappa}_{\rho \overline{\kappa}} \, \left. \rho \wedge \overline{\kappa} \right.
+
V^{\kappa}_{\rho \overline{\zeta}} \,\left.  \rho \wedge \overline{\zeta} \right.
+
V^{\kappa}_{\kappa \zeta} \, \kappa \wedge \zeta
+
V^{\kappa}_{\kappa \overline{\kappa}} \, \kappa \wedge \overline{\kappa}
+
\zeta \wedge \overline{\kappa}.
\end{dmath*}

The computation of $d \zeta$ involves the expression of the form $\gamma^3$ and is therefore modified as 
\begin{dmath*}
d \zeta
=
i \, \frac{d H}{\cb^2} \wedge \rho + i \, \delta_{2} \wedge \kappa + \delta_{1} \wedge 
\zeta - \overline{\delta_{1}} \wedge \zeta \\
+
V^{\zeta}_{\rho \kappa}  \, \rho \wedge \kappa
+
V^{\zeta}_{\rho \zeta} \, \rho \wedge \zeta
+
V^{\zeta}_{\rho \overline{\kappa}} \, \rho \wedge \overline{\kappa} 
+
V^{\zeta}_{\rho \overline{\zeta}} \, \rho \wedge \overline{\zeta} 
+
V^{\zeta}_{\kappa \zeta} \, \kappa \wedge \zeta
+
V^{\zeta}_{\kappa \overline{\kappa}} \, \kappa \wedge \overline{\kappa}
+
V^{\zeta}_{\kappa \overline{\zeta}} \, \kappa \wedge \overline{\zeta} 
+
V^{\zeta}_{\zeta \overline{\kappa}} \, \zeta \wedge \overline{\kappa}
+
V^{\zeta}_{\zeta \overline{\zeta}} \, \zeta \wedge \overline{\zeta}
.\end{dmath*}
The term  $\frac{d H}{\cb^2} \wedge \rho$ involves torsion terms in $\rho \wedge \kappa$, $\rho \wedge \zeta$, $\rho \wedge \ov{\kappa}$ and 
$\rho \wedge \ov{\zeta}$, which only affect the expressions of the coefficients $V_{\rho \kappa}^{\zeta}$,  $V_{\rho \zeta}^{\zeta}$,  $V_{\rho \ov{\kappa}}^{\zeta}$ and 
 $V_{\rho \zeta}^{\zeta}$. If we write $W_{\rho \kappa}^{\zeta}$,  $W_{\rho \zeta}^{\zeta}$,  $W_{\rho \ov{\kappa}}^{\zeta}$ and 
 $W_{\rho \zeta}^{\zeta}$ for these modified torsion coefficients, we get 

\begin{dmath*}
d \zeta
=
 i \, \delta_{2} \wedge \kappa + \delta_{1} \wedge 
\zeta - \overline{\delta_{1}} \wedge \zeta \\
+
W^{\zeta}_{\rho \kappa}  \, \rho \wedge \kappa
+
W^{\zeta}_{\rho \zeta} \, \rho \wedge \zeta
+
W^{\zeta}_{\rho \overline{\kappa}} \, \rho \wedge \overline{\kappa} 
+
W^{\zeta}_{\rho \overline{\zeta}} \, \rho \wedge \overline{\zeta} 
+
V^{\zeta}_{\kappa \zeta} \, \left.  \kappa \wedge \zeta \right.
+
V^{\zeta}_{\kappa \overline{\kappa}} \, \kappa \wedge \overline{\kappa}
+
V^{\zeta}_{\kappa \overline{\zeta}} \, \kappa \wedge \overline{\zeta} 
+
V^{\zeta}_{\zeta \overline{\kappa}} \, \zeta \wedge \overline{\kappa}
+
V^{\zeta}_{\zeta \overline{\zeta}} \, \zeta \wedge \overline{\zeta}
.\end{dmath*}

Before computing the actual value of the coefficients $W^{\smallbullet}_{\smallbullet \smallbullet}$, we proceed with the absorption phase. We make the two substitutions 
\begin{equation*}
\begin{aligned}
& \delta^1 : = \tilde{\delta}^1 + w^1_{\rho} \, \rho + w^1_{\kappa} \, \kappa + w^1_{\zeta} \, \zeta + w^1_{\ov{\kappa}} \, \ov{\kappa} + w^1_{\ov{\zeta}} \, \ov{\zeta}, \\
& \delta^2 : =  \tilde{\delta}^2 + w^2_{\rho} \, \rho + w^2_{\kappa} \, \kappa + w^2_{\zeta} \, \zeta + w^2_{\ov{\kappa}} \, \ov{\kappa} + w^2_{\ov{\zeta}} \, \ov{\zeta} 
\end{aligned}
\end{equation*}
in the previous equations.
We get:
\begin{dmath*}
d \rho = \tilde{\delta}^1 \wedge \rho +  
\overline{\tilde{\delta}^{1}} \wedge \rho \\
+ \left( V_{\rho \kappa}^{\rho} - w^1_{\kappa}- \ov{w^1_{\ov{\kappa}}} \right) \left. \rho \wedge \kappa \right. + 
\left( V_{\rho \zeta}^{\rho} - w^1_{\zeta} - \ov{w^1_{\ov{\zeta}}} \right) \left. \rho \wedge \zeta \right.
+ \left( V_{\rho \ov{\kappa}} - w^1_{\ov{\kappa}} -\ov{w^1_{\kappa}} \right) \left. \rho \wedge \ov{\kappa} \right.
+ \left( V^\rho_{\rho \ov{\zeta}} - \ov{w^1_{\zeta}} - w^1_{\ov{\zeta}} \right) \left. \rho \wedge \ov{\zeta} \right.,
\end{dmath*} 

\begin{dmath*}
d \kappa = \tilde{\delta}^{1} \wedge \kappa + \tilde{\delta}^{2} \wedge \rho \\
+ \left( V_{\rho \kappa}^{\kappa} - w^2_{\kappa} + w^{1}_{\rho} \right) \left. \rho \wedge \kappa \right. + \left( V_{\rho \zeta}^{\kappa}
 - w^2_{\zeta} \right) \left. \rho \wedge \zeta \right. + \left( V^{\kappa}_{\rho \ov{\kappa}} - w^2_{\ov{\kappa}} \right) \left. \rho \wedge \ov{\kappa} \right.
+ \left( V_{\rho \ov{\zeta}}- w^2_{\ov{\zeta}} \right) \left. \rho \wedge \ov{\zeta} \right. + \left( V_{\kappa \zeta}^{\kappa} - w^1_{\zeta} \right) 
\left. \kappa \wedge \zeta \right. + \left( V^{\kappa}_{\kappa \ov{\kappa}} - w^1_{\ov{\kappa}} \right) \left. \kappa \wedge \ov{\kappa} \right. 
+ \zeta \wedge \ov{\kappa} - w^1_{\ov{\zeta}} \left. \kappa \wedge \ov{\zeta} \right.,
\end{dmath*}
and 
\begin{dmath*}
d \zeta = i \, \tilde{\delta_{2}} \wedge \kappa + \tilde{\delta_{1}} \wedge 
\zeta - \overline{\tilde{\delta_{1}}} \wedge \zeta \\
+\left( W^{\zeta}_{\rho \kappa}  + i \, w_{\rho}^2 \right) \left. \rho \wedge \kappa \right. 
+ \left( W^{\zeta}_{\rho \zeta} + w^1_{\rho} - \ov{w^1_{\rho}} \right) \left. \rho \wedge \zeta \right. 
+  W^{\zeta}_{\rho \ov{\kappa}} \left. \rho \wedge \ov{\kappa} \right. +
 W^{\zeta}_{\rho \ov{\zeta}}  \left. \rho \wedge \ov{\zeta} \right. +
\left( V^{\zeta}_{\kappa \ov{\kappa}} - i \, w^2_{\ov{\kappa}} \right) \left. \kappa \wedge \ov{\kappa} \right.
+ \left( V^{\zeta}_{\kappa \ov{\zeta}} - i \, w^2_{\ov{\zeta}} \right) \left. \kappa \wedge \ov{\zeta} \right.
+ \left( V^{\zeta}_{\zeta \ov{\kappa}} - w^1_{\ov{\kappa}} + \ov{w^1_{\kappa}} \right) \left. \zeta \wedge \ov{\zeta} \right.
.\end{dmath*}
From the last equation, we immediately see that $ W^{\zeta}_{\rho \ov{\kappa}}$ and $W^{\zeta}_{\rho \ov{\zeta}} $ are two new essential torsion coefficients.
We find the remaining ones by solving the set of equations:

\begin{alignat*}{3}
w^1_{\kappa} + \ov{w^1_{\ov{\kappa}}}  &= V^{\rho}_{\rho \kappa}, 
& \qquad \qquad w^1_{\ov{\kappa}} + \ov{w^1_{\kappa}}  &= V^{\rho}_{\rho \ov{\kappa}}, 
& \qquad \qquad w^1_{\zeta} + \ov{w^1_{\ov{\zeta}}}  &= V^{\rho}_{\rho \zeta},  \\
\ov{w^1_{\zeta}} + w^1_{\ov{\zeta}} &= V^{\rho}_{\rho \ov{\zeta}}, 
& \qquad \qquad w^2_{\kappa} - w^1_{\rho} & = V^{\kappa}_{\rho \kappa}, 
& \qquad \qquad w^2_{\ov{\kappa}} &= V^{\kappa}_{\rho \ov{\kappa}}, \\
w^2_{\zeta}& = V^{\kappa}_{\rho \zeta}, 
& \qquad \qquad w^2_{\ov{\zeta}} &= V^{\kappa}_{\rho \ov{\zeta}}, 
& \qquad \qquad w^1_{\zeta}& = V^{\kappa}_{\kappa \zeta}, \\
w^1_{\ov{\zeta}} &= 0, 
& \qquad \qquad w^1_{\ov{\kappa}} &= V^{\kappa}_{\kappa \ov{\kappa}}, 
& \qquad \qquad  - i \, w^2_{\rho} &= V^{\zeta}_{\rho \kappa}, \\
- w^1_{\rho} + \ov{w^1_{\rho}} &= V^{\zeta}_{\rho \zeta}, 
& \qquad \qquad w^1_{\kappa} - \ov{w^1_{\ov{\kappa}}} - i\, w^2_{\zeta} &= - V^{\zeta}_{\kappa \zeta}, 
& \qquad \qquad i \, w^2_{\ov{\kappa}} &= V^{\zeta}_{\kappa \ov{\kappa}}, \\ 
w^1_{\kappa} - \ov{w^1_{\kappa}} &= V^{\zeta}_{\zeta \ov{\kappa}},
& \qquad \qquad i \, w^2_{\ov{\zeta}} &= V^{\zeta}_{\kappa \ov{\zeta}},
& \qquad \qquad w^1_{\ov{\zeta}} - \ov{w^1_{\zeta}} &= V^{\zeta}_{\zeta \ov{\zeta}},
\end{alignat*}

which lead easily as before to:
\begin{equation} \label{absorption}
\left\{
\begin{aligned}
w^1_{\kappa}  & = \ov{V^{\rho}_{\rho \ov{\kappa}}}, \\
w^1_{\ov{\kappa}} & = V^{\kappa}_{\kappa \ov{\kappa}}, \\
w^1_{\zeta} & = V^{\rho}_{\rho \zeta}, \\
w^1_{\ov{\zeta}} & = 0, \\
w^2_{\ov{\kappa}} & = V^{\kappa}_{\rho \ov{\kappa}}, \\
w^2_{\ov{\zeta}} & = V^{\kappa}_{\rho \ov{\zeta}}, \\
w^2_{\zeta} & = V^{\kappa}_{\rho \zeta}, \\
w^2_{\kappa} & = V^{\kappa}_{\rho \kappa} + w^1_{\rho}, \\
w^2_{\rho}& = W^{\zeta}_{\rho \kappa} \\
-w^1_{\rho} + \ov{w^1_{\rho}} & = W_{\rho \zeta}^{\zeta}.
\end{aligned}
\right.
\end{equation}
Eliminating the $w^{\smallbullet}_{\smallbullet}$ from (\ref{absorption}), we get one additionnal condition on the $W^{\smallbullet}_{\smallbullet \smallbullet}$ which has not yet been checked, namely
that $W_{\rho \zeta}^{\zeta}$ shall be purely imaginary.
We now need to compute the two essential torsion coefficients $ W^{\zeta}_{\rho \ov{\kappa}}$ and $W^{\zeta}_{\rho \ov{\zeta}} $.
As they both involves the term $d H \wedge \rho$, we start with the computation of this term.
Standard differentiation with respect to base coframe $(\rho_0, \kappa_0, \zeta_0, \ov{\kappa_0}, \ov{\zeta_0})$ yields:
\begin{equation*}
d H = \mathcal{T}(H) \, \rho_0 + \LL(H) \, \kappa_0 + \mathcal{K}(H) \, \zeta_0 + \Lb(H) \, \ov{\kappa_{0}} + \ov{\mathcal{K}}(H) \,  \ov{\zeta_0}.
\end{equation*}
Taking the wedge product with $\rho$ and using the fact that 
\begin{equation*}
\kappa_0 \wedge \rho  = \hat{\kappa}_0 \wedge \rho 
\end{equation*}
and
\begin{equation*}
\zeta_0 \wedge \rho = \frac{\check{\zeta}_0}{\Lbk} \wedge \rho, 
\end{equation*}
which is easily seen from the definitions of $\hat{\kappa}_0$ and $\check{\zeta}_0$, we get:
\begin{equation*}
d H \wedge \rho  =  \left(\LL(H) \, \hat{\kappa}_0 + \frac{\mathcal{K}(H)}{\Lbk} \, \check{\zeta_0} +  \Lb(H) \, \ov{\hat{\kappa}_{0}} + 
\frac{\ov{\mathcal{K}}(H)}{\Lkb} \,  
\ov{\check{\zeta_0}} \right) \wedge \rho.
\end{equation*}
We now use the expressions of the $1$-forms  $\hat{\kappa}_0$ and $\check{\zeta}_0$ in terms of $\rho$, $\kappa$ and $\zeta$, which are deduced by the use of
(\ref{eq:change}),  that is:
\begin{equation*}
\left\{
\begin{aligned}
\hat{\kappa}_{0} & = i \, \frac{\ee}{\cc^2} \, \rho + \frac{1}{\cc} \, \kappa \\
\check{\zeta}_{0} & = - i \, \frac{1}{2} \frac{\ee^2 \cb}{\cc^3} \, \rho - \frac{\ee \cb}{\cc^2} \, \kappa + \frac{\cb}{\cc} \, \zeta.
\end{aligned}
\right.
\end{equation*}
As a result, we get:
\begin{multline*}
d H \wedge \rho =
\left( \frac{\ee \cb}{\cc^2}  \frac{\mathcal{K}(H)}{\Lbk}- \frac{\LL(H)}{\cc} \right) \left. \rho \wedge \kappa \right. - \frac{\cb}{\cc} \, 
\frac{\mathcal{K}(H)}{\Lbk} 
\left. \rho \wedge \zeta \right. \\
+ \left( \frac{\eb \cc}{\cb^2} \frac{\mathcal{K}(H)}{\Lbk}
- \frac{\Lb(H)}{\cb} \right) \left. \rho \wedge \ov{\kappa} \right.  - \frac{\cc}{\cb} \, \frac{\ov{\mathcal{K}}(H)}{\Lbk} \left. 
\rho \wedge \ov{\zeta} \right.
.\end{multline*}
Inserting this equation in the expression of $d\zeta$, we find that:
\begin{dmath*}
d \zeta
=
 i \, \delta_{2} \wedge \kappa + \delta_{1} \wedge 
\zeta - \overline{\delta_{1}} \wedge \zeta \\
 +
\left( V^{\zeta}_{\rho \kappa} + \frac{i}{\cc^2 \cb} \frac{\mathcal{K}(H)}{\Lbk} - \frac{ i }{\cc \cb^2} \, \LL (H) \right)  \left. \rho \wedge \kappa \right. 
 +
\left( V^{\zeta}_{\rho \zeta} - \frac{i}{\cc \cb} \, \frac{\mathcal{K}(H)}{\Lbk}  \right) \left.  \rho \wedge \zeta \right. 
+
\left( V^{\zeta}_{\rho \overline{\kappa}} + i \, \frac{\eb \cc}{\cb^4} \, \frac{\ov{\mathcal{K}}(H)}{\Lkb} - \frac{i}{\cb^3} \,  \Lb (H) \right) \left. \rho \wedge 
\overline{\kappa} 
\right.
+
\left(
V^{\zeta}_{\rho \overline{\zeta}} - i \, \frac{\cc}{\cb^3} \, \frac{\ov{\mathcal{K}}(H)}{\Lkb}
\right) \left. \rho \wedge \overline{\zeta} \right. 
+
V^{\zeta}_{\kappa \zeta} \, \left. \kappa \wedge \zeta \right. 
+
V^{\zeta}_{\kappa \overline{\kappa}} \, \left. \kappa \wedge \overline{\kappa} \right.
+
V^{\zeta}_{\kappa \overline{\zeta}} \, \left. \kappa \wedge \overline{\zeta} \right. 
+
V^{\zeta}_{\zeta \overline{\kappa}} \, \left. \zeta \wedge \overline{\kappa} \right.
+
V^{\zeta}_{\zeta \overline{\zeta}} \, \left. \zeta \wedge \overline{\zeta} \right.
.\end{dmath*}
We thus have 
\begin{equation}
\label{eq:int1}
W^{\zeta}_{\rho \ov{\zeta}} = V^{\zeta}_{\rho \ov{\zeta}}   - i \, \frac{\cc}{\cb^3} \, \frac{\ov{\mathcal{K}}(H)}{\Lkb}
\end{equation}
and
\begin{equation}
\label{eq:int1bis}
W^{\zeta}_{\rho \ov{\kappa}} =  V^{\zeta}_{\rho \overline{\kappa}} + i \, \frac{\eb \cc}{\cb^4} \,  \frac{\ov{\mathcal{K}}(H)}{\Lkb} - \frac{i}{\cb^3} \, \Lb (H). 
\end{equation}
We first compute $W^{\zeta}_{\rho \ov{\zeta}}$.
Performing the substitution $\dd = - \frac{i}{2} \, \frac{e^2 \cb}{cc} + i \, \frac{\cc}{\cb} \, H$ in $V^{\zeta}_{\rho \ov{\zeta}}$ gives
\begin{equation}
\label{eq:int2}
V^{\zeta}_{\rho \ov{\zeta}} = -2 i \, \frac{\cc}{\cb^3} \, \frac{\Lbkb}{\Lkb} \, H.
\end{equation}
On the other hand, straightforward computations using the commutation relations given by the set of equations (\ref{eq:ls}) lead to:
\begin{multline*}
\ov{\mathcal{K}}(H) = - \frac{4}{9} \, \Lbkb \frac{\Lb \left( \Lbk \right)^2}{\Lbk^2} - \frac{1}{9} \, \Lbkb \, \frac{\Lb \left(\Lbk \right)}{\Lbk} \, \ov{P} 
+ \frac{2}{9} \, \Lbkb^2 \, \ov{P}^2  \\ + \frac{1}{3} \, \Lbkb \, \frac{\Lb \left( \Lb \left( \Lbk \right) \right)}{\Lbk} - \frac{1}{3} \, \Lbkb \,
 \Lb\left( \ov{P} \right),
\end{multline*}
that is:
\begin{equation*}
\ov{\mathcal{K}}(H) = - 2 \, \Lbkb \, H.
\end{equation*}
Combining this with (\ref{eq:int1}) and (\ref{eq:int2}) leads to 
\begin{equation*}
W^{\zeta}_{\rho \ov{\zeta}} = 0,
\end{equation*}
which therefore do not provide any new normalization of the group parameters.
We now turn our attention on $W^{\zeta}_{\rho \ov{\kappa}}$. As before, the substitution $\dd =  - \frac{i}{2} \, \frac{e^2 \cb}{cc} + i \, \frac{\cc}{\cb} \, H$ gives
\begin{multline*}
V^{\zeta}_{\rho \ov{\kappa}} =
2i \, \frac{\eb \cc}{\cb^4} \, \frac{\Lbkb}{\Lkb} \, H + \frac{i}{\cb^3} \left(  \frac{4}{3} \, \frac{\Lb \left( \Lbk \right)}{\Lbk} + \ov{P} \right) H \\ +
i \, \frac{\ee}{\cb^2 \cc} \left(- \frac{1}{3} \,  \frac{\Lb \left( \Lbk \right)}{\Lbk} - \frac{2}{9} \, \ov{P}^2 + \frac{1}{9} \, 
\frac{\Lb \left( \Lbk \right)}{\Lbk} \, \ov{P} + \frac{4}{9} \, \frac{\Lb \left( \Lbk \right)^2}{\Lbk^2} + \frac{1}{3} \, \Lb (P) - 2 \, H \right),
\end{multline*}
that is, taking into account the expression of $H$,
\begin{equation*}
V^{\zeta}_{\rho \ov{\kappa}} =
2i \, \frac{\eb \cc}{\cb^4} \, \frac{\Lbkb}{\Lkb} \, H + \frac{i}{\cb^3} \left(  \frac{4}{3} \, \frac{\Lb \left( \Lbk \right)}{\Lbk} + \ov{P} \right) H.
\end{equation*}
Combining this equation with (\ref{eq:int1bis}), we thus get the value of  $W^{\zeta}_{\rho \ov{\kappa}}$:
\begin{align*}
W^{\zeta}_{\rho \ov{\kappa}} & = i \, \frac{\eb \cc}{\cb^4} \, \frac{1}{\Lkb} \left( 2 \, \Lbkb \,  H + \ov{\mathcal{K}}(H) \right) 
+ \frac{i}{\cb^3}  \left[ \frac{2}{3} \left( 2 \, \frac{\Lb \left( \Lbk \right) }{\Lbk} + \ov{P} \right) H - \Lb(H) \right] \\
& = \frac{i}{\cb^3} \left[ \frac{2}{3} \left( 2 \, \frac{\Lb \left( \Lbk \right) }{\Lbk} + \ov{P} \right) H - \Lb(H) \right],
\end{align*}
as the last equality follows from the relation (\ref{eq:int2}).
This provide us with a new essential torsion coefficient, leading to a new invariant of the problem.
Indeed we define the function $J$ by:
\begin{equation*}
\ov{J} := \left[ \frac{2}{3} \left( 2 \, \frac{\Lb \left( \Lbk \right) }{\Lbk} + \ov{P} \right) H - \Lb(H) \right].
\end{equation*}
If $J$ does not vanish, one can perform the normalization 
$\cb^3 := \ov{J}$. We now give the expression of the invariant $J$ in terms of the functions $k$, $\ov{P}$ and their coframe derivatives.
Straightforward computations lead to 
\begin{dmath*}
\Lb(H) = - \frac{4}{9} \, \frac{\Lb \left( \Lbk \right)^3}{\Lbk^3} + \frac{11}{18} \, 
\frac{\Lb \left( \Lbk \right) \Lb \left( \Lb \left( \Lbk \right) \right)}{\Lbk^2} 
 - \frac{1}{18} \, \frac{\Lb \left( \Lbk \right)^2}{\Lbk^2} \, \ov{P}   + 
\frac{1}{18}\, \frac{\Lb  \left( \Lbk \right) }{\Lbk} \, \ov{P} + \frac{1}{18} \, \frac{\Lb \left(\Lbk \right) \, \Lb \left( \ov{P} \right)}{\Lbk}
- \frac{2}{9} \, \ov{P} \, \Lb \left( \ov{P} \right)
- \frac{1}{6} \, \frac{\Lb \left( \Lb \left( \Lb \left( \Lbk \right) \right) \right)}{\Lbk} 
+ \frac{1}{6} \, \Lb \left( \Lb \left( \ov{P} \right) \right),
\end{dmath*}
which in turn gives the expression of $\ov{J}$:
\begin{dmath*}
\ov{J} = \frac{5}{18} \, \frac{\Lb \left( \Lbk \right)^2}{\Lbk^2} \, \ov{P} + \frac{1}{3} \, \ov{P} \, \Lb \left( \ov{P} \right) -
\frac{1}{9} \,  \frac{\Lb \left( \Lbk \right)}{\Lbk} \, \ov{P}^2 + \frac{20}{27} \, \frac{ \Lb \left( \Lbk \right)^3}{\Lbk^3} - \frac{5}{6} 
\, \frac{ \Lb \left( \Lbk \right) \, \Lb \left( \Lb \left( \Lbk \right) \right)}{\Lbk^2} + \frac{1}{6} \, \frac{\Lb \left( \Lbk \right) \, \Lb(\ov{P})}{\Lbk}
- \frac{1}{6} \, \frac{\Lb \left( \Lb \left( \Lbk \right) \right)}{\Lbk} \, \ov{P} - \frac{2}{27} \, \ov{P}^3 - 
\frac{1}{6} \, \Lb \left( \Lb \left( \ov{P} \right) \right) + \frac{1}{6} \,  \frac{\Lb \left( \Lb \left( \Lb \left( \Lbk \right) \right) \right)}{\Lbk} 
\end{dmath*}.
\section{Case $J \neq 0$} \label{section:J}
We now turn our attention on the case $J \neq 0$. We show here how the last group parameter $\ee$ can be normalized,
reducing thus the $G$-equivalence problem to the study of an ${e}$-structure.
From the normalization $c^3 = J$, we get $$\frac{d \cc}{\cc} = \frac{1}{3} \, \frac{d J}{J}.$$ The expression of $d\rho$
is thus modified as:
\begin{equation*}
d \rho  = \left.\frac{1}{3} \left( \frac{d J}{J} + \frac{ d \ov{J}}{\ov{J}} \right) \wedge \rho \right. 
+ V^{\rho}_{\rho \kappa} \left. \rho \wedge \kappa \right. + 
V^{\rho}_{\rho \zeta} \left. \rho \wedge \zeta \right. + V^{\rho}_{\rho \ov{\kappa}} \left. \rho \wedge \ov{\kappa} \right. 
+ V^{\rho}_{\rho \zeta} \left. \rho \wedge \ov{\zeta}
\right. + i \, \kappa \wedge \ov{\kappa},
\end{equation*}
which rewrites
\begin{equation*}
d \rho = S^{\rho}_{\rho \kappa} \left. \rho \wedge \kappa \right. + S^{\rho}_{\rho \zeta} \left. \rho \wedge \zeta \right. + S
^{\rho}_{\rho \ov{\kappa}} \left. \rho \wedge \ov{\kappa} \right.
+ S^{\rho}_{\rho \ov{\zeta}} \left. \rho \wedge \ov{\zeta} \right. + i \, \left. \kappa \wedge \ov{\kappa} \right..
\end{equation*}
From this expression, we see that $S^{\rho}_{\rho \kappa}$, $S^{\rho}_{\rho \zeta}$, $S^{\rho}_{\rho \ov{\kappa}}$ and $S^{\rho}_{\rho \ov{\zeta}}$ 
are essential torsion coefficients. We now turn our attention on the computation of  $S^{\rho}_{\rho \ov{\kappa}}$.

The expression of $d J \wedge \rho$ is obtained in a similar way as that of $d H \wedge \rho$, namely:
\begin{multline*}
d J \wedge \rho =
\left( \frac{\ee \cb}{\cc^2} \,  \frac{\mathcal{K}(J)}{\Lbk}- \frac{\LL(J)}{\cc} \right) \left. \rho \wedge \kappa \right. 
- \frac{\cb}{\cc} \, \frac{\mathcal{K}(J)}{\Lbk} 
\left. \rho \wedge \zeta \right. \\
+ \left( \frac{\eb \cc}{\cb^2} \frac{\ov{\mathcal{K}}(J)}{\Lbk}
- \frac{\Lb(J)}{\cb} \right) \left. \rho \wedge \ov{\kappa} \right.  - \frac{\cc}{\cb} \, \frac{\ov{\mathcal{K}}(J)}{\Lkb} \left. 
\rho \wedge \ov{\zeta} \right.
.\end{multline*}
Replacing $\cc$ by $J^{1/3}$, we thus get that 
\begin{multline*}
\left( \frac{d J}{J} + \frac{d \ov{J}}{\ov{J}} \right) \,  \wedge \,  \rho = \\ \left[
 \frac{ \ee  }{ \Lbk} \, \frac{\ov{J}^{1/3}}{J^{2/3}} \left( \frac{\mathcal{K}(J)}{J} + \frac{\mathcal{K} \left( \ov{J} \right) }{\ov{J}} \right) 
- \frac{\LL(J)}{J^{4/3}} - 
\frac{\LL \left( \ov{J} \right)}{J^{1/3} \, \ov{J}} \right] \left. \rho \wedge \kappa \right. \\
- \frac{1}{\Lbk} \, \frac{\ov{J}^{1/3}}{J^{1/3}}  \, \left( \frac{\mathcal{K}(J)}{J} + \frac{\mathcal{K}\left( \ov{J} \right)}{\ov{J}} \right) 
\left. \rho \wedge \zeta \right. \\
+ \left[
 \frac{ \eb  }{ \Lkb} \, \frac{J^{1/3}}{\ov{J}^{2/3}} \left( \frac{\ov{\mathcal{K}}(J)}{J} + \frac{\ov{\mathcal{K}} \left( \ov{J} \right) }{\ov{J}} \right) 
- \frac{\Lb(J)}{J \ov{J}^{1/3}} - 
\frac{\Lb \left( \ov{J} \right)}{\ov{J}^{4/3}} \right] \left. \rho \wedge \ov{\kappa} \right. \\
- \frac{1}{\Lkb} \, \frac{J^{1/3}}{\ov{J}^{1/3}} \, \left( \frac{\ov{\mathcal{K}}(J)}{J} + \frac{\ov{\mathcal{K}}\left( \ov{J} \right)}{\ov{J}} \right)
 \left. \rho \wedge \ov{\zeta} \right.
\end{multline*}
On the other hand, after replacing $\cc$ by its normalization in $V^{\rho}_{\rho \ov{\kappa}}$, we get:
\begin{equation*}
V^{\rho}_{\rho \ov{\kappa}}
=
-
\frac {{\sf e}} {J^{1/3}}
+
\frac {1 }{3} \, \frac{1}{\ov{J}^{1/3}} \, \frac{\Lb \left( \Lb( 
k)  \right)} {\Lbk}
+
\frac{2}{3} \, \frac {\ov{P}} {\ov{J}^{1/3}} 
+
\eb \, \frac {J^{1/3}} {\ov{J}^{2/3}} \, \frac{\Lbkb }{\Lkb}. 
\end{equation*}
We thus obtain the following essential torsion coefficient, which depends on $\ee$ and $\eb$:
\begin{multline*}
S^{\rho}_{\rho \ov{\kappa}} = -
\frac{\ee}{J^{1/3}} + \frac{\eb}{\Lkb} \, \frac{J^{1/3}}{\ov{J}^{2/3}} \left( \Lbkb + \frac{1}{3} \, \frac{\ov{\mathcal{K}}( \ov{J})}{\ov{J}} +
\frac{1}{3} \,  \frac{\ov{\mathcal{K}}(J)}{J} \right) \\
+ \frac{1}{3} \, \frac{1}{\ov{J}^{1/3}} \left( 2\, \ov{P} + \frac{\Lb \left( \Lbk \right)}{\Lbk} - \frac{\Lb(J)}{J} - \frac{\Lb(\ov{J})}{\ov{J}} \right).
\end{multline*}
The actual computation of the other essential torsion coefficients $S^{\rho}_{\rho \kappa}$, $S^{\rho}_{\rho \zeta}$ and $S^{\rho}_{\rho \ov{\zeta}}$ 
do not lead to any useful equation depending in $\ee$.
On the other hand, the study of the third structure equation provides us with another meaningful essential torsion coefficient.
Indeed we have:
\begin{dmath*}
d \zeta
=
 i \, \delta_{2} \wedge \kappa + \left. \frac{1}{3} \left(\frac{d J}{J} - \frac{d \ov{J}}{\ov{J}}  \right) \wedge 
\zeta \right. \\
+
W^{\zeta}_{\rho \kappa}  \, \rho \wedge \kappa
+
W^{\zeta}_{\rho \zeta} \, \rho \wedge \zeta
+
W^{\zeta}_{\rho \overline{\kappa}} \, \rho \wedge \overline{\kappa} 
+
W^{\zeta}_{\rho \overline{\zeta}} \,\left.  \rho \wedge \overline{\zeta} \right.
+
V^{\zeta}_{\kappa \zeta} \, \left. \kappa \wedge \zeta \right.
+
V^{\zeta}_{\kappa \overline{\kappa}} \, \kappa \wedge \overline{\kappa}
+
V^{\zeta}_{\kappa \overline{\zeta}} \, \kappa \wedge \overline{\zeta} 
+
V^{\zeta}_{\zeta \overline{\kappa}} \, \zeta \wedge \overline{\kappa}
+
V^{\zeta}_{\zeta \overline{\zeta}} \, \zeta \wedge \overline{\zeta}
,\end{dmath*}
which, taking into account the facts that $W^{\zeta}_{\rho \ov{\kappa}} =0$ and that $W^{\zeta}_{\rho \ov{\zeta}}$ as been normalized to $1$, can be rewritten as
\begin{dmath*}
d \zeta
=
 i \, \delta_{2} \wedge \kappa \\
+
S^{\zeta}_{\rho \kappa}  \, \rho \wedge \kappa
+
S^{\zeta}_{\rho \zeta} \, \rho \wedge \zeta
+
\rho \wedge \overline{\zeta}
+
S^{\zeta}_{\kappa \zeta} \, \kappa \wedge \zeta 
+
S^{\zeta}_{\kappa \overline{\kappa}} \, \kappa \wedge \overline{\kappa}
+
S^{\zeta}_{\kappa \overline{\zeta}} \, \kappa \wedge \overline{\zeta} 
+
S^{\zeta}_{\zeta \overline{\kappa}} \, \zeta \wedge \overline{\kappa}
+
S^{\zeta}_{\zeta \overline{\zeta}} \, \zeta \wedge \overline{\zeta}
,\end{dmath*}
where the $S^{\smallbullet}_{\smallbullet \smallbullet}$ are new torsion coeficients.
We easily deduce from this equation that 
\begin{equation*}
S^{\zeta}_{\zeta \ov{\kappa}} = V^{\zeta}_{\zeta \ov{\kappa}} + \frac{1}{3} \left[
\frac{ \eb  }{ \Lkb} \, \frac{J^{1/3}}{\ov{J}^{2/3}} \left( \frac{\ov{\mathcal{K}}(J)}{J} - \frac{\ov{\mathcal{K}} \left( \ov{J} \right) }{\ov{J}} \right) 
- \frac{\Lb(J)}{J \ov{J}^{1/3}} +
\frac{\Lb \left( \ov{J} \right)}{\ov{J}^{4/3}} \right]
\end{equation*}
is an essential torsion coefficient.
From the expression of $V^{\zeta}_{\zeta \ov{\kappa}}$ obtained by performing the substitution $c := J^{\frac{1}{3}}$, we have 
\begin{multline*}
S^{\zeta}_{\zeta \ov{\kappa}} = \frac{\ee}{J^{1/3}} - \eb \, \frac{J^{1/3}}{\ov{J}^{2/3}} \, \frac{\Lbkb}{\Lkb} - 
\frac{1}{\ov{J}^{1/3}} \, \frac{\Lb \left( \Lbk \right)}{\Lbk} \\
+ \frac{1}{3} \left[
\frac{ \eb  }{ \Lkb} \, \frac{J^{1/3}}{\ov{J}^{2/3}} \left( \frac{\ov{\mathcal{K}}(J)}{J} - \frac{\ov{\mathcal{K}} \left( \ov{J} \right) }{\ov{J}} \right) 
- \frac{\Lb(J)}{J \ov{J}^{1/3}} +
\frac{\Lb \left( \ov{J} \right)}{\ov{J}^{4/3}} \right].
\end{multline*}
We now substract the two essential torsion coefficients that we have get so far:
\begin{multline*}
-S^{\rho}_{\rho \ov{\kappa}} + S^{\zeta}_{\zeta \ov{\kappa}} =
2 \, \frac{\ee}{J^{1/3}} - 2 \, \eb \, \frac{J^{1/3}}{\ov{J}^{2/3}} \, \frac{1}{\Lkb} \left( \Lbkb +  \frac{1}{3} \, \frac{\ov{\mathcal{K}}(\ov{J})}{\ov{J}} \right) \\
 + \frac{2}{3} \, \frac{1}{\ov{J}^{1/3}} \left( \frac{\Lb(\ov{J})}{\ov{J}} - 2\, \frac{\Lb \left( \Lbk \right) }{\Lbk} - \ov{P} \right).
\end{multline*}
From the full expression of $\mathcal{K}(J)$ in terms of the coframe derivatives, obtained by using extensively the commutations relations
(\ref{eq:ls}),
we find the relation:
\begin{equation*}
  \frac{1}{3} \, \ov{\mathcal{K}}\left(\ov{J}\right) + \Lbkb  \cdot \ov{J} = 0,
\end{equation*}
from which we deduce that the following expression:
\begin{equation*}
\frac{ \ee}{J^{1/3}} + \frac{1}{3} \, \frac{1}{\ov{J}^{1/3}} \left( \frac{\Lb(\ov{J})}{\ov{J}} - 2 \, \frac{\Lb \left( \Lbk \right) }{\Lbk} - \ov{P} \right)
\end{equation*}
is an essential torsion coefficient.
Setting this coefficient to zero, gives the normalization of $\ee$:
\begin{equation*}
\ee = \frac{1}{3} \, \frac{J^{1/3}}{\ov{J}^{1/3}} \left( - \frac{\Lb(\ov{J})}{\ov{J}} + 2 \, \frac{\Lb \left( \Lbk \right) }{\Lbk} + \ov{P} \right).
\end{equation*}

\section{Case $W \neq 0$} \label{section:W} 
We now assume that the fonction $W$ does not vanish on $M$, and we show how the group parameter $\ee$ can be normalized.
We choose the normalization $c:= W$. We recall that prior to this last normalization, the structure equations read:

\begin{dmath*}
d \rho
 =
\left.\delta^1 \wedge \rho  \right.+ \left.  \overline{\delta^{1}} \wedge \rho \right.\\
+
V^{\rho}_{\rho \kappa} \, \left. \rho \wedge \kappa \right.
+
V^{\rho}_{\rho \zeta} \, \left. \rho \wedge \zeta \right.
+
V^{\rho}_{\rho \overline{\kappa}} \, \left.  \rho \wedge \overline{\kappa} \right.
+
V^{\rho}_{\rho \overline{\zeta}} \, \left.  \rho \wedge \overline{\zeta} \right.
+
i \, \left. \kappa \wedge \overline{\kappa} \right.
\end{dmath*},

\begin{dmath*}
d \kappa
 =
\left. \delta^{1} \wedge \kappa \right. + \left.  \delta^{2} \wedge \rho \right.\\
+
V^{\kappa}_{\rho \kappa} \, \left. \rho \wedge \kappa \right.
+
V^{\kappa}_{\rho \zeta} \, \left.  \rho \wedge \zeta \right.
+
V^{\kappa}_{\rho \overline{\kappa}} \, \left. \rho \wedge \overline{\kappa} \right.
+
V^{\kappa}_{\rho \overline{\zeta}} \,\left.  \rho \wedge \overline{\zeta} \right.
+
V^{\kappa}_{\kappa \zeta} \, \left.  \kappa \wedge \zeta \right.
+
V^{\kappa}_{\kappa \overline{\kappa}} \, \left. \kappa \wedge \overline{\kappa} \right.
+
\left.\zeta \wedge \overline{\kappa} \right.
\end{dmath*}

and 

\begin{dmath*}
d \zeta
=
 i \, \left. \delta_{2} \wedge \kappa \right. + \left. \delta_{1} \wedge 
\zeta \right. - \left. \overline{\delta_{1}} \wedge \zeta \right.\\
+
W^{\zeta}_{\rho \kappa}  \, \left. \rho \wedge \kappa \right.
+
W^{\zeta}_{\rho \zeta} \,\left.  \rho \wedge \zeta \right.
+
W^{\zeta}_{\rho \overline{\kappa}} \, \left. \rho \wedge \overline{\kappa}  \right.
+
V^{\zeta}_{\kappa \zeta} \, \left.  \kappa \wedge \zeta \right.
+
V^{\zeta}_{\kappa \overline{\kappa}} \, \left. \kappa \wedge \overline{\kappa} \right.
+
V^{\zeta}_{\kappa \overline{\zeta}} \, \left. \kappa \wedge \overline{\zeta}  \right.
+
V^{\zeta}_{\zeta \overline{\kappa}} \, \left. \zeta \wedge \overline{\kappa} \right.
+
V^{\zeta}_{\zeta \overline{\zeta}} \, \left.  \zeta \wedge \overline{\zeta} \right.
,\end{dmath*}

where 
\begin{equation*}
\delta^1 = \frac{d \cc}{\cc} \quad, \quad \delta^2 =  i \, \ee \frac{d \cc}{\cc^2} - i\, \frac{\ee \, d \cb}{\cc \cb} - i\, \frac{d \ee}{\cc},
\end{equation*}
and 
\begin{equation*}
W^{\zeta}_{\rho \overline{\kappa}} =i \, \frac{\ov{J}}{\cb^3}.
\end{equation*}

As we have 
\begin{equation*}
 \delta^2 =  - i\, \frac{\ee \, d \cb}{\cc \cb} - i\, d \left( \frac{\ee}{\cc} \right),
\end{equation*}
it is convenient to introduce the new parameter $\epsilon$ defined by
\begin{equation*}
\epsilon := \frac{\ee}{\cc}.
\end{equation*}
With the normalization $c:=W$, we get: 
\begin{equation*}
\delta^1 = \frac{d W}{W},
\end{equation*}
\begin{equation*}
\delta^2 = - i \, d \epsilon - i \epsilon \frac{d \ov{W}}{W}
\end{equation*}
and
\begin{equation*}
W^{\zeta}_{\rho \overline{\kappa}} =i \, \frac{\ov{J}}{W^3}.
\end{equation*}

As a result, the new structure equations take the form:

\begin{dmath*}
d \rho = X^{\rho}_{\rho \kappa} \left. \rho \wedge \kappa \right. + X^{\rho}_{\rho \zeta} \left. \rho \wedge \zeta \right. + X
^{\rho}_{\rho \ov{\kappa}} \left. \rho \wedge \ov{\kappa} \right.
+ X^{\rho}_{\rho \ov{\zeta}} \left. \rho \wedge \ov{\zeta} \right. + i \, \left. \kappa \wedge \ov{\kappa} \right.,
\end{dmath*}

\begin{dmath*}
d \kappa = - i \, d \epsilon \wedge \rho \\
 + X^{\kappa}_{\rho \kappa} \, \left. \rho \wedge \kappa \right. + X^{\kappa}_{\rho \zeta}\, \left.  \rho \wedge \zeta \right.
 + X^{\kappa}_{\rho \ov{\kappa}} \, \left. \rho \wedge \ov{\kappa} \right.+
X^{\kappa}_{\rho \ov{\zeta}} \, \rho \wedge \ov{\zeta} +X^{\kappa}_{\kappa \zeta} \left. \kappa \wedge \zeta \right.
 + X^{\kappa}_{\kappa \ov{\kappa}} \,  \kappa \wedge \ov{\kappa} + X^{\kappa}_{\kappa \ov{\zeta}} \, 
 \left. \kappa \wedge \ov{\zeta} \right.
+ \zeta \wedge \ov{\kappa},
\end{dmath*}

\begin{dmath*}
d \zeta = d \epsilon \wedge \kappa \\
 + X^{\zeta}_{\rho \kappa} \, \left. \rho \wedge \kappa \right. + X^{\zeta}_{\rho \zeta}\,  \left. \rho \wedge \zeta \right.
 + X^{\zeta}_{\rho \ov{\kappa}} \, \left. \rho \wedge \ov{\kappa} \right. + X^{\zeta}_{\kappa \zeta} \, \left. \kappa \wedge \zeta \right.
 + X^{\zeta}_{\kappa \ov{\kappa}} \, \left. \kappa \wedge \ov{\kappa} \right. + X^{\zeta}_{\kappa \ov{\zeta}} \, 
\left. \kappa \wedge \ov{\zeta} \right.
+ X^{\zeta}_{\zeta \ov{\kappa}} \left. \zeta \wedge \ov{\kappa} \right.+ X^{\zeta}_{\zeta \ov{\zeta}} \, \left. \zeta \wedge \ov{\zeta} \right.
,\end{dmath*}
for a new set of torsion coefficients $X^{\smallbullet}_{\smallbullet \smallbullet}$.
The absorption process is straightforward and leads to the following essential torsion coefficients:
\begin{alignat*}{3}
& X^{\rho}_{\rho \kappa}, \qquad   X^{\rho}_{\rho \zeta}, \qquad  X^{\rho}_{\rho \ov{\kappa}}, \qquad
X^{\rho}_{\rho \ov{\zeta}},  \\   & X^{\kappa}_{\kappa \zeta}, \qquad   X^{\kappa}_{\kappa \ov{\kappa}},
\qquad X^{\kappa}_{\kappa \ov{\zeta}}, \qquad  X^{\zeta}_{\rho \zeta}, \\ & X^{\zeta}_{\rho \ov{\kappa}} ,
\qquad  X^{\zeta}_{\zeta \ov{\kappa} }, \qquad X^{\zeta}_{\zeta \ov{\zeta}}, \qquad  i \, X^{\zeta}_{\kappa \zeta} + X^{\kappa}_{\rho \zeta},\\
& i \, X^{\zeta}_{\kappa \ov{\kappa}} + X^{\kappa}_{\rho \ov{\kappa}}, \qquad 
i \, X^{\zeta}_{\kappa \ov{\zeta}} + X^{\kappa}_{\rho \ov{\zeta}}.
\end{alignat*}
The careful computation of the coefficient $X^{\kappa}_{\kappa \ov{\kappa}}$ gives:
\begin{equation*}
X^{\kappa}_{\kappa \ov{\kappa}} = \ov{\epsilon} \, \frac{\Kb \left(W \right)}{\ov{W} \Lkb} - \frac{\Lb \left( W \right)}{W \ov{W}} - 
\frac{1}{3} \, \frac{\Lb \left( \Lbk \right)}{W \Lbk} + \frac{1}{3} \, \frac{\ov{P}}{\ov{W}}.
\end{equation*}

The expression of $\Kb \left( W \right)$ can be simplified by using the commutations relations
(\ref{eq:ls}), as in the case of $\KK \left( J \right)$.
We find the relation:
\begin{equation*}
\Kb(W) + 2 \, \Lbk \, \ov{W} = 0,
\end{equation*}
from which we deduce that 
$X^{\kappa}_{\kappa \ov{\kappa}}$ rewrites:
\begin{equation*}
X^{\kappa}_{\kappa \ov{\kappa}} = - 2 \, \ov{\epsilon} - \frac{\Lb \left( W \right)}{W \ov{W}} - 
\frac{1}{3} \, \frac{\Lb \left( \Lbk \right)}{W \Lbk} + \frac{1}{3} \, \frac{\ov{P}}{\ov{W}}
.\end{equation*}
Setting this coefficient to zero, we get a normalization of $\epsilon$, and hence of $\ee$, provided that $\Kb \left(W \right)$ does not vanish on $M$, 
which is given by the following lemma:

\begin{lemma} \label{lemma:KbW}
$\Kb \left(W \right)$ does not vanish identically on $M$. 
\end{lemma}
\begin{proof}
The computation of $\Kb(W)$, using the commutation relations (\ref{eq:ls}) leads to the following formula:
\begin{equation*}
\Kb(W) + 2\,  \Lkb \, \ov{W} +2  i \, \Tkb = 0.
\end{equation*}
If $\Kb(W) = 0$ then $W= i \frac{\Tk}{\Lbk}$ which implies  $\Kb(W) =- i \, \Tk$ (using (\ref{eq:ls}) once again), that is 
\begin{equation*}
\Kb(W) = \ov{W} \Lkb,
\end{equation*}
which gives a contradiction with the fact 
that $W \neq 0$.
\end{proof}

\section{Case $J=0$ and $W=0$} \label{section:prolongation}
We show that in this case, $M$ is biholomorphically equivalent to the light cone.
We start by showing that the coefficient $W^{\zeta}_{\rho \zeta}$ is purely imaginary, which implies that no further group reductions are allowed at this stage.
The full computation of this coefficient leads to:
\begin{dmath*}
i \, \cc \cb \,  W^{\zeta}_{\rho \zeta} = - \frac{1}{6} \, \Lb(P) - \frac{1}{6} \, \LL(\ov{P}) - \frac{2}{3} \, \frac{\cb \, \ee}{\cc} \, 
\frac{\LL \left( \Lbk \right)}{\Lbk} + \frac{1}{2} \,  \frac{\cb^2 \, \ee^2}{\cc^2} \, \frac{\Lk}{\Lbk} + i \, \frac{\mathcal{T} \left( \Lbk \right)}{\Lbk} 
+ \frac{1}{3} \, \frac{\cb \ee}{\cc} \, \frac{\KK \left( \Lb \left( \Lbk \right) \right)}{\Lbk^2} + \frac{1}{3} \,  \frac{\cb \, \ee}{\cc} \, 
\frac{\LL \left( \Lkb \right)}{\Lkb} + \frac{1}{2} \,  \frac{\cc^2 \, \eb^2}{\cb^2} \, \frac{\Lbkb}{\Lkb}
 - \frac{1}{3} \,  \frac{\cb \, \ee}{\cc} \, \frac{\Lb \left( \Lbk \right) \KK \left( \Lbk \right)}{\Lbk^3} + \frac{1}{18} \, \frac{\Lbkb \LL \left( \Lkb \right)}
{\Lkb^2} \, P - \frac{1}{3} \, \frac{\Lb \left( \Lbk \right) \LL \left( \Lkb \right)}{\Lbk \Lkb} - \frac{1}{18} \, 
\frac{\KK \left( \Lbk \right) \Lb \left( \Lbk \right)}{\Lbk^3} \, \ov{P} - \ee \, \eb + \frac{2}{9} \, P \, \ov{P} + \frac{4}{9} \, \frac{\Lb \left( \Lbk \right) \KK
\left( \Lb \left( \Lbk \right) \right)}{\Lbk^3} - \frac{1}{6} \, \frac{\KK \left( \Lb \left( \Lb \left( \Lbk \right) \right) \right)}{\Lbk^2} - \frac{1}{6} \, 
\frac{\Lb \left( \LL \left(\Lbk \right) \right)}{\Lbk} - \frac{1}{9} \, \frac{\Lk}{\Lbk} \, \ov{P}^2 - \frac{1}{6} \, \frac{\Lbkb \LL \left( \LL \left( \Lkb \right) 
\right)}{\Lkb^2} + \frac{1}{6} \, \frac{\Lbkb  \LL(P)}{\Lkb} +  \frac{1}{6} \, \frac{\Lk \Lb(\ov{P})}{\Lbk} -
\frac{4}{9} \, \frac{\KK \left( \Lbk \right) \Lb \left( \Lbk \right)^2}
{\Lbk^4} -  \frac{1}{9} \, \frac{\Lbkb P^2}{\Lkb} +  \frac{1}{18} \, \frac{\KK \left( \Lb \left( \Lbk \right) \right) \ov{P}}{\Lbk^2} +  \frac{2}{9} \, 
\frac{\Lbkb \LL \left( \Lkb \right)^2}{\Lkb^3} - \frac{1}{6} \, \frac{\Lk \Lb \left( \Lb \left( \Lbk \right) \right)}{\Lbk^2} + \frac{5}{18} \, 
\frac{\LL \left( \Lbk \right) \Lb \left( \Lbk \right)}{\Lbk^2} + \frac{2}{9} \, \frac{\LL(k) \Lb \left( \Lbk \right)^2}{\Lbk^3} + \frac{1}{6} \, \frac{\KK
\left( \Lbk \right) \Lb \left( \Lb \left( \Lbk \right) \right)}{\Lbk^3} - \frac{1}{9} \, \frac{\LL \left( \Lbk \right)}{\Lbk} \, \ov{P} + \frac{1}{9} \, \frac{\Lb 
\left( \Lbk \right)}{\Lbk} \, P + \frac{1}{18} \, \frac{\Lk \Lb \left( \Lbk \right) }{\Lbk^2} \, \ov{P} + \frac{\cc \eb}{\cb} \, \frac{\Lb \left( \Lbk \right)}{\Lbk}
- \frac{i}{3} \, \frac{\ee \cb}{\cc} \, \frac{\Tk}{\Lbk}
- \frac{i}{3} \, \frac{\Lb \left( \Tk \right)}{\Lbk}
- \frac{i}{9} \, \frac{\Lb\left( \Lbk \right)}{\Lbk^2} \, \Tk
+ \frac{4i}{9} \, \frac{\ov{P}}{\Lbk} \, \Tk.
\end{dmath*} 
As we shall check that this expression is real, we just drop the terms which come together with their conjugate counterpart, i.e., we perform a computation 
mod $\mathbb{R}$. We thus get:
\begin{dmath*}
i \, \cc \cb \,  W^{\zeta}_{\rho \zeta} \equiv - \frac{2}{3} \, \frac{\cb \, \ee}{\cc} \, 
\frac{\LL \left( \Lbk \right)}{\Lbk} + i \, \frac{\mathcal{T} \left( \Lbk \right)}{\Lbk}
+ \frac{1}{3} \, \frac{\cb \ee}{\cc} \, \frac{\KK \left( \Lb \left( \Lbk \right) \right)}{\Lbk^2}
+ \frac{1}{3} \,  \frac{\cb \, \ee}{\cc} \, \frac{\LL \left( \Lkb \right)}{\Lkb} 
- \frac{1}{3} \,  \frac{\cb \, \ee}{\cc} \, \frac{\Lb \left( \Lbk \right) \KK \left( \Lbk \right)}{\Lbk^3}  - \frac{1}{18} \, 
\frac{\KK \left( \Lbk \right) \Lb \left( \Lbk \right)}{\Lbk^3} \, \ov{P} + \frac{4}{9} \, \frac{\Lb \left( \Lbk \right) \KK
\left( \Lb \left( \Lbk \right) \right)}{\Lbk^3} - \frac{1}{6} \, \frac{\KK \left( \Lb \left( \Lb \left( \Lbk \right) \right) \right)}{\Lbk^2} 
  -\frac{4}{9} \, \frac{\KK \left( \Lbk \right) \Lb \left( \Lbk \right)^2}{\Lbk^4} 
+ \frac{1}{18} \, \frac{\KK \left( \Lb \left( \Lbk \right) \right) \ov{P}}{\Lbk^2}  + \frac{5}{18} \, 
\frac{\LL \left( \Lbk \right) \Lb \left( \Lbk \right)}{\Lbk^2}+ \frac{1}{6} \, \frac{\KK
\left( \Lbk \right) \Lb \left( \Lb \left( \Lbk \right) \right)}{\Lbk^3} - \frac{1}{9} \, \frac{\LL \left( \Lbk \right)}{\Lbk} \, \ov{P} + \frac{1}{9} \, \frac{\Lb 
\left( \Lbk \right)}{\Lbk} \, P + \frac{\cc \eb}{\cb} \, \frac{\Lb \left( \Lbk \right)}{\Lbk}
- \frac{i}{3} \, \frac{\ee \cb}{\cc} \, \frac{\Tk}{\Lbk}
- \frac{i}{3} \, \frac{\Lb \left( \Tk \right)}{\Lbk}
- \frac{i}{9} \, \frac{\Lb\left( \Lbk \right)}{\Lbk^2} \, \Tk
+ \frac{4i}{9} \, \frac{\ov{P}}{\Lbk} \, \Tk.
\end{dmath*}
We now give an expression of $i \, \cc \cb \,  W^{\zeta}_{\rho \zeta}$ in terms of the function $W$ and its derivative by $\Lb$.
Using the expression of $W$ given by $(\ref{eq:W})$ and dropping 
once again the terms which come with their conjugate counterpart, we get the formula:
\begin{equation*}
i \, \cc \cb \,  W^{\zeta}_{\rho \zeta} \equiv \frac{1}{6} \left(
\frac{\Lb \left( \Lbk \right)}{\Lbk} - \ov{P} \right) W
+ \frac{1}{2} \, \Lb(W) - \frac{\ee \cb}{\cc} \, W,
\end{equation*}
from which we get that  $W^{\zeta}_{\rho \zeta}$ is purely imaginary under that assumption that $W$ does vanish identically on $M$.

The normalization step of Cartan's algorithm stops here and we shall now perform a prolongation of the problem.
We introduce the modified Maurer Cartan forms of the group $G_4$, namely:
\begin{equation*}
\left\{
\begin{aligned}
& \hat{\delta}^1 : = \delta^1 -  w^1_{\rho} \, \rho - w^1_{\kappa} \, \kappa - w^1_{\zeta} \, \zeta - w^1_{\ov{\kappa}} \, \ov{\kappa} - w^1_{\ov{\zeta}} \, \ov{\zeta} \\
& \hat{\delta}^2 : =  \delta^2 - w^2_{\rho} \, \rho - w^2_{\kappa} \, \kappa - w^2_{\zeta} \, \zeta - w^2_{\ov{\kappa}} \, \ov{\kappa} - w^2_{\ov{\zeta}} \, \ov{\zeta} 
\end{aligned}
\right.
\end{equation*}
where $w^i_{\rho}$, $w^i_{\kappa}$, $w^i_{\zeta}$, $w^i_{\ov{\kappa}}$, $w^i_{\ov{\zeta}}$, $i = 1 , \,  2 $,  are the solutions 
of the system of equations (\ref{absorption}) corresponding to 
$w^1_{\rho} + \ov{w^1_{\rho}} = 0,$ that is:
\begin{equation*}
\left\{
\begin{aligned}
& \hat{\delta}^1 : = \delta^1 + \frac{1}{2} \, V^{\zeta}_{\rho \zeta} \, \rho 
- \ov{V^{\rho}_{\rho \kappa}} \, \kappa - V^{\rho}_{\rho \zeta} \, \zeta - V^{\kappa}_{\kappa \ov{\kappa}} \, \ov{\kappa} \\
& \hat{\delta}^2 : =  \delta^2 - V^{\zeta}_{\rho \kappa} \, \rho - \left( V^{\kappa}_{\rho \kappa} - \frac{1}{2}  V^{\zeta}_{\rho \zeta} \right) \kappa 
- V^{\kappa}_{\rho \zeta} \, \zeta - V^{\kappa}_{\rho \ov{\kappa}} \, \ov{\kappa} - V^{\kappa}_{\rho \ov{\zeta}} \, \ov{\zeta} 
.\end{aligned}
\right.
\end{equation*}
We also introduce the modified Maurer Cartan forms which correspond to solutions of the system (\ref{absorption}) when ${\sf Re}( w^1_{\rho})$ 
is not necessarily set to zero, namely:
\begin{equation*}
\left\{
\begin{aligned}
& \pi^1 : =\hat{ \delta}^1 - {\sf Re}(w^1_{\rho}) \, \rho  \\
& \pi^2 : =  \hat{\delta}^2 - {\sf Re}(w^1_{\rho}) \, \kappa.
\end{aligned}
\right.
\end{equation*}
Let $P^9$ be the nine dimensionnal $G_4$-structure constituted  by the set of all coframes of the form $(\rho, \kappa, \zeta, \ov{\kappa}, \ov{\zeta})$ on $M$.
The initial coframe $(\rho_0, \kappa_0, \zeta_0, \ov{\kappa}_0, \ov{\zeta}_0)$ gives a natural trivialisation $P^9 \stackrel{p} \longrightarrow M \times G_{4}$
which allows us to consider any differential form on $M$ or $G^4$ as a differential form on $P^9$. If $\omega$ is a differential form on $M$ for example, we just 
consider $p^*( pr_1^*(\omega))$,  where $pr_1$ is the projection on the first component $M \times G_4 \stackrel{pr_1} \longrightarrow M$. 
We still denote this form by $\omega$ in the sequel.
Following \cite{Olver-1995}, we introduce the two coframes $(\rho, \kappa, \zeta, \ov{\kappa}, \ov{\zeta}, \delta^1, \delta^2, \ov{\delta^1}, \ov{\delta^2})$ 
and $(\rho, \kappa, \zeta, \ov{\kappa}, \ov{\zeta}, \pi^1, \pi^2, \ov{\pi^1}, \ov{\pi^2})$ on $P^9$. Setting $ t:= - {\sf Re}(w^1_{\rho})$, they relate to each other
by the relation:
\begin{equation*}
\begin{pmatrix}
\rho \\
\kappa \\
\zeta \\
\ov{\kappa} \\
\ov{\zeta} \\
\pi^1 \\
\pi^2 \\
\ov{\pi^1} \\
\ov{\pi^2} \\
\end{pmatrix}
= g_{t}
\cdot
\begin{pmatrix}
\rho \\
\kappa \\
\zeta \\
\ov{\kappa} \\
\ov{\zeta} \\
\delta^1 \\
\delta^2 \\
\ov{\delta^1} \\
\ov{\delta^2} \\
\end{pmatrix}
\end{equation*}
where $g_t$ is defined by 
\begin{equation*}
g_t:=\begin{pmatrix}
1 & 0 & 0 & 0 & 0 & 0 & 0 & 0 & 0 \\
0 & 1 & 0 & 0 & 0 & 0 & 0 & 0 & 0 \\
0 & 0 & 1 & 0 & 0 & 0 & 0 & 0 & 0 \\
0 & 0 & 0 & 1 & 0 & 0 & 0 & 0 & 0\\
0 & 0 & 0 & 0 & 1 & 0 & 0 & 0 & 0\\
t & 0 & 0 & 0 & 0 & 1 & 0 & 0 & 0 \\
0 & t & 0 & 0 & 0 & 0 & 1 & 0 & 0 \\
t & 0 & 0 & 0 & 0 & 0 & 0 & 1 & 0 \\
0 & 0 & 0 & t & 0 & 0 & 0 & 0 & 1 \\
\end{pmatrix}
.
\end{equation*}
The set $\left \{ g_t, t \in \mathbb{R} \right\}$ defines a one dimensional Lie group $G_{prol}$, whose Maurer Cartan form is given by $d t$.
We now start the absorption-normalization procedure in Cartan's method on $P^9$. 

From the definition of $\pi^1$ and $\pi^2$ as the solutions of the absorption equations $(\ref{absorption})$, the five first structure equations read
as 
\begin{equation}
\label{steq}
\begin{aligned}
d \rho & = \pi^1 \wedge \rho + \ov{\pi^1} \wedge \rho + i \, \kappa \wedge \ov{\kappa}, \\
d \kappa & = \pi^1 \wedge \kappa + \pi^2 \wedge \rho + \zeta \wedge \ov{\kappa}, \\
d \zeta & = i \, \pi^2 \wedge \kappa + \pi^1 \wedge \zeta - \ov{\pi^1} \wedge \zeta,  \\
d \ov{\kappa} & = \ov{\pi^1} \wedge \ov{\kappa} + \ov{\pi^2} \wedge \rho - \kappa \wedge \ov{\zeta}, \\
d \ov{\zeta} & = - i \, \ov{\pi^2} \wedge \ov{\kappa} + \ov{\pi^1} \wedge \ov{\zeta} - {\pi^1} \wedge \ov{\zeta}. \\
\end{aligned}
\end{equation}

The computations that follow aim to determine the expressions of $d \pi^1$ and $d \pi^2$. Both of these expressions can be deduced from the 
the set of equations $(\ref{steq})$.
For example, taking the exterior derivative of both sides of the equation giving $d \rho$, we get:
\begin{equation*}
0 = d \pi^1 \wedge \rho - \pi^1 \wedge d \rho + d \ov{\pi^1} \wedge \rho - \ov{\pi^1} \wedge d \rho + i \, d \kappa \wedge \ov{\kappa} 
- i \, \kappa \wedge d \ov{\kappa}
.\end{equation*}
Replacing each two-form $d\rho$, $d \kappa$ and $d \ov{\kappa}$ by its expression given by (\ref{steq}) yields:
\begin{dmath*}
0 = d \pi^1 \wedge \rho  + d \ov{\pi^1} \wedge \rho - \pi^1 \wedge \left( \pi^1 \wedge \rho + \ov{\pi^1} \wedge \rho + i \, \kappa \wedge \ov{\kappa} \right) \\
- \ov{\pi^1} \wedge \left( \pi^1 \wedge \rho + \ov{\pi^1} \wedge \rho + i \, \kappa \wedge \ov{\kappa} \right) + i \, \left(\pi^1 \wedge \kappa + \pi^2 \wedge 
\rho + \zeta \wedge \ov{\kappa} \right) \wedge \ov{\kappa}  \\ 
-i \, \kappa \wedge \left( \ov{\pi^1} \wedge \ov{\kappa} + \ov{\pi^2} \wedge \rho - \kappa \wedge \ov{\zeta}  \right),
\end{dmath*}
which can be simplified as:
\begin{equation*}
0 = \left( d \pi^1  - i \, \kappa \wedge \ov{\pi^2} +  \ov{d \pi^1} +  i \, \ov{\kappa} \wedge \pi^2 \right) \wedge \rho.
\end{equation*}

Performing the same computation from the equation giving $d \kappa$, we get:
\begin{equation*}
0 = d \pi^1 \wedge \kappa -  \pi^1 \wedge d \kappa + d \pi^2 \wedge  \rho - \pi^2 \wedge d \rho + d \zeta \wedge \ov{\kappa}
-  \zeta \wedge d \ov{\kappa},
\end{equation*}
that is:
\begin{dmath*}
0 = d \pi^1 \wedge \kappa -  \pi^1 \wedge \left(  \pi^1 \wedge \kappa + \pi^2 \wedge \rho + \zeta \wedge \ov{\kappa} \right) 
+ d \pi^2 \wedge  \rho - \pi^2 \wedge \left( \pi^1 \wedge \rho + \ov{\pi^1} \wedge \rho + i \, \kappa \wedge \ov{\kappa} \right) 
+ \left( i \, \pi^2 \wedge \kappa + \pi^1 \wedge \zeta - \ov{\pi^1}  \wedge \zeta \right) \wedge \ov{\kappa}
-  \zeta \wedge \left(\ov{\pi^1} \wedge \ov{\kappa} + \ov{\pi^2} \wedge \rho - \kappa \wedge \ov{\zeta} \right),
\end{dmath*}
which yields:
\begin{dmath*}
 0 = \left(d \pi^1 - \zeta \wedge \ov{\zeta} \right) \wedge \kappa + \left( d \pi^2 - \pi^2 \wedge \ov{\pi^1} 
- \zeta \wedge \ov{\pi^2} \right) \wedge \rho.
\end{dmath*}

On the other hand, the same computation with the equation giving $d \zeta$ leads to 
\begin{dmath*}
0 = i \, d\pi^2 \wedge \kappa  - i \, \pi^2 \wedge \left(  \pi^1 \wedge \kappa + \pi^2 \wedge \rho + \zeta \wedge \ov{\kappa}  \right)+ d \pi^1 \wedge \zeta 
- d \ov{\pi^1} \wedge \zeta +  \left( \ov{\pi^1} - \pi^1  \right) \wedge \left( i \, \pi^2 \wedge \kappa + \pi^1 \wedge \zeta - \ov{\pi^1} \wedge \zeta \right), 
\end{dmath*}
that is:
\begin{equation*}
0 = \left( d \pi^1 - d \ov{\pi^1} - i \, \ov{\kappa} \wedge \pi^2  \right) \wedge \zeta +
 i \, \left( d \pi^2 - \pi^2 \wedge  \ov{\pi^1} \right) \wedge \kappa.
\end{equation*}
Let us introduce the two-forms $\Omega_1$ and $\Omega_2$ defined by
\begin{equation*}
\Omega_1 := d \pi^1 - i \, \kappa \wedge \ov{\pi^2} - \zeta \wedge \ov{\zeta},
\end{equation*}
and
\begin{equation*}
\Omega_2 := d \pi^2 - \pi^2 \wedge \ov{\pi^1} - \zeta \wedge \ov{\pi^2}.
\end{equation*}
With these notations, the three equations that we have obtained so far rewrite:
\begin{equation} \label{eq:intermediate}
\left\{\begin{aligned}
0 & = \left(\Omega_1 + \ov{\Omega_1}  \right) \wedge \rho, \\
0 & =\Omega_1 \wedge \kappa + \Omega_2 \wedge \rho, \\
0 & = \left( \Omega_1 - \ov{\Omega_1} \right) \wedge \zeta + i \, \Omega_2 \wedge \kappa.
\end{aligned}
\right.
\end{equation}

Taking the exterior product with $\kappa$ in the second equation gives:
\begin{equation*}
0 = \Omega_2 \wedge \rho \wedge \kappa,
\end{equation*}
from which we can deduce the two relations:
\begin{equation*}
\begin{aligned}
0 &= \left(\Omega_1 + \ov{\Omega_1}  \right) \wedge \rho \wedge \zeta, \\
0 & = \left( \Omega_1 - \ov{\Omega_1} \right) \wedge \rho \wedge \zeta,
\end{aligned}
\end{equation*}
which yields:
\begin{equation*}
\Omega_1  \wedge \rho \wedge \zeta = 0.
\end{equation*}
This implies the existence of two $1$-forms $\alpha$ and $\beta$ such that:
\begin{equation*}
\Omega_1 = \alpha \wedge \rho + \beta \wedge \zeta.
\end{equation*}

Similarly, there exist two $1$-form $\gamma$ and $\delta$ such that:
\begin{equation*}
\Omega_2 = \gamma \wedge \rho + \delta \wedge \kappa.
\end{equation*}
Inserting these two expressions in the second equation of $(\ref{eq:intermediate})$,
we obtain the existence of a real $1$-form $\Lambda$ such that:
\begin{equation*}
\begin{aligned}
\Omega_1 & = \Lambda \wedge \rho, \\
\Omega_2 & = \Lambda \wedge \kappa.
\end{aligned}
\end{equation*}
If we come back to the expression of $d \pi^1$ and $d \pi^2$, we get the two following additional structure equations:
\begin{equation*}
\begin{aligned}
d \pi^1  & = i \, \kappa \wedge \ov{\pi^2} + \zeta \wedge \ov{\zeta} + \Lambda \wedge \rho, \\
d \pi^2  & = \pi^2 \wedge \ov{\pi^1} + \zeta \wedge \ov{\pi^2} + \Lambda \wedge \kappa. 
\end{aligned}
\end{equation*}
From the definition of $\pi^1$ and $\pi^2$, $\Lambda$ shall involve a term in $dt$. 
By adding $\Lambda$ to the set of $1$-forms $\rho$, $\kappa$, $\zeta$, $\ov{\kappa}$, $\ov{\zeta}$, $\pi^1$, $\pi^2$, $\ov{\pi^1}$, 
$\ov{\pi^2}$, we thus get a $10$-dimensional $\{ e\}$-structure on $G_{prol} \times P^9$, which constitutes the second (and last) $1$-dimensional prolongation to the
equivalence problem. It remains to compute the exterior derivative of $\Lambda$, which is done in what follows.

Taking the exterior derivative of the equation giving $d \pi^1$, we get:
\begin{equation*}
0 = i \, d \kappa \wedge \ov{\pi^2} - i \, \kappa \wedge \ov{ d \pi^2} + d \zeta \wedge \ov{\zeta} - \zeta \wedge d \ov{\zeta} 
+ d \Lambda \wedge \rho - \Lambda \wedge d \rho,
\end{equation*}
that is
\begin{dmath*}
0 = i \left( \pi^1 \wedge \kappa + \pi^2 \wedge \rho + \zeta \wedge  \ov{\kappa} \right) \wedge \ov{\pi^2}
- i \, \kappa \wedge \left( \ov{\pi^2} \wedge \pi^1 + \ov{\zeta}\wedge \pi^2 + \Lambda \wedge \ov{\kappa} \right)
+ \left( i \, \pi^2 \wedge \kappa + \pi^1 \wedge \zeta - \ov{\pi^1} \wedge \zeta \right) \wedge \ov{\zeta}
- \zeta \wedge \left( -i \, \ov{\pi^2} \wedge \ov{\kappa} + \ov{\pi^1} \wedge \ov{\zeta} - \pi^1 \wedge \ov{\zeta} \right) 
+ d \Lambda \wedge \rho - \Lambda \wedge \left( \pi^1 \wedge \rho + \ov{\pi^1} \wedge \rho + i \, \kappa \wedge \ov{\kappa} \right),
\end{dmath*}
which yields:
\begin{equation*}
0 = \left( d \Lambda - \Lambda \wedge \pi ^1 - \Lambda \wedge \ov{\pi^1} - i \, \pi^2 \wedge \ov{\pi^2} \right) \wedge \rho
.\end{equation*}

On the other hand, a similar computation starting from the expression of $d \pi^2$ gives:
\begin{dmath*}
0 = d \pi^2 \wedge \ov{\pi^1} - \pi^2 \wedge \ov{d \pi^1} + d \zeta \wedge \ov{\pi^2} - \zeta \wedge d \ov{\pi^2} + d \Lambda \wedge \kappa -
\Lambda \wedge d \kappa,
\end{dmath*}
that is
\begin{dmath*}
\left( \pi^2 \wedge \ov{\pi^1} + \zeta \wedge \ov{\pi^2} + \Lambda \wedge \kappa \right) \wedge \ov{\pi^1} - \pi^2 \wedge \left( - \ov{\kappa} \wedge \pi^2 +
\ov{\zeta}\wedge \zeta + \Lambda \wedge \rho \right) + \left( i \, \pi^2 \wedge \kappa + \pi^1 \wedge \zeta - \ov{\pi^1}\wedge \zeta \right) \wedge \ov{\pi^2}
- \zeta \wedge \left( \ov{\pi^2} \wedge \pi^1 + \ov{\zeta}\wedge \pi^2 + \Lambda \wedge \ov{\kappa} \right) + d \Lambda \wedge \kappa - 
\Lambda \wedge \left( \pi^1 \wedge \kappa + \pi^2 \wedge \rho + \zeta \wedge \ov{\kappa} \right),
\end{dmath*}
or
\begin{equation*}
\left( d \Lambda - i \, \pi^2 \wedge \ov{\pi^2} - \Lambda\wedge \pi^1 - \Lambda \wedge \ov{\pi^1} \right) \wedge \kappa = 0.  
\end{equation*}
From these last two equations, we deduce that:
\begin{equation*}
d \Lambda =  i \, \pi^2 \wedge \ov{\pi^2} + \Lambda \wedge \pi^1 + \Lambda \wedge \ov{\pi^1}.
\end{equation*}

Summing up the results that we have obtained so far,
the ten $1$-differential forms  $\rho$, $\kappa$, $\zeta$, $\ov{\kappa}$, $\ov{\zeta}$, $\pi^1$, $\pi^2$, $\ov{\pi^1}$, 
$\ov{\pi^2}$, $\Lambda$, satisfy the structure equations:

\begin{equation*}
\begin{aligned}
d \rho & = \pi^1 \wedge \rho + \ov{\pi^1} \wedge \rho + i \, \kappa \wedge \ov{\kappa}, \\
d \kappa & = \pi^1 \wedge \kappa + \pi^2 \wedge \rho + \zeta \wedge \ov{\kappa}, \\
d \zeta & = i \, \pi^2 \wedge \kappa + \pi^1 \wedge \zeta - \ov{\pi^1} \wedge \zeta,  \\
d \ov{\kappa} & = \ov{\pi^1} \wedge \ov{\kappa} + \ov{\pi^2} \wedge \rho - \kappa \wedge \ov{\zeta}, \\
d \ov{\zeta} & = - i \, \ov{\pi^2} \wedge \ov{\kappa} + \ov{\pi^1} \wedge \ov{\zeta} - {\pi^1} \wedge \ov{\zeta}, \\
d \pi^1  & = i \, \kappa \wedge \ov{\pi^2} + \zeta \wedge \ov{\zeta} + \Lambda \wedge \rho, \\
d \pi^2  & = \pi^2 \wedge \ov{\pi^1} + \zeta \wedge \ov{\pi^2} + \Lambda \wedge \kappa, \\ 
d \Lambda & =  i \, \pi^2 \wedge \ov{\pi^2} + \Lambda \wedge \pi^1 + \Lambda \wedge \ov{\pi^1}.
\end{aligned}
\end{equation*}

The torsion coefficients of these structure equations are all constant, and they do not depend on the graphing function $F$ of $M$.
This proves that all the hypersurfaces $M$ which satisfy
\begin{equation*}
J = W = 0
\end{equation*}
are locally biholomorphic.
A direct computation shows that the hypersurface defined by
\begin{equation*}
u = \frac{z_1 \ov{z_1} + \frac{1}{2} z_1^2 \ov{z_2} + \frac{1}{2} \ov{z_1^2} z_2}{1 - z_2 \ov{z_2}}
\end{equation*}
is precisely such that $J= W=0$. 
This completes the proof of theorem $\ref{thm:pocchiola}$.

\end{document}